\documentclass[a4paper,reqno,hidelinks]{amsart}

\usepackage{amsfonts}
\usepackage{mathrsfs}

\usepackage{graphicx}
\usepackage[normalem]{ulem}
\usepackage{url}

\usepackage{color}
\usepackage{mathtools}

\usepackage{amsthm,amsfonts,amssymb,amsmath,esint}
\usepackage[colorlinks=black, urlcolor=black]{hyperref}
\usepackage{amscd}
\usepackage{enumerate,enumitem,bbm}
\usepackage{appendix}
\usepackage{stackrel}
\usepackage{stackengine}
\setstackEOL{\\}
\setstackgap{L}{\normalbaselineskip}

\usepackage{fullpage}

\usepackage[all]{xy}
\usepackage{xspace}
\xspaceaddexceptions{]\}}

\newcommand{\RNum}[1]{\uppercase\expandafter{\romannumeral #1\relax}}

\DeclarePairedDelimiter\floor{\lfloor}{\rfloor}
\newcommand{\Sm}{\mathrm{Sm}}
\newcommand{\Kir}{\mathrm{Kir}}

\DeclareMathOperator{\Span}{span}
\newcommand{\CGLC}{\mathcal{GLC}}
\DeclareMathOperator{\GU}{GU}
\newcommand{\naive}{\mathrm{naive}}
\newcommand{\Nilp}{\mathrm{Nilp}}
\newcommand{\sG}{\mathscr{G}}
\newcommand{\sH}{\mathscr{H}}
\newcommand{\sL}{\mathscr{L}}

\newcommand{\matx}{{\begin{pmatrix}
	X\\ I_n
\end{pmatrix} }}

\newcommand{\smatx}{{\begin{psmallmatrix}
	X\\ I_n
\end{psmallmatrix} }}

\newcommand{\rfl}{\mathrm{fl}}

\newcommand{\sFl}{\mathscr{F}\!\ell}

\newcommand{\bigzero}{\mbox{\normalfont\Large\bfseries 0}}
\newcommand{\rvline}{\hspace*{-\arraycolsep}\vline\hspace*{-\arraycolsep}}

\newcommand{\BA}{{\mathbb {A}}}

\newcommand{\BG}{{\mathbb {G}}} 
\newcommand{\BI}{{\mathbb {I}}} 
 
\newcommand{\BM}{{\mathbb {M}}} 
 \newcommand{\BP}{{\mathbb {P}}}
\newcommand{\BQ}{{\mathbb {Q}}} \newcommand{\BR}{{\mathbb {R}}}

 \newcommand{\BZ}{{\mathbb {Z}}}

\newcommand{\CA}{{\mathcal {A}}} \newcommand{\CB}{{\mathcal {B}}}
\newcommand{\CC}{{\mathcal {C}}} 
 \newcommand{\CF}{{\mathcal {F}}}
\newcommand{\CG}{{\mathcal {G}}} 
\newcommand{\CI}{{\mathcal {I}}} 
 
 \newcommand{\CN}{{\mathcal {N}}}
\newcommand{\CO}{{\mathcal {O}}} 
 \newcommand{\CR}{{\mathcal {R}}}
\newcommand{\CS}{{\mathcal {S}}}

\newcommand{\RM}{{\operatorname {M}}} 

\newcommand{\RU}{{\operatorname {U}}}

\newcommand{\fm}{{\mathfrak{m}}}

\newcommand{\ad}{{\operatorname{ad}}}

\newcommand{\Char}{{\operatorname{char}}}

\newcommand{\dR}{{\operatorname{dR}}}
\newcommand{\disc}{{\operatorname{disc}}}

\newcommand{\Gal}{{\operatorname{Gal}}} \newcommand{\GL}{{\operatorname{GL}}}

\newcommand{\Hom}{{\operatorname{Hom}}}

\renewcommand{\Im}{{\operatorname{Im}}}

\newcommand{\loc}{{\operatorname{loc}}}

\newcommand{\op}{{\operatorname{op}}}
\newcommand{\ord}{{\operatorname{ord}}}

 \newcommand{\Sch}{{\operatorname{Sch}}}

\newcommand{\Sets}{{\operatorname{Sets}}}

\newcommand{\Sim}{{\operatorname{Sim}}}
\newcommand{\Spec}{\operatorname{Spec}}

\newcommand{\sgn}{{\operatorname{sgn}}}

\newcommand{\tr}{\operatorname{tr}}

\newcommand{\wt}[1]{{\widetilde {#1}}}
\newcommand{\wh}[1]{{\widehat {#1}}}

\newcommand{\pair}[1]{\langle {#1} \rangle}

\newcommand{\ol}[1]{{\overline{#1}}}

\newcommand{\lra}{\longrightarrow}

\newcommand{\ra}{\rightarrow}

\newcommand{\ud}[1]{{\underline{#1}}}

\newcommand{\rbra}[1]{\left({#1}\right)}

\newcommand{\cbra}[1]{\left\{{#1}\right\}}

\newcommand{\tcbra}[1]{\{{#1}\}} 
\newcommand{\trbra}[1]{({#1})} 

\makeatletter\def\varW@#1#2{%
\vtop{\m@th\ialign{##\cr
\hfil$#1 \operatorname{colim} $\hfil\cr
\noalign{\nointerlineskip\kern1.5\ex@}#2\cr
\noalign{\nointerlineskip\kern-\ex@}\cr}}
}
\makeatother
\makeatletter
\def\colim{%
\mathop{\mathpalette\varW@{}}\nmlimits@
}\makeatother

\theoremstyle{plain}
\newtheorem{thm}{Theorem}[section] 
\newtheorem{corollary}[thm]{Corollary}
\newtheorem{lem}[thm]{Lemma}  \newtheorem{prop}[thm]{Proposition}
 
\newtheorem{lemma}[thm]{Lemma}

\theoremstyle{remark} \newtheorem{remark}[thm]{Remark}
\theoremstyle{definition}  
\theoremstyle{definition} \newtheorem{example}[thm]{Example}  
\newtheorem{defn}[thm]{Definition}

\newcommand{\etale}{\'{e}tale\xspace}

\renewcommand{\theequation}{\arabic{equation}}
\numberwithin{equation}{section}

\newcommand{\inverse}{^{-1}}

\newcommand{\sset}{\subset}

\newcommand{\cross}{^\times}

\newcommand{\half}{ \frac{1}{2}}

\newcommand{\Tr}{\operatorname{Tr}}
\newcommand{\resp}{resp.\xspace}

\newcommand{\Spd}{\operatorname{Spd}}

\newcommand{\simto}{\overset{\sim}{\lra}}

\newcommand{\Gr}{\operatorname{Gr}}

\newcommand{\dfn}[1]{\emph{#1}}

\allowdisplaybreaks

\begin{document}
\title{Wildly ramified unitary local models for special parahorics. The odd dimensional case}
\author{Jie Yang}
\address{Yau Mathematical Sciences Center, Tsinghua University, Haidian District, Beijing 100084, China}
\email{jie-yang@mail.tsinghua.edu.cn}

\begin{abstract}
	We construct local models for wildly ramified unitary similitude groups of odd dimension $n\geq 3$ with special parahoric level structure and signature $(n-1,1)$. We first give a lattice-theoretic description for parahoric subgroups using Bruhat-Tits theory in residue characteristic two, and apply them to define local models following the lead of Rapoport-Zink \cite{rapoport1996period} and Pappas-Rapoport \cite{pappas2009local}. In our case, there are two conjugacy classes of special parahoric subgroups. We show that the local models are smooth for the one class and normal, Cohen-Macaulay for the other class. We also prove that they represent the v-sheaf local models of Scholze-Weinstein. Under some additional assumptions, we obtain an explicit moduli interpretation of the local models. 
\end{abstract}

\maketitle

\setcounter{tocdepth}{1}
\tableofcontents

\section{Introduction}
Let $p$ be a prime number. For the study of arithmetic properties of Shimura varieties, it is desirable to extend them to reasonable integral models over the $p$-adic integers. We have so far many results on the integral models of Shimura varieties of abelian type with parahoric level structure, see \cite{rapoport1996period, kisin2018integral, kisin2024independenceellfrobeniusconjugacy, pappas2021p}. Local models are certain flat projective schemes over the $p$-adic integers which are expected to model the singularities of these integral models. 

Rapoport and Zink initiated a systematic study of local models for Shimura varieties of PEL type  with parahoric level structure at $p$ in \cite{rapoport1996period}. Their local models were later called \dfn{naive local models}, since they are not always flat if the corresponding reductive group is ramified at $p$ as pointed out in \cite[\S 4]{pappas2000arithmetic}. The construction of the naive local models relies on the lattice-theoretic description of parahoric subgroups, which is significantly more involved if $p=2$ and the group is ramified. A more general approach is given in \cite{pappas2013local} (see also a variant in \cite{hpr20}) which constructs (flat) local models attached to purely group-theoretic data $(G,\sG,\cbra{\mu})$, where $G$ is a tamely ramified connected reductive group over a $p$-adic field $L$, $\sG$ is a parahoric group scheme over $\CO_L$ with generic fiber $G$, and $\cbra{\mu}$ is a geometric conjugacy class of cocharacters of $G$ with reflex field $E$. Subsequent works \cite{lev16, lourencco2023grassmanniennes, fakhruddin2022singularities} allow us to define local models for all triples $(G,\sG,\cbra{\mu})$ excluding the case that $p=2$ and $G^\ad$ contains, as an $\breve L$-factor, a wildly ramified unitary group of odd dimension. Here $\breve L$ denotes the completion of the maximal unramified extension of $L$. The constructions a priori depend on certain auxiliary choices. One can show that these local models are flat, projective, normal and Cohen-Macaulay schemes over $\CO_E$. Furthermore, the geometric special fibers are reduced and their irreducible components are normal, Cohen-Macaulay, with rational singularities, and compatibly Frobenius split. A key point in the proof is the identification of the special fibers of local models with a union of (semi-normalizations of) Schubert varieties in affine flag varieties. 

Another construction of local models is proposed in \cite{scholze2020berkeley} using v-sheaves. The advantage is that this approach is canonical (without any auxiliary choices) and applies to arbitrary triples $(G,\sG,\cbra{\mu})$, even for wildly ramified groups $G$ and $p=2$. It has been proven in \cite{anschutz2022p, gleason2022tubular} that when $\cbra{\mu}$ is minuscule, the v-sheaf local models are representable by flat normal projective schemes over $\CO_E$ with reduced special fibers. Excluding the case that $p=2$ and $G^\ad$ contains, as an $\breve L$-factor, a wildly ramified unitary group of odd dimension, one can show that the corresponding scheme local models are Cohen-Macaulay with Frobenius split special fibers. We refer the readers to \cite[Remark 2.2]{fakhruddin2022singularities} for some explanation on this exceptional case. 

In this paper, we focus on this exceptional case and study local models for unitary similitude groups of odd dimension $n\geq 3$ with parahoric level structure when $p=2$. We aim to construct unitary local models following the lead of Rapoport-Zink \cite{rapoport1996period} and Pappas-Rapoport \cite{pappas2009local}, specifically for cases where the parahoric level structure is special and the signature is $(n-1,1)$. Furthermore, we will prove that these local models have good geometric properties. We emphasize that many of these geometric properties do not seem to easily follow from the works mentioned above.

Let $F_0/\BQ_2$ be a finite extension and $F$ be a (wildly) ramified quadratic extension of $F_0$. For any $x\in F$, we write $\ol{x}$ for the Galois conjugate of $x$ in $F$. We can pick uniformizers $\pi\in F$ and $\pi_0\in F_0$ such that $F/F_0$ falls into one of the following two distinct cases (see \S \ref{subsec-prel}):   

\begin{enumerate}
	\item [(R-U)] $F=F_0(\sqrt{\theta})$, where $\theta$ is a unit in $\CO_{F_0}$. The uniformizer $\pi$ satisfies an Eisenstein equation  $$\pi^2-t\pi+\pi_0=0,$$ where $t=\pi+\ol{\pi}\in\CO_{F_0}$ satisfies $\pi_0|t$ and $t|2$. We have $\sqrt{\theta}=1-{2\pi}/{t}$ and $\theta=1-{4\pi_0}/{t^2}$. 
	\item [(R-P)] $F=F_0(\sqrt{\pi_0})$, where $\pi^2+\pi_0=0$. 
\end{enumerate}

Let $(V,h)$ be a hermitian space, where $V$ is an $F$-vector space of dimension $n=2m+1 \geq 3$ and $h: V\times V\ra F$ is a non-degenerate hermitian form. We assume that $h$ is split, i.e., there exists an $F$-basis $(e_i)_{1\leq i\leq n}$ of $V$ such that $h(e_i,e_j)=\delta_{i,n+1-j}$ for $1\leq i,j\leq n$. This condition can be satisfied after an unramified field extension of $F_0$. Let $G\coloneqq \GU(V,h)$ denote the unitary similitude group over $F_0$ attached to $(V,h)$. Our first result is the lattice-theoretic description of parahoric subgroups of $G(F_0)$. 

\begin{thm}[Proposition \ref{prop-para}] \label{thm1.1}
	Let $I$ be a non-empty subset of $\cbra{0,1,\ldots,m}$. Define \begin{flalign*}
	   \Lambda_i &\coloneqq \CO_F\pair{\pi\inverse e_1,\ldots,\pi\inverse e_i,e_{i+1},\ldots,e_{m+1},\lambda e_{m+2},\ldots,\lambda e_n }, \text{\ for\ $0\leq i\leq m$}, 
\end{flalign*} where $\lambda={\ol{\pi}}/{t}$ in the (R-U) case and $\lambda={1}/{2}$ in the (R-P) case. Then the subgroup \begin{flalign*}
	    P_I &\coloneqq \cbra{g\in G(F_0)\ |\ g\Lambda_i=\Lambda_i, \text{\ for\ }i\in I }
\end{flalign*} is a parahoric subgroup of ${G}(F_0)$. Furthermore, any parahoric subgroup of ${G}(F_0)$ is conjugate to $P_I$ for a unique $I\subset\cbra{0,1,\ldots,m}$. The conjugacy classes of special parahoric subgroups correspond to the sets $I=\cbra{0}$ and $\cbra{m}$. 
\end{thm}

The proof of Theorem \ref{thm1.1} is based on Bruhat-Tits theory in (residue) characteristic two. Note that in our case, parahoric subgroups of $G(F_0)$ no longer correspond to self-dual lattice chains, which causes difficulties in the study of local models. 

Given a special parahoric subgroup of $G(F_0)$ corresponding to $I=\cbra{0}$ or $\cbra{m}$, we define in \S \ref{sec-main} the \dfn{naive local model} $\mathrm{M}^\naive_I$ of signature $(n-1,1)$, which is an analogue of the naive unitary local model considered in \cite{rapoport1996period}. To explain the construction, we start with a crucial but simple observation on the structure of the lattices $\Lambda_i$ in Theorem \ref{thm1.1}. Set \begin{equation}
       \begin{gathered}
       	   \varepsilon\coloneqq \begin{cases}
	    	 t\quad &\text{in the (R-U) case},\\ 2 &\text{in the (R-P) case}.
	    \end{cases}
       \end{gathered} \label{epsilon}
\end{equation} 
The hermitian form $h$ defines a symmetric $F_0$-bilinear form $s(-,-): V\times V\ra F_0$ and a quadratic form $q: V\ra F_0$ via \begin{flalign}
	    s(x,y) \coloneqq {\varepsilon}\inverse \Tr_{F/F_0} h(x,y) \text{\ and\ }q(x) \coloneqq  \half s(x,x), \text{\ for $x,y\in V$}.  \label{symform}
\end{flalign}  
Set $\sL\coloneqq {\varepsilon}\inverse \CO_{F_0}$, which is an invertible $\CO_{F_0}$-module. Then for $0\leq i\leq m$, the forms in \eqref{symform} induce the $\sL$-valued forms \begin{flalign}
	     s : \Lambda_i\times \Lambda_i\lra \sL \text{\ and\ } q : \Lambda_i\lra \sL. \label{sqLambda} 
\end{flalign} The triple $(\Lambda_i,q,\sL)$ is an $\sL$-valued hermitian quadratic module over $\CO_{F_0}$ in the sense of Definition \ref{defn-hermqm}, which roughly means that the quadratic form $q$ is compatible with the $\CO_F$-action. 

For $I=\cbra{0}$ or $\cbra{m}$, denote $\Lambda_I\coloneqq \Lambda_0$ or $\Lambda_m$ respectively. Let $\Lambda_I^s\coloneqq \cbra{x\in V\ |\ s(x,\Lambda_I)\sset \CO_{F_0}}$ be the dual lattice of $\Lambda_I$ with respect to the pairing $s$ in \eqref{symform}. Then we have a perfect $\CO_{F_0}$-bilinear pairing \begin{flalign}
	   \Lambda_I \times \Lambda_I^s \lra \CO_{F_0}  \label{perfpairing}
\end{flalign}  induced by the symmetric pairing in \eqref{symform}, and an inclusion of lattices \begin{flalign*}
	     \Lambda_I\hookrightarrow \alpha\Lambda_I^s,  \ \text{\ where\ } \alpha \coloneqq \begin{cases}
	     	{\ol{\pi}}/{\varepsilon}\quad &\text{if $I=\cbra{0}$},\\  {1}/{\varepsilon} &\text{if $I=\cbra{m}$}.
	     \end{cases}
\end{flalign*}
We define the naive unitary local model $\RM_I^\naive$ to be the functor $$\RM_I^\naive: (\Sch/\CO_F)^\op \lra \Sets$$ which sends an $\CO_F$-scheme $S$ to the set of $\CO_S$-modules $\CF$ such that 
	\begin{enumerate}
		\item ($\pi$-stability condition) $\CF$ is an $\CO_F\otimes_{\CO_{F_0}}\CO_S$-submodule of $\Lambda_I\otimes_{\CO_{F_0}}\CO_S$ and as an $\CO_S$-module, it is a locally direct summand of rank $n$.
		\item (Kottwitz condition) The action of $\pi\otimes 1\in\CO_F\otimes_{\CO_{F_0}}\CO_S$ on $\CF$ has characteristic polynomial $$\det(T-\pi\otimes 1\ |\ \CF)=(T-\pi)(T-\ol{\pi})^{n-1}.$$
		\item Let $\CF^\perp$ be the orthogonal complement of $\CF$ in $\Lambda_I^s\otimes_{\CO_{F_0}}\CO_S$ with respect to the perfect pairing $$(\Lambda_I\otimes_{\CO_{F_0}}\CO_S) \times (\Lambda_I^s\otimes_{\CO_{F_0}}\CO_S) \ra \CO_S$$ induced by the perfect pairing in \eqref{perfpairing}. We require the map $\Lambda_I\otimes_{\CO_{F_0}}\CO_S \ra \alpha \Lambda_I^s\otimes_{\CO_{F_0}}\CO_S$ induced by $\Lambda_I\hookrightarrow \alpha\Lambda_I^s$ sends $\CF$ to $\alpha\CF^\perp,$ where $\alpha\CF^\perp$ denotes the image of $\CF^\perp$ under the isomorphism $\alpha:\Lambda_I^s\otimes_{\CO_{F_0}}\CO_S \simto \alpha\Lambda_I^s\otimes_{\CO_{F_0}}\CO_S$.
        \item $\CF$ is totally isotropic with respect to the pairing \begin{flalign*}
        	   s: (\Lambda_I\otimes_{\CO_{F_0}}\CO_S) \times (\Lambda_I\otimes_{\CO_{F_0}}\CO_S) \ra \sL\otimes_{\CO_{F_0}}\CO_S
        \end{flalign*} induced by $s$ in \eqref{sqLambda}, i.e., $s(\CF,\CF)=0$ in $\sL\otimes_{\CO_{F_0}}\CO_S$. 
	\end{enumerate}
The functor $\RM^\naive_I$ is representable by a closed $\CO_F$-subscheme of the Grassmannian $\Gr(n,\Lambda_I)_{\CO_F}$. 
It turns out that $\RM^\naive_I$ is not flat over $\CO_F$. We define, as in \cite{pappas2009local}, the \dfn{local model} $\RM_I^\loc$ to be the flat closure of the generic fiber in $\RM^\naive_I$. By construction, we have closed immersions \begin{flalign*}
	     \RM^\loc_I\hookrightarrow \RM^\naive_I
\end{flalign*} of projective schemes over $\CO_F$ whose generic fibers are isomorphic to the $(n-1)$-dimensional projective space over $F$. We have the following results on further geometric properties of $\RM^\loc_I$. 

\begin{thm} \label{thm-intromain}
    \begin{enumerate}
		\item If $I=\cbra{0}$, then $\RM^\loc_{\cbra{0}}$ is flat projective of relative dimension $n-1$ over $\CO_F$, normal and Cohen-Macaulay with geometrically integral special fiber. Moreover, $\RM^\loc_{\cbra{0}}$ is smooth over $\CO_F$ on the complement of a single closed point in the special fiber. 
		\item If $I=\cbra{m}$, then $\RM_{\cbra{m}}^\loc$ is smooth projective of relative dimension $n-1$ over $\CO_F$ with geometrically integral special fiber.
	\end{enumerate}
\end{thm}

Let us explain the strategy of the proof of Theorem \ref{thm-intromain} in greater detail. For $I=\cbra{0}$ or $\cbra{m}$, let $\sH_I$ denote the group scheme\footnote{In Appendix \ref{appB}, we prove that $\sH_I$ is smooth over $\CO_{F_0}$ and isomorphic to the parahoric group scheme attached to $\Lambda_I$. } of similitude automorphisms of the hermitian quadratic module $(\Lambda_m,q,\sL)$ (resp. $(\Lambda_0,q,\sL,\phi)$), see Definition \ref{defn-simisom} and \ref{defn-hermphi}. Then $\sH_I$ acts naturally on $\RM^\naive_I$, and hence on $\RM^\loc_I$. Let $\ol{k}$ denote the algebraic closure of the residue field of $F$. Using the results in Appendix \ref{appB}, we can show that the (geometric) special fiber $\RM_I^\loc\otimes_{\CO_F}\ol{k}$ has two orbits under the action of $\sH_I\otimes_{\CO_{F_0}}\ol{k}$. One of the orbits consists of just one closed point. We call it the \dfn{worst point} of the local model. Using this, we are reduced to proving that there exists an open affine subscheme of $\RM^\loc_I$ containing the worst point and satisfying the geometric properties (normality, Cohen-Macaulayness, etc) as stated in Theorem \ref{thm-intromain}. When $I=\tcbra{0}$, the worst point turns out to be the unique singular point of $\RM_I^\loc$.

To get the desired open affine subscheme of $\RM^\loc_I$, we introduce a refinement $\RM_I$, as a closed subfunctor, of the moduli functor $\RM^\naive_I$ such that \begin{flalign*}
	     \RM^\loc_I\sset \RM_I\sset \RM^\naive_I.
\end{flalign*} It turns out that the underlying topological space of $\RM_I$ is equal to that of $\RM^\loc_I$. For a matrix $A$, we will write $\CO_F[A]$ for the polynomial ring over $\CO_F$ whose variables are entries of the matrix $A$. Viewing $\RM_I$ as a closed subscheme of the Grassmannian $\Gr(n,\Lambda_I)_{\CO_F}$, we can find an open affine subscheme $\RU_I$ of $\RM_I$ which contains the worst point and which is isomorphic to a closed subscheme of $\Spec \CO_F[Z]$, where $Z$ is an $n\times n$ matrix, such that the worst point is defined by $Z=0$ and $\pi=0$. Then we explicitly write down the affine coordinate ring of $\RU_I$ defined by matrix identities. From this, we obtain the affine coordinate ring of $\RU_I\cap \RM^\loc_I$ by calculating the flat closure of $\RU_I$.

\begin{thm} \label{thm16}
	Let $Y$ (\resp $X$) be a $2m\times 2m$ (resp. $2m\times 1$) matrix with variables as entries. Let $H_{2m}$ denote the $2m\times 2m$ anti-diagonal unit matrix. There is an open affine subscheme $\RU_I^\loc$ of $\RM^\loc_I$ which contains the worst point and satisfies the following properties. 
	\begin{enumerate}
		\item If $I=\cbra{0}$, then $\RU^\loc_{\cbra{0}}$ is isomorphic to \begin{flalign*}			       
			         	&\Spec \frac{\CO_F[Y| X]}{\rbra{\wedge^2(Y|X), Y-Y^t, (\frac{\pi}{\ol{\pi}} \frac{\tr({H}_{2m}Y)}{2}+\pi\sqrt{\theta})Y+XX^t   }}, \quad \text{in the (R-U) case},\\ 
	                    &\Spec \frac{\CO_F[Y|X]}{\rbra{\wedge^2(Y|X), Y-Y^t, (\frac{\tr(H_{2m}Y)}{2}-\pi)Y+XX^t   }},  \quad \text{in the (R-P) case}.			               
		       \end{flalign*}
		      (We remark that under the relation $Y-Y^t=0$, the polynomial $\tr(H_{2m}Y)$, which is the sum of the anti-diagonal entries of $Y$, is indeed divisible by $2$ in $\CO_F[Y]$.)  
	  \item If $I=\cbra{m}$, then $\RU^\loc_{\cbra{m}}$ is isomorphic to \begin{flalign*}			       
			         	&\Spec \frac{\CO_F[Y|X ]}{\rbra{\wedge^2(Y|X), Y-Y^t, (\frac{\tr({H}_{2m}Y)}{t}+\sqrt{\theta})Y+XX^t   }}, \quad \text{in the (R-U) case},\\ 
	                    &\Spec \CO_F[X],  \quad \text{in the (R-P) case}.			               
		       \end{flalign*}
	\end{enumerate} 
\end{thm}

Using the above result, we reduce the proof of Theorem \ref{thm-intromain} to a purely commutative algebra problem. We need to show that the affine coordinate rings in Theorem \ref{thm16} satisfy the geometric properties as stated in Theorem \ref{thm-intromain}. The hardest part is to show the Cohen-Macaulayness when $I=\cbra{0}$, where we use a converse version of the miracle flatness theorem. We refer to Lemma \ref{lem-spfiberCM} for more details.

We can also relate $\RM_I^\loc$ to the v-sheaf local models considered in \cite[\S 21.4]{scholze2020berkeley}. By results in \cite{anschutz2022p, fakhruddin2022singularities, gleason2022tubular}, we already know that the v-sheaf local models in our case are representable by normal projective flat $\CO_F$-schemes $\BM_I$ (denoted as $\BM_{\sG,\cbra{\mu}}$ in \S \ref{sec-compa}). 

\begin{thm}[Theorem \ref{prop-M=Mv}] \label{thm1.3}
	The local model $\RM^\loc_I$ is isomorphic to $\BM_I$. 
\end{thm}
As a corollary, our result gives a very explicit construction of $\BM_I$ and a more elementary proof of the representability of the v-sheaf local models in our setting. 
\begin{remark}
   Recently, Cass and Louren{\c{c}}o \cite[Remark 4.17, Corollary 1.5]{cass2025mod} proved that $\BM_I$ is Cohen-Macaulay and has Frobenius-split special fiber, using a more general and uniform approach. 
    Our method, on the other hand, provides explicit local affine coordinate rings and the existence of a closed immersion $\BM_\BI\hookrightarrow\Gr(n,2n)_{\CO_F}$. These results have applications to the study of integral models of Shimura varieties, see \cite{yang20232}.  
\end{remark}

It should be pointed out that it could be useful to provide an explicit moduli interpretation of $\RM^\loc_I$. Let us first briefly review some previous results. We start with reviewing results for odd primes $p$. Let $F/F_0$ be a ramified quadratic extension of $p$-adic local fields. Then we can find a uniformizer $\pi$ (\resp $\pi_0$) of $F$ (\resp $F_0$) such that $\pi^2=\pi_0$, which is possible since $p>2$. Let $(V,h)$ be a split non-degenerate $F/F_0$-hermitian space of \dfn{any} dimension $n\geq 3$. We fix an $F$-basis $(e_i)_{1\leq i\leq n}$ of $V$ such that $h(e_i, e_j)=\delta_{i,n+1-j}$ for $1\leq i,j\leq n$. Set \begin{flalign*}
	    \Lambda_i\coloneqq \CO_F\pair{\pi\inverse e_1,\ldots, \pi\inverse e_i, e_{i+1},\ldots, e_n}, \text{\ for $0\leq i\leq n-1$}. 
\end{flalign*}   
The hermitian form on $V$ defines a perfect alternating $F_0$-bilinear form $\pair{-,-}$ on $V$: \begin{flalign*}
	   \pair{x,y}\coloneqq \half\Tr_{F/F_0}(\pi\inverse h(x,y)), \text{\ for $x,y\in V$}.
\end{flalign*}
The form satisfies the property that $\pair{x,\pi y}=-\pair{\pi x,y}$ for $x,y\in V$. For lattices $\Lambda_i$ as above, we have the relation $\wh{\Lambda}_i=\pi\Lambda_{n-i}$, where $\wh{\Lambda}_i$ denotes the dual lattice of $\Lambda_i$ with respect to the pairing $\pair{-,-}$.
Let $G\coloneqq \GU(V,h)$ denote the unitary similitude group over $F_0$ attached to $(V,h)$.  Set $m\coloneqq \floor{n/2}$. For a nonempty subset $I\sset \cbra{0,\ldots,m}$ with the property (P) that if $n$ is even and $m-1\in I$ then also $m\in I$, we can form a self-dual lattice chain $\Lambda_I\coloneqq \cbra{\Lambda_j}_{j\in  n\BZ\pm I}$ by setting $\Lambda_{kn-i}\coloneqq \pi^{-k}\wh{\Lambda}_i$ and $\Lambda_{kn+i}\coloneqq \pi^{-k}\Lambda_i$ for $i\in I$ and $k\in\BZ$. We have $\wh{\Lambda}_j=\Lambda_{-j}$ for all $j\in  n\BZ\pm I$. Then the stabilizer in $G(F_0)$ of the lattice chain $\Lambda_I$ contains a parahoric subgroup with index at most $2$. In fact, this provides a one-to-one correspondence between non-empty subsets $I$ of $\cbra{0,\ldots,m}$ with the property (P) and conjugacy classes of parahoric subgroups in $G(F_0)$. Building on works of \cite{rapoport1996period,pappas2000arithmetic,pappas2009local}, Smithling formulated in \cite{smithling2015moduli} the ``strengthened spin condition" and used it to define an explicit moduli functor $\RM_I^\Sm(r,s)$ refining the naive unitary local model of \dfn{any} signature $(r,s)$ with \dfn{any} parahoric level structure corresponding to the subset $I$ . We refer to \cite{smithling2015moduli} for the detailed definition of this moduli functor.  
Smithling conjectured that $\RM^\Sm_I(r,s)$ represents the unitary local model of signature $(r,s)$ with parahoric level structure corresponding to $I$, see  \cite[Conjecture 1.3]{smithling2015moduli}. When the signature is $(n-1,1)$, the conjecture was recently proved by \cite{luo} generalizing the results in \cite{smithling2015moduli,yu2019moduli}.  

We now discuss $p=2$. Let $F/F_0$ be a ramified quadratic extension of $2$-adic local fields. Let $G$ denote the unitary similitude group attached to a split non-degenerate $F/F_0$-hermitian space $(V,h)$ of \dfn{even} dimension $n=2m$. In his unpublished manuscript \cite{kirchunitary}, Kirch obtained a description of parahoirc subgroups of $G(F_0)$ in terms of self-dual lattice chains of hyperbolic lattices. We quickly recall the description here. The hermitian form $h$ induces a perfect symmetric $F_0$-bilinear pairing and a quadratic form on $V$: $$s(x,y)\coloneqq \varepsilon\inverse\Tr_{F/F_0}h(x,y),\ q(x)\coloneqq \half s(x,x),\ \text{for $x,y\in V$}, $$ where $\varepsilon$ is defined as in \eqref{epsilon}. We say a lattice $\Lambda$ in $V$ is a \dfn{hyperbolic} lattice if $\Lambda$ is an orthogonal sum (with respect to the symmetric pairing $s$) of hyperbolic planes in the sense of \cite[\S 2]{kirch2017construction}. See also discussion in Remark \ref{rmk-hyp}. For $0\leq i\leq m$, Kirch found standard hyperbolic lattices $\Lambda_i$ with explicit generators. These lattices $\Lambda_i$ satisfy the property that for $x\in\Lambda_i$, we have $q(x)\in \CO_{F_0}$. For a non-empty subset $I$ of $\cbra{0,\ldots,m}$ with the property (P) that if $m-1\in I$ then $m\in I$, one can obtain a self-dual lattice chain $\Lambda_I\coloneqq (\Lambda_j)_{j\in n\BZ\pm I}$ of hyperbolic lattices by setting $\Lambda_{kn-i}\coloneqq \pi^{-k}{\Lambda}^s_i$ and $\Lambda_{kn+i}\coloneqq \pi^{-k}\Lambda_i$ for $i\in I$ and $k\in\BZ$. Here $\Lambda_i^s$ denotes the dual lattice of $\Lambda_i$ with respect to the symmetric pairing $s$. Then Kirch showed that the stabilizer of $\Lambda_I$ contains a parahoric subgroup of $G(F_0)$ with index at most $2$, and this gives a one-to-one correspondence between subsets $I$ with the property (P) and conjugacy classes of parahoric subgroups of $G(F_0)$. His proof follows the analysis of \cite[\S 4]{pappas2008twisted}, which is different from our method in the proof of Theorem \ref{thm1.1}. For a subset $I\sset\cbra{0,\ldots,m}$ containing $m$, Kirch defined a moduli functor $$\RM_I^\Kir(n-1,1): (\Sch/\CO_F)^\op\ra \Sets$$ which sends an $\CO_F$-scheme $S$ to the set of $\CO_S$-modules $(\CF_j)_{j\in n\BZ\pm I}$ such that 
    \begin{enumerate}
	  	\item ($\pi$-stability condition) For all $j\in  n\BZ\pm I$, $\CF_j$ is an $\CO_F\otimes_{\CO_{F_0}}\CO_S$-submodule of $\Lambda_j\otimes_{\CO_{F_0}}\CO_S$ and as an $\CO_S$-submodule, it is a locally direct summand of rank $n$.
	  	\item (Kottwitz condition) For all $j\in  n\BZ\pm I$, the action of $\pi\otimes 1\in \CO_F\otimes_{\CO_{F_0}}\CO_S$ on $\CF_j$ has characteristic polynomial $$\det(T-\pi\otimes 1\ |\ \CF_j)= (T-\pi)(T-\ol{\pi})^{n-1}.$$
	  	\item For all $j<j'$ in $ n\BZ\pm I$, the map $\Lambda_j\otimes_{\CO_{F_0}}\CO_S\ra \Lambda_{j'}\otimes_{\CO_{F_0}}\CO_S$ induced by the inclusion $\Lambda_j\hookrightarrow \Lambda_{j'}$ sends $\CF_j$ to $\CF_{j'}$. Furthermore, for each $j$, the isomorphism $\pi: \Lambda_j\ra \Lambda_{j-n}$ induces an isomorphism between $\CF_j$ and $\CF_{j-n}$. 
	  	\item For all $j\in n\BZ\pm I$, we have $\CF_{n-j}=\CF_j^\perp$, where $\CF^\perp_j$ denotes the orthogonal complement of $\CF_j$ under the perfect pairing \begin{flalign*}
	  		     (\Lambda_j\otimes_{\CO_{F_0}}\CO_S)\times (\Lambda_{n-j}\otimes_{\CO_{F_0}}\CO_S)\ra \CO_S
	  	      \end{flalign*} induced by the perfect symmetric paring $s(-,-)$ on $V$.
	  	\item (Hyperbolicity condition) For all $j\in n\BZ\pm I$, the quadratic form $q: \Lambda_j\otimes_{\CO_{F_0}}\CO_S\ra \CO_S$ induced by $q: \Lambda_j\ra \CO_S$ satisfies $q(\CF_j)=0$. 
	  	\item (Wedge condition) For all $j\in  n\BZ\pm I$, the action of $\pi\otimes 1-1\otimes\ol{\pi}\in \CO_F\otimes_{\CO_{F_0}}\CO_S$ on $\CF_j$ satisfies $$\wedge^2(\pi\otimes 1-1\otimes\ol{\pi}\ |\ \CF_j)=0.$$
	  	\item (Spin condition) The action of $\pi\otimes 1-1\otimes \ol{\pi} \in \CO_F\otimes_{\CO_{F_0}}\CO_S$ on $\CF_m$ is nowhere zero.
	  	      (See more discussion of this condition in \cite[Remark 9.9]{rapoport2018regular}.)
     \end{enumerate}
In \cite{kirchunitary}, Kirch claimed (without proof) that $\RM^\Kir_I(n-1,1)$ is representable by a Cohen-Macaulay and flat projective $\CO_F$-scheme. When $n=2$, the scheme $\RM^\Kir_I(n-1,1)$ descends to a scheme over $\CO_{F_0}$. One should observe that our construction of $\RM_I$ in the present paper is inspired by Smithling and Kirch.  

However, for odd unitary unitary groups and $p=2$, it seems hard to give a moduli interpretation. As a by-product of our analysis of $\RU^\loc_I$, we can obtain such a description in a special case.

\begin{thm} \label{intro-thm1.6}
	Suppose $F/F_0$ is of type (R-U) and assume that the valuations of $t$ and $\pi_0$ are equal\footnote{This holds if $F_0$ is unramified over $\BQ_2$, see some more discussion in Remark \ref{rmk-t=pi0}.   }. Then $\RM_{\cbra{0}}^\loc$ represents the functor $$(\Sch/\CO_{F})^\op\lra \Sets$$ which sends an $\CO_F$-scheme $S$ to the set of $\CO_S$-modules $\CF$ such that \footnote{As in \cite[Lemma 5.2, Remark 5.4]{smithling2015moduli}, the conditions \textbf{LM2} and \textbf{LM5} are in fact implied by \textbf{LM6}. }
	\begin{enumerate}[label=\textbf{LM\arabic*}]
		\item ($\pi$-stability condition) $\CF$ is an $\CO_F\otimes_{\CO_{F_0}}\CO_S$-submodule of $\Lambda_0\otimes_{\CO_{F_0}}\CO_S$ and as an $\CO_S$-module, it is a locally direct summand of rank $n$.
		\item (Kottwitz condition) The action of $\pi\otimes 1\in\CO_F\otimes_{\CO_{F_0}}\CO_S$ on $\CF$ has characteristic polynomial $$\det(T-\pi\otimes 1\ |\ \CF)=(T-\pi)(T-\ol{\pi})^{n-1}.$$
		\item Let $\CF^\perp$ be the orthogonal complement in $\Lambda_0^s\otimes_{\CO_{F_0}}\CO_S$ of $\CF$ with respect to the perfect pairing $$  (\Lambda_0\otimes_{\CO_{F_0}}\CO_S) \times (\Lambda_0^s\otimes_{\CO_{F_0}}\CO_S) \ra \CO_S$$ induced by the perfect pairing in \eqref{perfpairing}. We require that the map $\Lambda_0\otimes_{\CO_{F_0}}\CO_S \ra \frac{\ol{\pi}}{t} \Lambda_0^s\otimes_{\CO_{F_0}}\CO_S$ induced by $\Lambda_0\hookrightarrow \frac{\ol{\pi}}{t}\Lambda_I^s$ sends $\CF$ to $\frac{\ol{\pi}}{t}\CF^\perp,$ where $\frac{\ol{\pi}}{t} \CF^\perp$ denotes the image of $\CF^\perp$ under the isomorphism $\frac{\ol{\pi}}{t}:\Lambda_0^s\otimes_{\CO_{F_0}}\CO_S \simto \frac{\ol{\pi}}{t}\Lambda_0^s\otimes_{\CO_{F_0}}\CO_S$.		
		\item (Hyperbolicity condition) The quadratic form $q: \Lambda_0\otimes_{\CO_{F_0}}\CO_S\ra \sL\otimes_{\CO_{F_0}}\CO_S$ induced by $q:\Lambda_0\ra \sL$ satisfies $q(\CF)=0$. 
		\item (Wedge condition) The action of $\pi\otimes 1-1\otimes\ol{\pi}\in \CO_F\otimes_{\CO_{F_0}}\CO_S$ on $\CF$ satisfies $$\wedge^2(\pi\otimes 1-1\otimes \ol{\pi}\ |\ \CF )=0.$$
		\item (Strengthened spin condition) The line $\wedge^n\CF\sset W(\Lambda_0)\otimes_{\CO_F}\CO_S$ is contained in $$\Im\rbra{W(\Lambda_0)^{n-1,1}_{- 1}\otimes_{\CO_F}\CO_S\ra W(\Lambda_0)\otimes_{\CO_F}\CO_S }.$$	 
		     (See \S \ref{subsubsec-spin} for the explanation of the notation in this condition.)
	\end{enumerate}
\end{thm}

We now give an overview of the paper. In \S \ref{sec-BTtheory}, we discuss Bruhat-Tits theory for (odd) unitary groups in residue characteristic two. In particular, we describe the maxi-minorant norms (\dfn{norme maximinorante} in French) used in \cite{bruhat1987schemas} in terms of graded lattice chains, and thus obtain a lattice-theoretic description of the Bruhat-Tits buildings of unitary groups. As a corollary, we deduce Theorem \ref{thm1.1}. In \S \ref{sec-main}, we first discuss some basic facts about quadratic extensions of $2$-adic fields. Then we equip the lattices $\Lambda_i$ in Theorem \ref{thm1.1} with the structure of hermitian quadratic modules. Using this, we define the naive local models $\RM_I^\naive$ and local models $\RM_I^\loc$. In \S 4-7, we prove Theorem \ref{thm-intromain}, \ref{thm16} and \ref{intro-thm1.6}. We address the (R-U) and (R-P) case separately, although the techniques are very similar. In each section, we introduce the refinement $\RM_I$ of $\RM_I^\naive$ by imposing certain linear algebraic conditions and then explicitly write down the local affine coordinate rings. We then obtain Theorem \ref{thm16} by computing the flat closure of these affine coordinate rings. Utilizing the group action on local models, we finish the proof of Theorem \ref{thm-intromain} and Theorem \ref{intro-thm1.6}. In \S \ref{sec-compa}, we review the Beilinson-Drinfeld Grassmannian (in mixed characteristic) and v-sheaf local models of Scholze-Weinstein. Then we show that the local models in Theorem \ref{thm-intromain} represent the v-sheaf local models, thereby proving Theorem \ref{thm1.3}. In Appendix \ref{appB}, we show that, under certain conditions, hermitian quadratic modules \etale locally have a normal form up to similitude. Along the way, we prove in Theorem \ref{thmsimm} and Theorem \ref{thmsimm0} that the similitude automorphism group scheme of $\Lambda_m$ (resp. $(\Lambda_0,\phi)$) is affine smooth over $\CO_{F_0}$ and is isomorphic to the parahoric group scheme attached to $\Lambda_m$ (resp. $\Lambda_0$).

\subsection*{Acknowledgments}
This work is part of my PhD thesis at Michigan State University. I thank my advisor G. Pappas for patient and helpful discussions and for reading drafts of the paper. I am grateful to M. Rapoport for conversations at MSRI in the beginning of this project. I also want to thank Y. Luo for sharing the preliminary version of \cite{luo}, and T. Haines for discussions regarding the works \cite{anschutz2022p} and \cite{fakhruddin2022singularities}. Finally, I would like to thank the anonymous referees for their careful reading and valuable suggestions and comments. This project is partially supported by the Graduate Research Associate Fellowship at MSU and NSF Grant DMS-2100743.

\section{Bruhat-Tits theory for unitary groups in residue characteristic two} 
\label{sec-BTtheory}

\subsection{Notations}\label{subsec-notation}
Let $F_0$ be a finite extension of $\BQ_2$. Let $\omega: F_0\ra \BZ\cup\{+\infty\}$ denote the normalized valuation on $F_0$. Let $F/F_0$ be a (wildly totally) ramified quadratic extension. The valuation $\omega$ uniquely extends to a valuation on $F$, which is still denoted by $\omega$. Denote by $\sigma$ the nontrivial element in $\Gal(F/F_0)$. For $x\in F$, we will write $x^\sigma$ or $\ol{x}$ for the Galois conjugate of $x$ in $F$. Let $\CO_F$ (resp. $\CO_{F_0}$) be the ring of integers of $F$ (resp. $F_0$) with uniformizer $\pi$ (resp. $\pi_0$). We assume $N_{F/F_0}(\pi)=\pi_0$. Let $k$ be the common residue field of $F$ and $F_0$. Let $V$ be an $F$-vector space of dimension $n=2m+1\geq 3$ with a non-degenerate hermitian form $h: V\times V\ra F$. We assume that there exists an $F$-basis $(e_i)_{1\leq i\leq n}$ of $V$ such that $h(e_i,e_j)=\delta_{i,n+1-j}$ for $1\leq i,j\leq n$. In this case, we will say the hermitian form $h$ is \dfn{split}, or $(V,h)$ is a split hermitian space.

(We remark that all results in \S \ref{sec-BTtheory} are valid when $F_0$ is a finite extension of $\BQ_p$ for $p>2$, see Remark \ref{rmk-25} and \ref{rmk-29}.)

\subsection{Bruhat-Tits buildings in terms of norms}\label{subsec-buildingnorm}
In this subsection, we would like to recall the description of Bruhat-Tits buildings of odd dimensional (quasi-split) unitary groups in residue characteristic two in terms of norms. The standard reference is \cite{bruhat1987schemas}. There is a summary (in English) in \cite[\S 1]{lemaire2009comparison}. See also \cite[Example 1.15,2.10]{tits1979reductive}. 

Let $G\coloneqq \operatorname{U}(V,h)$ denote the unitary group over $F_0$ attached to $(V,h)$. Then there is an embedding of (enlarged) buildings $$\CB(G,F_0)\hookrightarrow \CB(\GL_F(V),F).$$ 
 
\begin{defn}
	A \dfn{norm} on $V$ is a map $\alpha: V\ra \BR\cup\cbra{+\infty}$ such that for $x,y\in V$ and $\lambda\in F$, we have \begin{flalign*}
	   \alpha(x+y) \geq \inf\cbra{\alpha(x),\alpha(y)},\ \alpha(\lambda x)=\omega(\lambda)+\alpha(x), \text{\ and\ } x=0 \Leftrightarrow \alpha(x)=+\infty.
\end{flalign*}
\end{defn}

\begin{example}\label{ex-norms}
	\begin{enumerate}
	    \item Let $V$ be a one dimensional $F$-vector space. Then any norm $\alpha$ on $V$ is uniquely determined by its value of a non-zero element in $V$: for any $0\neq x\in V$ and $\lambda\in F$, we have $$\alpha(\lambda x)=\omega(\lambda)+\alpha(x).$$ 
		\item Let $V_1$ and $V_2$ be two finite dimensional $F$-vector spaces. Let $\alpha_i$ be a norm on $V_i$ for $i=1,2$. The \dfn{direct sum} of $\alpha_1$ and $\alpha_2$ is defined as a norm $\alpha_1\oplus\alpha_2: V_1\oplus V_2\ra \BR\cup\cbra{+\infty}$ via $$(\alpha_1\oplus\alpha_2)(x_1+x_2)\coloneqq \inf\cbra{\alpha_1(x_1),\alpha_2(x_2)}, \text{\ for\ }x_i\in V_i.$$
	\end{enumerate}
\end{example}

\begin{prop}[{\cite[15.1.11]{kaletha2023bruhat}}]
	Let $\alpha$ be a norm on $V$. Then there exists a basis $(e_i)_{1\leq i\leq n}$ of $V$ and $n$ real numbers $c_i$ for $1\leq i\leq n$ such that $$\alpha(\sum_{i=1}^nx_ie_i)=\inf_{1\leq i\leq n}\cbra{\omega(x_i)-c_i}.$$
	In this case, we say $(e_i)_{1\leq i\leq n}$ is a splitting basis of $\alpha$, or $\alpha$ is split by $(e_i)_{1\leq i\leq n}$.
\end{prop}

Denote by $\CN$ the set of all norms on $V$. Then $\CN$ carries a natural $\GL_F(V)(F)$-action via \begin{flalign}
	  (g\alpha)(x)\coloneqq \alpha(g\inverse x), \text{\ for\ } g\in \GL_F(V)(F) \text{\ and\ } x\in V. 
	 \label{action}
\end{flalign} 

For each $F$-basis $\rbra{e_i}_{1\leq i\leq n}$ of $V$, we have a corresponding maximal $F$-split torus $T$ of $\GL_F(V)$ whose $F$-points are diagonal matrices with respect to the basis $(e_i)_{1\leq i\leq n}$. 
The cocharacter group $X_*(T)$ has a $\BZ$-basis $(\mu_i)_{1\leq i\leq n}$, where $\mu_i: \BG_{m,F}\ra T$ is a cocharacter characterized by \begin{flalign}
	     \mu_i(t)e_j=t^{-\delta_{ij}} e_j, \text{\ for $t\in F\cross$ and $1\leq i,j\leq n$}, \label{eqmui}
\end{flalign} where $\delta_{ij}$ is the Kronecker symbol. Fixing an origin, we may identify the apartment $\CA\sset \CB(\GL_F(V),F)$ corresponding to $T$ with $X_*(T)_\BR$. 

\begin{prop}[{\cite[2.8,2.11]{bruhat1984schemas}}]  \label{prop23}
   The map \begin{flalign}
	    \CA=X_*(T)_\BR &\lra \CN \label{mapnorm} \\ \sum_{i=1}^n c_i\mu_i &\mapsto \rbra{\sum_{i=1}^n x_i e_i\mapsto \inf_{1\leq i\leq n} \cbra{\omega(x_i)-c_i }   }, \notag
\end{flalign} where $c_i\in \BR$, $x_i\in F$ and $\sum_{i=1}^n x_ie_i\in V$, extends uniquely to an isomorphism of $\GL_F(V)$-sets $$\CB(\GL_F(V),F) \simto \CN.$$ Moreover, the image of $X_*(T)_\BR$ in $\CN$ is the set of norms on $V$ admitting $(e_i)_{1\leq i\leq n}$ as a splitting basis. 
\end{prop}

By Proposition \ref{prop23}, we can identify the building $\CB(\GL_F(V),F)$ with the set $\CN$ of norms on $V$.
Next we will describe the image of the inclusion $\CB(G,F_0)\hookrightarrow \CB(\GL_F(V),F)=\CN$ in terms of maxi-minorant norms (\dfn{norme maximinorante} in French).  

Set $F_\sigma\coloneqq \cbra{\lambda-\lambda^\sigma\ |\ \lambda\in F}$. Then $F_\sigma$ is an $F_0$-subspace of $F$ and we denote by $F/F_\sigma$ the quotient space. We can associate the hermitian form $h$ with a map $\ol{q}: V\ra F/F_\sigma$, called the \dfn{pseudo-quadratic form} in \cite{bruhat1987schemas}, defined by $$\ol{q}(x)\coloneqq \frac{1}{2}h(x,x)+F_\sigma, \text{\ for\ }x\in V.$$
The valuation $\omega$ induces a quotient norm $\ol{\omega}$ on the $F_0$-vector space $F/F_\sigma$: \begin{flalign*}
	   \ol{\omega}(\lambda+F_\sigma)\coloneqq \sup\cbra{\omega(\lambda+\mu-\mu^\sigma )\ |\ \mu\in F }, \text{\ for $\lambda\in F$}.
\end{flalign*}
\begin{defn}
	Let $\alpha$ be a norm on $V$. We say $\alpha$ \dfn{minorizes} (\dfn{minores} in French) $(h,\ol{q})$ if for all $x,y\in V$,  \begin{flalign*}
		    \alpha(x)+\alpha(y) \leq \omega(h(x,y))\text{\ and\ } \alpha(x) \leq  \frac{1}{2}\ol{\omega}(\ol{q}(x)).
	\end{flalign*}
	Following the terminology of \cite[Remark 15.2.12]{kaletha2023bruhat}, we say $\alpha$ is \dfn{maxi-minorant} (\dfn{maximinorante} in French) for $(h,\ol{q})$ if $\alpha$ minorizes $(h,\ol{q})$ and $\alpha$ is maximal for this property. 
\end{defn}
Denote by $\CN_{mm}$ ($\subset \CN$) the set of maxi-minorant norms for $(h,\ol{q})$ on $V$. One can easily check that $\CN_{mm}$ carries a $G(F_0)$-action via \eqref{action}. Here we view $G(F_0)$ as a subgroup of $\GL_F(V)$.  

\begin{remark} \label{rmk-25}
	Let $\alpha$ be a norm on $V$. Set $$\alpha^\vee(x)\coloneqq \inf_{y\in V}\cbra{\omega(h(x,y))-\alpha(y)}, \text{\ for $x\in V$}.$$ Then $\alpha^\vee$ is also a norm on $V$, called the \dfn{dual norm} of $\alpha$. We say $\alpha$ is \dfn{self-dual} if $\alpha=\alpha^\vee$. If $F$ has odd residue characteristic, then by \cite[2.16]{bruhat1987schemas}, the norm $\alpha\in\CN_{mm}$ if and only if $\alpha$ is self-dual. 
\end{remark}

Note that for $x\in V$, we have \begin{flalign*}
	  \ol{q}(x) = \frac{1}{2}h(x,x)+F_\sigma=\tcbra{\frac{1}{2}h(x,x)+\mu-\mu^\sigma\ |\ \mu\in F }  =\tcbra{\lambda h(x,x)\ |\ \lambda \in F, \lambda +\lambda ^\sigma=1 }\in F/F_\sigma.
\end{flalign*} Therefore, \begin{flalign*}
	\ol{\omega}(\ol{q}(x)) =\sup\cbra{\omega(\lambda h(x,x))\ |\ \lambda \in F, \lambda +\lambda ^\sigma=1  } =\omega(h(x,x))+\sup\cbra{\omega(\lambda)\ |\ \lambda\in F,\lambda+\lambda^\sigma=1 }.
\end{flalign*} 
Set \begin{flalign}
	\delta\coloneqq \sup\cbra{\omega(\lambda)\ |\ \lambda\in F,\lambda+\lambda^\sigma=1 }. \label{delta}
\end{flalign}
We obtain that $\alpha$ minores $(h,\ol{q})$ if and only if for $x,y\in V$, we have \begin{flalign*}
		    \alpha(x)+\alpha(y) \leq \omega(h(x,y)) \text{\ and\ } \alpha(x) \leq \frac{1}{2}\omega(h(x,x))+\frac{1}{2} \delta.
	\end{flalign*}

\begin{defn} 
 Let $(V,h)$ be a (split) hermitian $F$-vector space of dimension $n$ as in \S \ref{subsec-notation}.
    \begin{enumerate}
    	\item A \dfn{Witt decomposition} of $V$ is a decomposition $V=V_-\oplus V_0\oplus V_+$ such that $V_-$ and $V_+$ are two maximal isotropic subspaces of $V$, and $V_0$ is the orthogonal complement of $V_-\oplus V_+$ with respect to $h$. As we assume $h$ is split, we have $\dim_F V_-=\dim_F V_+=m$ and $\dim_F V_0=1$. 
    	\item For any $F$-basis $(e_i)_{1\leq i\leq n}$ of $V$, we put \begin{flalign*}
    		    V_-\coloneqq \Span_F \cbra{e_1,\ldots,e_m}, V_0\coloneqq \Span_F\cbra{e_{m+1}}, V_+\coloneqq \Span_F\cbra{e_{m+2},\ldots,e_n}.
    	    \end{flalign*} We say $(e_i)_{1\leq i\leq n}$ induces a Witt decomposition of $V$ if $V_-\oplus V_0\oplus V_+$ is a Witt decomposition of $V$ and $h(e_i,e_j)=\delta_{i,n+1-j}$ for $1\leq i,j\leq n$.
    \end{enumerate}
\end{defn}
Let $(e_i)_{1\leq i\leq n}$ be a basis of $V$ inducing a Witt decomposition. Such a basis defines a maximal $F_0$-split torus $S$ of $G$ whose $F_0$-points are given by \begin{flalign*}
	     \cbra{g\in G(F_0)\subset\GL_F(V)(F) \ |\ ge_i=x_ie_i \text{\ and $x_ix_{n+1-i}=x_{m+1}=1$ for some $x_i\in F_0$ and $1\leq i\leq n$} }.
\end{flalign*}
The centralizer of $S$ in $G\otimes_{F_0}F\simeq\GL_F(V)$ is $T$. For $m+2\leq i\leq n$, let $\lambda_i: \BG_{m,F_0}\ra S$ be the cocharacter of $S$ defined by \begin{flalign}
	    \lambda_i(t)e_i=t^{-1} e_i,\ \lambda_i(t)e_{n+1-i}=t e_{n+1-i}, \text{\ and $\lambda_i(t)e_j=e_j$ for $t\in F_0\cross$ and $j\neq i, n+1-i$}.   \label{lambdabasis}
\end{flalign} Then the set $(\lambda_i)_{m+2\leq i\leq n}$ forms a $\BZ$-basis of $X_*(S)$. Fixing an origin, we may identify the apartment $\CA(G,S)$ of $\CB(G,F_0)$ corresponding to $S$ with $X_*(S)_\BR$. Then we have the following proposition.

\begin{prop} \label{prop-mmnorm}
	The map \begin{flalign}
	    X_*(S)_\BR &\lra \CN_{mm} \label{eq-mapnorm1} \\ 
	    \sum_{i=m+2}^nc_i\lambda_i &\mapsto \rbra{\sum_{i=1}^nx_ie_i\mapsto \inf\tcbra{\omega(x_i)-c_i,\omega(x_{m+1})+ \frac{1}{2}\delta\ |\ 1\leq i\leq n \text{\ and\ } i\neq m+1 } }, \notag
\end{flalign}
where $c_i\coloneqq -c_{n+1-i}$ if $1\leq i\leq m$, extends uniquely to an isomorphism of $G(F_0)$-sets $$\CB(G,F_0)\ra \CN_{mm}.$$ The image of $X_*(S)_\BR$ in $\CN_{mm}$ is the set of maxi-minorant norms admitting $(e_i)_{1\leq i\leq n}$ as a splitting basis.  

Moreover, a norm $\alpha\in\CN_{mm}$ is special, i.e., $\alpha$ corresponds to a special point in $\CB(G,F_0)$, if and only if there is a basis $(f_i)_{1\leq i\leq n}$ of $V$ inducing a Witt decomposition and a constant $C\in \frac{1}{4}\BZ$ such that for $x_i\in F$, we have \begin{flalign*}
	    \alpha(\sum_{i=1}^nx_if_i) =\inf\tcbra{\omega(x_i)-C, \omega(x_j)+C, \omega(x_{m+1})+\frac{1}{2}\delta\ |\ 1\leq i<m+1 \text{\ and\ } m+1<j\leq n  }.
\end{flalign*}
\end{prop}
\begin{proof}
	See \cite[2.9, 2.12]{bruhat1987schemas} and \cite[Example 2.10]{tits1979reductive}.
\end{proof}

\begin{corollary}\label{cor-mmnorm}
	Let $\alpha\in\CN$. Then $\alpha\in\CN_{mm}$ if and only if there exists a basis $(f_i)_{1\leq i\leq n}$ of $V$ inducing a Witt decomposition $V=V_-\oplus V_0\oplus V_+$ such that $\alpha=\alpha_{\pm}\oplus\alpha_0$, where $\alpha_{\pm}$ is a self-dual norm on $V_-\oplus V_+$ split by the basis $(f_i)_{i\neq m+1}$, and $\alpha_0$ is the unique norm on $V_0$ with $\alpha(f_{m+1})=\frac{1}{2}\delta$. 
\end{corollary}
\begin{proof}
	($\Rightarrow$) We can view $X_*(S)_\BR$ as a subset of $\CN_{mm}$ via the map \eqref{eq-mapnorm1}. Using the $G(F_0)$-action, we may assume $\alpha$ lies in $X_*(S)_\BR$, say $\alpha=\sum_{i=m+2}^nc_i\lambda_i \in X_*(S)_\BR$ for $c_i\in \BR$. Then we take $(f_i)$ to be $(e_i)$, which induces a Witt decomposition $V=V_-\oplus V_0\oplus V_+$. Define the norm $\alpha_\pm$ on $V_-\oplus V_+$ by  \begin{flalign}
		     V_-\oplus V_+ &\lra \BR\cup\cbra{+\infty} \notag \\  \sum_{1\leq i\leq n,i\neq m+1} x_if_i&\mapsto \inf\cbra{\omega(x_i)-c_i\ |\ 1\leq i\leq n \text{\ and\ } i\neq m+1}, \label{eq-alphapm}
	\end{flalign} 
	where we define $c_i\coloneqq -c_{n+1-i}$ for $1\leq i\leq m$. Clearly $\alpha_\pm$ is split by $(f_i)_{i\neq m+1}$. As $h(f_i,f_{n+1-j})=\delta_{ij}$ and $c_i=-c_{n+1-i}$ for $1\leq i,j\leq n$, we deduce that $\alpha_\pm$ is self-dual by \cite[Remark 15.2.7]{kaletha2023bruhat}. Moreover, from the expression of \eqref{eq-mapnorm1}, we immediately see that $\alpha$ decomposes as $\alpha=\alpha_{\pm}\oplus\alpha_0$. 
	
	($\Leftarrow$) Under the assumptions, there exist $n$ real numbers $c_i$ for $1\leq i\leq n$ such that $c_{n+1-i}=-c_i$ and $\alpha_\pm$ is given by the norm as in \eqref{eq-alphapm}. Let $S'$ be the maximal $F_0$-split torus in $G$ corresponding to the basis $(f_i)_{1\leq i\leq n}$. Let $(\lambda_i')_{m+2\leq i\leq n}$ be a $\BZ$-basis of $X_*(S')$ defined as in \eqref{lambdabasis}. Then $\alpha$ is the norm corresponding to the point $\sum_{i=m+2}^nc_i\lambda_i'\in X_*(S')_\BR$ via a similar map as in \eqref{eq-mapnorm1}. In particular, $\alpha\in\CN_{mm}$. 
\end{proof}

\begin{remark} \label{rmk-29}
	Assume $F$ has odd residue characteristic. Then $\delta=0$, and hence $\alpha_0$ is self-dual. Then the norm $\alpha_\pm\oplus\alpha_0$ as in the Corollary \ref{cor-mmnorm}  is self-dual. When $F$ has odd residue characteristic, any self-dual norm admits a splitting basis inducing a Witt decomposition of $V$, see for example \cite[Proposition 15.2.10]{kaletha2023bruhat}. Then we see again that $\alpha\in\CN_{mm}$ if and only $\alpha$ is self-dual. 
\end{remark}
\begin{remark}
	We can define a ``twisted" Galois action of $\Gal(F/F_0)$ on $\GL_F(V)(F)$ as follows: for $g\in \GL_F(V)(F)$, define $\sigma(g)$ to be the element satisfying $$h(g\inverse x,y)=h(x,\sigma(g)y), \text{\ for $x,y\in V$}. $$ Then we have $G(F_0)=\GL_F(V)(F)^{\sigma=1}$, the set of fixed points of $\sigma$. This twisted Galois action induces an involution on $\CN=\CB(\GL_F(V),F)=\CB(G\otimes_{F_0}F,F)$, which is still denoted by $\sigma$. Next we give an explicit description of this involution.  
	
	Let $(e_i)_{1\leq i\leq n}$ be a basis inducing a Witt decomposition $V=V_-\oplus V_0\oplus V_+$. Let $T$ be the induced maximal torus of $\GL_F(V)$. Let $\CA(T)\sset \CB(\GL_F(V),F)$ be the apartment corresponding to $T$. We can identify $\CA(T)$ with $X_*(T)_\BR$ through the injection (cf. \eqref{mapnorm}) 
	\begin{flalign*}
	     X_*(T)_\BR &\lra \CN  \\ \sum_{i=1}^n c_i\mu_i &\mapsto \rbra{\sum_{i=1}^n x_i e_i\mapsto \inf \tcbra{\omega(x_{m+1})-c_{m+1}+ \half\delta,\ \omega(x_i)-c_i \text{\ for\ }1\leq i\leq n \text{\ and $i\neq m+1$}  } }, \notag
        \end{flalign*} where $\mu_i$ is defined as in \eqref{eqmui}, $x_i\in F$ and $\sum_{i=1}^n x_ie_i\in V$. As $G$ is quasi-split, we can pick a $\sigma$-stable point as the origin such that the twisted $\sigma$-action on $\CA(T)$ is transported by the twisted $\sigma$-action on $X_*(T)_\BR$. For $\alpha\in\CN$, there is a $g\in \GL_F(V)(F)$ such that $g\alpha\in X_*(T)_\BR$, since $\GL_F(V)(F)$ acts transitively on the apartments of $\CN$. Then $$g\alpha=\alpha_1\oplus(\alpha_0+C),$$ where $\alpha_1$ is a norm on $V_-\oplus V_+$ admitting $(e_i)_{i\neq m+1}$ as a splitting basis, $\alpha_0$ is the norm on $V_0$ as in the Corollary \ref{cor-mmnorm}, and $C\in \BR$ is a certain constant. The twisted $\sigma$-action on $X_*(T)_\BR$ implies that $\sigma(\alpha_1\oplus(\alpha_0+C))=\alpha_1^\vee\oplus(\alpha_0-C)$. Hence, we see that $\sigma$ acts on $\alpha$ as \begin{flalign*}
		    \sigma(\alpha)=\sigma(g\inverse)\rbra{\alpha_1^\vee\oplus(\alpha_0-C) }.
	\end{flalign*}
	For $\alpha\in\CN_{mm}=\CB(G,F_0)$, we may take $g\in G(F_0)$ and $C=0$. Thus, we get an inclusion $$\CB(G,F_0)\hookrightarrow \CB(\GL_F(V),F)^{\sigma=1}.$$ The inclusion is strict: any norm of the form $\alpha_1\oplus\alpha_0$, where $\alpha_1$ is a self-dual norm on $V_-\oplus V_+$ but not split by any basis of $V_-\oplus V_+$ inducing a Witt decomposition, lies in $\CB(\GL_F(V),F)^{\sigma=1}$ but not in $\CB(G,F_0)$. Such a norm can only exist when the residue characteristic of $F$ is two. For an explicit example, see Example \ref{ex-nonmmnorm}.
\end{remark}

\subsection{Bruhat-Tits buildings in terms of lattices}
In this subsection, we will translate the results in \S \ref{subsec-buildingnorm} into the language of lattices, which is more useful in the theory of local models. 

\begin{defn}
    Let $V$ be a finite dimensional $F$-vector space.  
	\begin{enumerate}
	    \item A \dfn{lattice} $L$ in $V$ is a finitely generated $\CO_F$-submodule of $V$ such that $L\otimes_{\CO_F}F=V$. 
		\item A \dfn{(periodic) lattice chain} of $V$ is a non-empty set $L_\bullet$ of lattices in $V$ such that lattices in $L_\bullet$ are totally ordered with respect to the inclusion relation, and $\lambda L\in L_\bullet$ for $\lambda\in F\cross$ and $L\in L_\bullet$.
		\item A \dfn{graded lattice chain} is a pair $(L_\bullet,c)$, where $L_\bullet$ is a lattice chain of $V$ and $c: L_\bullet\ra \BR$ is a strictly decreasing function such that for any $\lambda\in F$ and $L\in L_\bullet$, we have $$c(\lambda L)=\omega(\lambda)+c(L).$$ The function $c$ is called a \dfn{grading} of $L_\bullet$.
		\item An $F$-basis $(e_i)_{1\leq i\leq n}$ of $V$ is called \dfn{adapted to} a graded lattice chain $(L_\bullet,c)$ of $V$ if for every $L\in L_\bullet$, there exist $x_1,\ldots,x_n\in F$ such that $(x_ie_i)_{1\leq i\leq n}$ is an $\CO_F$-basis of $L$. In this case, we also say $(L_\bullet,c)$ is adapted to the basis $(e_i)_{1\leq i\leq n}$. 
	\end{enumerate}
\end{defn}

\begin{remark}
	Since $L_\bullet$ is stable under homothety, the set $L_\bullet$ is determined by a finite number of lattices satisfying $$\pi L_0\subsetneq L_{r-1}\subsetneq L_{r-2}\subsetneq\cdots\subsetneq L_1\subsetneq L_0.$$ We say $(L_0,L_1,\ldots,L_{r-1})$ is a \dfn{segment} of $L_\bullet$, and the integer $r$ is the \dfn{rank} of $L_\bullet$. 
\end{remark}

Denote by $\CGLC$ the set of graded lattice chains of $V$. There is a $\GL_F(V)(F)$-action on $\CGLC$: for $(L_\bullet,c)\in \CGLC$ and $g\in\GL_F(V)(F)$, define $g(L_\bullet,c)\coloneqq (gL_\bullet,gc)$, where $gL_\bullet$ consists of lattices of the form $gL$  for $L\in L_\bullet$, and $(gc)(gL)\coloneqq c(L)$ for $L\in L_\bullet$. 

\begin{lemma}\label{lem-lat-norm}
    \begin{enumerate}
    	\item There is a one-to-one correspondence between $\CN$ and $\CGLC$. More precisely, given $\alpha\in\CN$, we can associate a graded lattice chain $(L_\alpha,c_\alpha)$, where $L_\alpha$ is the set of following lattices $$L_{\alpha,r}=\cbra{x\in V\ |\ \alpha(x)\geq r}, \text{\ for\ }r\in \BR,$$ and the grading $c_\alpha$ is defined by $$c_\alpha(L_{\alpha,r})=\inf_{x\in L_{\alpha,r}}\alpha(x).$$
	Conversely, given a graded lattice chain $(L_\bullet,c)\in\CGLC$, we can associate a norm $$\alpha_{(L_\bullet,c)}(x)\coloneqq  \sup\cbra{c(L)\ |\ x\in L \text{\ and\ }L\in L_\bullet}.$$
	We say the norm $\alpha$ and the graded lattice chain $(L_\alpha,c_\alpha)$ in the above bijection \dfn{correspond to} each other.
	   \item The bijection in (1) is $\GL_F(V)(F)$-equivariant. 
	   \item Let $(e_i)_{1\leq i\leq n}$ be a basis of $V$. Let $(L_\bullet,c)$ be the graded lattice chain corresponding to a norm $\alpha$ via (1). Then $(e_i)_{1\leq i\leq n}$ is adapted to $(L_\bullet,c)$ if and only if $(e_i)_{1\leq i\leq n}$ is a splitting basis of $\alpha$.
    \end{enumerate}
\end{lemma}
\begin{proof}
	The proof of (1) and (3) can be found in \cite[Proposition 15.1.21]{kaletha2023bruhat}. The assertion in (2) can be checked by direct computation. 
\end{proof}
Using the above lemma, we can easily extend operations like direct sums or duality on norms to graded lattice chains. 

\begin{lemma}
   \begin{enumerate}
   	\item Let $V$ and $V'$ be two finite dimensional $F$-vector spaces. Let $\alpha$ and $\alpha'$ be two norms on $V$ and $V'$ respectively. Let $(L_\bullet,c)$ and $(L_\bullet',c')$ be graded lattice chains corresponding to $\alpha$ and $\alpha'$ respectively. Then the graded lattice chain $(L_\bullet,c)\oplus(L_\bullet',c')$ corresponding to $\alpha\oplus\alpha'$ is a pair $(L_\bullet\oplus L_\bullet' ,c\oplus c')$, where $L_\bullet\oplus L_\bullet'$ is the set of lattices of the form $L_{\alpha,r}\oplus L_{\alpha',r}$ for $r\in \BR$, and $$(c\oplus c')(L_{\alpha,r}\oplus L_{\alpha',r})\coloneqq  \inf\cbra{c(L_{\alpha,r}),c'(L_{\alpha',r}) }.$$
   	\item Let $(L_\bullet,c)$ be the graded lattice chain corresponding to a norm $\alpha$ on $V$. Then the dual norm $\alpha^\vee$ corresponds to the graded lattice chain $(L_\bullet^\vee,c^\vee)$, where $L_\bullet^\vee$ is the set of the lattices of the form $L^\vee\coloneqq \cbra{x\in V\ |\ h(x,L)\in\CO_F}$ for $L\in L_\bullet$, and $$c^\vee(L^\vee)\coloneqq -c(L^-)-1,$$ where $L^-$ is the smallest member of $L_\bullet$ that properly contains $L$.
   \end{enumerate}
\end{lemma} 
\begin{proof}
	The proof of (1) is straightforward. The proof of (2) can be found in \cite[Fact 15.2.18]{kaletha2023bruhat}.
\end{proof}
We say $(L_\bullet,c)$ is \dfn{self-dual} if $(L_\bullet,c)=(L_\bullet^\vee,c^\vee)$. 

\begin{prop}\label{prop-latchain}
	Let $(L_\bullet,c)\in\CGLC$. Then $(L_\bullet,c)$ corresponds to a norm in $\CN_{mm}$ if and only if there exists a basis $(f_i)_{1\leq i\leq n}$ of $V$ inducing a Witt decomposition $V=V_-\oplus V_0\oplus V_+$ and $(L_\bullet,c)$ decomposes as $(L_\bullet^\pm,c^\pm)\oplus(L_\bullet^0,c^0)$, such that $(L_\bullet^\pm,c^\pm)$ is a self-dual graded lattice chain of $V_-\oplus V_+$ adapted to the basis $(f_i)_{i\neq m+1}$, and $(L_\bullet^0,c^0)$ is the graded lattice chain corresponding to the norm $\alpha_0$ on $V_0$.
\end{prop}
\begin{proof}
	This is a translation of Corollary \ref{cor-mmnorm} in view of the previous two lemmas. 
\end{proof}

\begin{remark}\label{rmk-hyp}
	Let $(L_\bullet^\pm,c^\pm)$ be a self-dual graded lattice chain adapted to the basis $(f_i)_{i\neq m+1}$ as in Proposition \ref{prop-latchain}. Then for any $L\in L_\bullet^\pm$, there exist $x_i\in F$ for $i\neq m+1$ such that $(x_if_i)_{i\neq m+1}$ forms an $\CO_F$-basis of $L$. As $h(f_i,f_j)=\delta_{i,n+1-j}$, we see that $L$ is isomorphic to an orthogonal sum of ``hyperbolic planes" of the form $H(i)$ ($i\in\BZ$). Here $H(i)$ denotes a lattice in a two dimensional hermitian $F$-vector space $(W,h)$ such that $H(i)$ is $\CO_F\pair{x,y}$ spanned by some $x,y\in W$ with $h(x,x)=h(y,y)=0$ and $h(x,y)=\pi^i$. 
	
	A lattice in $W$ which is isomorphic to $H(i)$ for some $i\in \BZ$ is also called a \dfn{hyperbolic lattice} in the sense of \cite[\S 2]{kirch2017construction}. For any lattice $K$ in $W$, define the \dfn{norm ideal} $n(K)$ of $K$ to be the ideal in $\CO_{F_0}$ generated by $h(x,x)$ for $x\in K$. Let $K^\vee$ denote the dual lattice of $K$ with respect to the hermitian form $h$ on $W$. Then by \cite[\S 2]{kirch2017construction} (see also \cite[Proposition 9.2 (a)]{jacobowitz1962hermitian}), 
	any lattice $K\sset W$ satisfying $K=\pi^iK^\vee$ (that is, $K$ is $\pi^i$-modular) and $n(K)=n(H(i))$ is isomorphic to $H(i)$.
\end{remark}
\begin{example}\label{ex-nonmmnorm}
	Let $F_0=\BQ_2$ and $F=\BQ_2(\sqrt{3})$. Pick uniformizers $\pi=\sqrt{3}-1\in F$ and $\pi_0=-2\in F_0$ so that $\pi^2+2\pi-2=0$. We have $$ \delta=\sup\cbra{\omega(\lambda)\ |\ \lambda\in F,\lambda+\lambda^\sigma=1 }=\omega( \frac{\pi}{2}) = -\half.$$ Let $(V,h)$ be a $3$-dimensional (split) hermitian $F$-vector space. Let $(e_i)_{1\leq i\leq 3}$ be a basis of $V$ inducing a Witt decomposition $V=V_-\oplus V_0\oplus V_+$. Denote $V_\pm\coloneqq V_-\oplus V_+=F\pair{e_1,e_3}$. Set $$f_1\coloneqq \pi\inverse(e_1+e_3),\ f_2\coloneqq e_2,\ f_3\coloneqq \pi\inverse(e_1-e_3).$$ Then $L_1\coloneqq \CO_F\pair{f_1,f_3}$ is a self-dual lattice in $(V_\pm,h)$. By \cite[Equation (9.1)]{jacobowitz1962hermitian}, the self-dual hyperbolic plane $H(0)$ in $V_\pm$ has norm ideal $2\CO_{F_0}$. On the other hand, we have $n(L_1)=\CO_{F_0}$ by direct computation. In particular, the self-dual lattice $L_1$ in $(V_\pm,h)$ is not isomorphic to $H(0)$, and hence $L_1$ is not adapted to any basis of $V_\pm$ induing a Witt decomposition. 
	
	Now define $$L\coloneqq L_1\oplus\CO_Ff_2.$$ Then the graded lattice chain $(L_\bullet,c)$, where $L_\bullet\coloneqq \tcbra{\pi^iL}_{i\in\BZ}$ and $c(\pi^iL)\coloneqq \frac{i}{2}+\frac{\delta}{2}=\frac{i}{2}-\frac{1}{4}$, defines a norm \begin{flalign*}
		    \alpha: V &\lra \BR\cup\cbra{+\infty}\\ \sum_{i=1}^3x_if_i &\mapsto \inf_{1\leq i\leq 3}\tcbra{\omega(x_i)- \frac{1}{4}}. 
	\end{flalign*} 
	Then we see $\alpha$ lies in the fixed point set $\CB(\GL_F(V),F)^{\sigma=1}=\CN^{\sigma=1}$, but does not lie in $\CN_{mm}$.
\end{example}

\subsection{Parahoric subgroups and lattices}\label{subsec-paralattice}
Let us keep the notations as in \S \ref{subsec-buildingnorm}. In particular, the set $(e_i)_{1\leq i\leq n}$ denotes a basis of $V$ inducing a Witt decomposition $V=V_-\oplus V_0\oplus V_+$ and $S$ denotes the corresponding maximal $F_0$-split torus of $G=\operatorname{U}(V,h)$. Denote by $(a_i)_{m+2\leq i\leq n}\in X^*(S) $ the dual basis of $(\lambda_i)_{m+2\leq i\leq n}\in X_*(S)$. 

By the calculations in \cite[Example 1.15]{tits1979reductive}, the relative root system $\Phi=\Phi(G,S)$ is $$\cbra{\pm a_i\pm a_j \ |\ m+2\leq i,j\leq n,i\neq j}\cup \cbra{\pm a_i,\pm 2a_i \ |\ m+2\leq i\leq n},$$ and the affine root system $\Phi_a$ is \begin{flalign*}
	   &\tcbra{\pm a_i\pm a_j+ \half\BZ\ |\ m+2\leq i,j\leq n, i\neq j }\\  \cup & \tcbra{\pm a_i+  \half\delta+\half\BZ\ |\ m+2\leq i\leq n }\cup\tcbra{\pm 2a_i+  \half+\delta+\BZ\ |\ m+2\leq i\leq n }.  
\end{flalign*} 
Here, $\delta$ is defined as in \eqref{delta}. These affine roots endow $X_*(S)_\BR$ with a simplicial structure. Following \cite[Example 3.11]{tits1979reductive}, we pick a chamber defined by the inequalities $$ \half\delta<a_{m+2}<\cdots<a_n<\half\delta+\frac{1}{4}.$$ Then we obtain $m+1$ vertices $v_0,\ldots,v_m$ in $X_*(S)_\BR$ such that for $0\leq i\leq m$, \begin{flalign*}
	    a_{j}(v_i)=\begin{cases}
	    	\frac{1}{2}\delta \quad &\text{if $m+2\leq j\leq n-i$},\\ \half\delta+\frac{1}{4} &\text{if $n-i<j\leq n$}.
	    \end{cases}
\end{flalign*}
Now each $v_i$ defines a (maxi-minorant) norm, and hence a graded lattice chain, by Proposition \ref{prop-mmnorm} and Lemma \ref{lem-lat-norm}. Let $\lambda\in F$ be an element satisfying $\omega(\lambda)=\delta$. We shall see an explicit expression of $\lambda$ in Lemma \ref{lem-lambda}.
Define \begin{flalign}
	   \Lambda_i &\coloneqq \CO_F\pair{\pi\inverse e_1,\ldots,\pi\inverse e_i,e_{i+1},\ldots,e_{m+1},\lambda e_{m+2},\ldots,\lambda e_n } \label{eq-Lambdai},   \\ \Lambda_i' &=\CO_F\pair{e_1,\ldots,e_m,e_{m+1},\lambda e_{m+2},\ldots,\lambda e_{n-i},\lambda\pi e_{n+1-i},\ldots,\lambda\pi e_n }.  \notag
\end{flalign} Then the graded lattice chain corresponding to $v_i$ is of rank $2$ and has a segment \begin{flalign*}
	    \pi\Lambda_i\subset \Lambda_i'\subset \Lambda_i.
\end{flalign*} 
Let $\wt{G}=\GU(V,h)$ be the unitary similitude group attached to the hermitian space $(V,h)$. Let $I$ be a non-empty subset of $\cbra{0,1,\ldots,m}$. Define $$P_I\coloneqq \cbra{g\in \wt{G}(F_0)\ |\ g\Lambda_i=\Lambda_i, \text{\ for $i\in I$} }.$$
As in \cite[1.2.3]{pappas2009local}, the Kottwitz map restricted to $P_I$ is trivial. In particular, we obtain that the (maximal) parahoric subgroup of $\wt{G}(F_0)$ is the stabilizer of $v_i$ in $\wt{G}(F_0)$, which also equals the stabilizer of $\Lambda_i$ in $\wt{G}(F_0)$ (as the stabilizer of $\Lambda_i'$ is larger). More generally, we have the following proposition.

\begin{prop} \label{prop-para}
	Denote $\wt{G}=\GU(V,h)$. The subgroup $P_I$ is a parahoric subgroup of $\wt{G}(F_0)$. Any parahoric subgroup of $\wt{G}(F_0)$ is conjugate to a subgroup $P_I$ for a unique $I\subset\cbra{0,1,\ldots,m}$. The conjugacy classes of special parahoric subgroups correspond to the sets $I=\cbra{0}$ and $\cbra{m}$. 
\end{prop}
\begin{proof}
	The results are similar to those in \cite[\S 4]{pappas2008twisted} and \cite[1.2.3]{pappas2009local}. The first two assertions follow from the fact that $\wt{G}(F_0)$ acts transitively on the chambers in the building, and each $I$ determines a (unique) facet in a chamber. 
    By \eqref{eq-mapnorm1}, the vertex $v_0$ (resp. $v_m$) corresponds to the norm 
    \begin{flalign*}
        \sum_{i=1}^nx_ie_i\mapsto \inf\tcbra{\omega(x_i)-C, \omega(x_j)+C, \omega(x_{m+1})+\frac{1}{2}\delta\ |\ 1\leq i<m+1 \text{\ and\ } m+1<j\leq n  },
    \end{flalign*} 
    where $C=-\half\delta$ (resp. $-\half\delta-\frac{1}{4}$). By Proposition \ref{prop-mmnorm}, the vertices $v_0$ and $v_m$, and thus the subsets $I=\tcbra{0}$ and $\tcbra{m}$, correspond to the conjugacy classes of special parahoric subgroups. Moreover, it is known that there are precisely two special vertices in a chamber, as shown by the local Dynkin diagram of type $C\-BC_n$ in \cite[\S 4.2, p.~60]{tits1979reductive}). The final assertion then follows.     
\end{proof}

%

\section{Construction of the unitary local models}\label{sec-main}

\subsection{Quadratic extensions of $2$-adic fields}\label{subsec-prel}
We start with some basic facts about quadratic extensions of $2$-adic fields. The readers can find more details in \cite[\S 5]{jacobowitz1962hermitian} and \cite[\S 63]{Meara2000}.

\begin{prop}\label{quad}
	Let $E$ be a finite extension of $\BQ_2$ of degree $d$ with ring of integer $\CO_E$. Let $e$ (resp. $f$) be the ramification degree (resp. residue degree) of the field extension $E/\BQ_2$. Note that $d=ef$. 
	\begin{enumerate}
		\item The map sending $a$ to $E(\sqrt{a})$ defines a bijection between $E\cross/(E\cross)^2$ and the set of isomorphism classes of field extensions of $E$ of degree at most two. Furthermore, the cardinality of $E\cross/(E\cross)^2$ is $2^{2+d}$. In particular, we have $2^{2+d}-1$ quadratic extensions of $E$. 
		\item Let $U$ be the unit group of $\CO_E$ and $\varpi$ be a uniformizer of $\CO_E$. For $i\geq 1$, let $U_i\coloneqq 1+\varpi^i\CO_{E}$ be a subgroup of $U$. Then 
			     $U_{i}$ is contained in $U^2$ for $i\geq 2e+1$
		 and the quotient $U_{2e}/(U_{2e}\cap U^2)$ has two elements corresponding to the trivial extension and the unramified quadratic extension of $E$. Note that $U_{2e}=1+4\CO_E$.
		\item Any non-trivial element in $E\cross/(E\cross)^2$ has a representative of the following three forms: 
		    \begin{enumerate}[label=(\roman*)]
		          \item a unit in $U_{2e}- U_{2e+1}$ (elements in $U_{2e}$ but not in $U_{2e+1}$), 
		          \item a prime element in $E$,
		          \item a unit in $U_{2i-1}- U_{2i}$ for some $1\leq i\leq e$.
		    \end{enumerate}
             The corresponding quadratic extensions in (ii) and (iii) are ramified. Following \cite[\S 5]{jacobowitz1962hermitian}, we will say the (ramified) quadratic extensions in (ii) and (iii) are of type \dfn{(R-P)} and \dfn{(R-U)} respectively. There are $2^{1+d}$ quadratic extensions of $E$ of type (R-P) and $2^{1+d}-2$ quadratic extensions of $E$ of type (R-U).
         \item Let $E(\sqrt{\theta})/E$ be a quadratic extension of type (R-U) for some unit $\theta\in U_{2i-1}- U_{2i}$ for some $1\leq i\leq e$. Then there exists a prime $\pi$ in $E(\sqrt{\theta})$ and a prime $\pi_0$ in $E$ satisfying  \begin{flalign*}
         	     \pi^2-t\pi+\pi_0=0
         \end{flalign*} for some $t\in \CO_E$ with $\ord(t)=e+1-i$, where $\ord$ denotes the normalized valuation on $E$. 
	\end{enumerate} 
\end{prop}
\begin{proof}
	(1) The bijection is well-known from Kummer theory. The formula for the cardinality can be found in \cite[63:9]{Meara2000}. 
	
	(2) See \cite[63:1, 63:3]{Meara2000}.
	
	(3) See \cite[63:2]{Meara2000}. The number of quadratic extensions of type (R-U) or (R-P) follows from the cardinality formula of $E\cross/(E\cross)^2$ in (1). 
	
	(4) Let $\varpi$ be any prime in $E$. By assumption, $\theta=1+\varpi^{2i-1}u$ for some unit $u$. Set \begin{flalign*}
		     \pi \coloneqq  \frac{1-\sqrt{\theta}}{\varpi^{i-1}}\in E(\sqrt{\theta}). 
	\end{flalign*}
	    Let $\ol{\pi}$ be the Galois conjugate of $\pi$. Then \begin{flalign*}
	    	  \pi+\ol{\pi} =\frac{2}{\varpi^{i-1}}\text{\ and\ } \pi\ol{\pi} = -\varpi u.
	    \end{flalign*}
	    Now take $\pi_0$ to be $-\varpi u$ and $t$ to be $\frac{2}{\varpi^{i-1}}$. Then $t\in \CO_E$, as $\ord(t)=e+1-i\geq 1$, and $\pi$ satisfies $$\pi^2-t\pi+\pi_0=0.$$ In particular, $\pi$ is a prime element in $E(\sqrt{\theta})$.  
\end{proof}

\begin{example}
	The (ramified) quadratic extension $\BQ_2(\sqrt{3})/\BQ_2$ is of type (R-U), while $\BQ_2(\sqrt{2})/\BQ_2$ is a quadratic extension of type (R-P). 
\end{example}

Let us return to the setting in \S \ref{subsec-notation}. By Proposition \ref{quad}, we can find uniformizers $\pi\in F$ and $\pi_0\in F_0$ such that the quadratic extension $F/F_0$ falls into one of the following two distinct cases\footnote{When $F_0/\BQ_2$ is an  unramified finite extension, there is a description in \cite[\S 2A]{cho2016group} of these two cases in terms of the ramification groups of $\Gal(F/F_0)$.}:
\begin{enumerate}
	\item [(R-U)] $F=F_0(\sqrt{\theta})$, where $\theta$ is a unit in $\CO_{F_0}$. The uniformizer $\pi$ satisfies  $$\pi^2-t\pi+\pi_0=0.$$ Here $t\in\CO_{F_0}$ with $\pi_0|t\text{\ and\ }t|2$ and $\omega(t)$ depends only on $F$. We have $\sqrt{\theta}=1-\frac{2\pi}{t}$ and $\theta=1-\frac{4\pi_0}{t^2}$. 
	\item [(R-P)] $F=F_0(\sqrt{\pi_0})$, where $\pi^2+\pi_0=0$. 
\end{enumerate}

\begin{lemma}\label{lem-invdif}
	Let $F, F_0, \pi$ and $\pi_0$ be as above. \begin{enumerate}
		\item Suppose $F/F_0$ is of type (R-U). Then the inverse different (see \cite[Chapter III, \S 3]{serre2013local}) of $F/F_0$ is $\frac{1}{t}\CO_F$. 
		\item Suppose $F/F_0$ is of type (R-P). Then the inverse different of $F/F_0$ is $\frac{1}{2\pi}\CO_F$. 
	\end{enumerate}
\end{lemma}

\begin{proof}
	As $\pi$ satisfies an Eisenstein polynomial $f$, by \cite[Chapter III, \S 6, Corollary 2]{serre2013local} and \cite[Chapter I, \S 6, Proposition 18]{serre2013local}, we obtain that $\CO_F=\CO_{F_0}[\pi]$ and the inverse different of $F/F_0$ is given by  $$\delta_{F/F_0}\inverse=\frac{1}{f'(\pi)}\CO_F.$$ More precisely, \begin{enumerate}
		\item when $F/F_0$ is of type (R-U), then $f(T)=T^2-tT+\pi_0$ and $\delta_{F/F_0}\inverse =\frac{1}{2\pi-t}\CO_F=\frac{1}{t}\CO_F$, as $t|2$. 
		\item when $F/F_0$ is of type (R-P), then $f(T)=T^2+\pi_0$ and $\delta_{F/F_0}\inverse =\frac{1}{2\pi}\CO_F$. 
	\end{enumerate}
\end{proof}

\subsection{Hermitian quadratic modules and parahoric group schemes} \label{subsec-hermpara}
In this subsection, we shall define hermitian quadratic modules following \cite[\S 9]{anschutz2018extending} and relate them to parahoric group schemes.

Let $R$ be an $\CO_{F_0}$-algebra. The non-trivial Galois involution on $\CO_F$ extends to a map \begin{flalign*}
	   \CO_F\otimes_{\CO_{F_0}}R\ra \CO_F\otimes_{\CO_{F_0}}R,\ x\otimes r\mapsto \ol{x}\otimes r
\end{flalign*} for $x\in \CO_F$ and $r\in R$. We will also denote the map by $a\mapsto \ol{a}$ for $a\in \CO_F\otimes_{\CO_{F_0}}R$. The norm map on $\CO_F$ induces the map $$N_{F/F_0}: \CO_F\otimes_{\CO_{F_0}}R\ra R,\ a\mapsto a\ol{a}.$$
 
\begin{defn}[{\cite[Definition 9.1]{anschutz2018extending}}] \label{defn-hermqm}
	Let $R$ be an $\CO_{F_0}$-algebra. Let $d\geq 1$ be an integer. Consider a triple $(M,q,\sL)$, where $M$ is a locally free $\CO_F\otimes_{\CO_{F_0}}R$-module of rank $d$, $\sL$ is an invertible $R$-module, and $q: M\ra \sL$ is an $\sL$-valued quadratic form. Let $f: M\times M\ra \sL$ denote the symmetric $R$-bilinear form sending $(x,y)\in M\times M$ to $f(x,y)\coloneqq q(x+y)-q(x)-q(y)\in \sL$.
	
	We say the triple $(M,q,\sL)$ is a \dfn{hermitian quadratic module of rank $d$} over $R$ if for any $a\in \CO_F\otimes_{\CO_{F_0}}R$ and any $x,y\in M$, we have \begin{flalign}
		   q(ax)=N_{F/F_0}(a)q(x) \text{\ and\ } f(ax,y)=f(x,\ol{a} y). \label{eqq}
	\end{flalign} 
	A quadratic form $q: M\ra \sL$ satisfying \eqref{eqq} is called an \dfn{$\sL$-valued hermitian quadratic form} on $M$.
\end{defn}
\begin{defn}\label{defn-simisom}
	Let $(M_1,q_1,\sL_1)$ and $(M_2,q_2,\sL_2)$ be two hermitian quadratic modules over an $\CO_{F_0}$-algebra $R$. A \dfn{similitude isomorphism} or simply \dfn{similitude} between $(M_i,q_i,\sL_i)$ for $i=1,2$ is a pair $(\varphi,\gamma)$ of isomorphisms,  where $\varphi: M_1\simto M_2$ is an isomorphism of $\CO_F\otimes_{\CO_{F_0}}R$-modules and $\gamma:\sL_1\simto\sL_2$ is an isomorphism of $R$-modules such that $$q_2(\varphi(m_1))=\gamma(q_1(m_1)), \text{\ for any $m_1\in M_1$}.$$
\end{defn}
We will write \begin{flalign}
	  \ud{\Sim}\rbra{(M_1,q_1,\sL_1),(M_2,q_2,\sL_2) }, \text{\ or simply $\ud{\Sim}\rbra{M_1,M_2 }$}, \label{eqsim1}
\end{flalign}  for the functor over $R$ which sends an $R$-algebra $S$ to the set $\Sim(M_1\otimes_RS,M_2\otimes_RS)$ of similitude isomorphisms between $(M_i\otimes_RS,q_i\otimes_RS,\sL_i\otimes_RS)$ for $i=1,2$. In the case $(M_1,q_1,\sL_1)=(M_2,q_2,\sL_2)$, we will write \begin{flalign}
	   \ud{\Sim}(M_1,q_1,\sL_1), \text{\ or simply\ }\ud{\Sim}(M_1), \label{eqsim2}
\end{flalign}  for $\ud{\Sim}\rbra{(M_1,q_1,\sL_1),(M_2,q_2,\sL_2) }$. This is in fact a group functor, and represented by an affine group scheme of finite type over $R$. 

\begin{defn} \label{defn-hermphi}
	Let $R$ be an $\CO_{F_0}$-algebra. Denote by $\CC_R$ the category of quadruples $(M,q,\sL,\phi)$ such that $(M,q,\sL)$ is a hermitian quadratic module over $R$ and $\phi$ is an $R$-bilinear form $\phi: M\times M\ra \sL$ such that for $x,y\in M$ , we have \begin{flalign}
		    \phi(x,\pi y) =q(x+y)-q(x)-q(y),\ \phi(\pi x,y)=\phi(x,\ol{\pi}y),\ \phi(x,y) &=\phi(  \frac{\ol{\pi}}{\pi}y,x),\
		    \phi(x,x) = \frac{t}{\pi_0}q(x).  \label{eqphiq}
	\end{flalign}
	Here $t\coloneqq \pi+\ol{\pi}$. In particular, $t=0$ if $F/F_0$ is of type (R-P). 
	We will say an object $(M,q,\sL,\phi)\in\CC_R$ is a hermitian quadratic module with $\phi$, or simply a hermitian quadratic module.
	
	Let $(M_i,q_i,\sL_i,\phi_i)\in \CC_R$ for $i=1,2$. A \dfn{similitude isomorphism preserving $\phi$} between $(M_i,q_i,\sL_i,\phi_i)$ is a pair $(\varphi,\gamma)$ of isomorphisms such that $(\varphi,\gamma)$ is a similitude between $(M_i,q_i,\sL_i)$, and for $m_1,m_1'\in M_1$, we have \begin{flalign*}
		     \phi_2(\varphi(m_1),\varphi(m_1'))= \gamma(\phi_1(m_1,m_1')).
	\end{flalign*}
\end{defn}

We will use a similar notation as in \eqref{eqsim1} and \eqref{eqsim2} to denote the functor of similitudes preserving $\phi$ between two hermitian quadratic modules in $\CC_R$.

%

Recall that we defined in \S \ref{subsec-paralattice} lattices $\Lambda_i$ for $0\leq i\leq m$ via \begin{flalign*}
	 \Lambda_i &=\CO_F\pair{\pi\inverse e_1,\ldots,\pi\inverse e_i,e_{i+1},\ldots,e_{m+1},\lambda e_{m+2},\ldots,\lambda e_n },
\end{flalign*} where $\lambda$ is an element in $F$ such that $$\omega(\lambda)=\delta=\sup_{x\in F}\cbra{\omega(x)\ |\ x+\ol{x}=1}.$$ The stabilizer of $\Lambda_i$ is a maximal parahoric subgroup of $\GU(V,h)$. We sometimes call these lattices $\Lambda_i$ \dfn{standard lattices}. A more explicit expression of $\lambda$ is given as follows. 

\begin{lemma}\label{lem-lambda}
	\begin{enumerate}
		\item Suppose $F/F_0$ is of type (R-U). Then we may take $\lambda=\frac{\ol{\pi}}{t}$.
		\item Suppose $F/F_0$ is of type (R-P). Then we may take $\lambda=\frac{1}{2}$. 
	\end{enumerate}
\end{lemma}
\begin{proof}
	(1) By construction, we have $\omega(\lambda)\geq\omega(\frac{\ol{\pi}}{t})> \omega(\half)$. Write $\lambda=a+b\sqrt{\theta}\in F$ for some $a,b\in F_0$. Then $\ol{\lambda}=a-b\sqrt{\theta}$. Since $\lambda+\ol{\lambda}=1$, we get $a=\half$ and \begin{flalign*}
		     \omega(\lambda) &=\omega( \half +b\sqrt{\theta}).
	\end{flalign*}
	If $\omega(\half)\neq \omega(b\sqrt{\theta})$, then $$\omega(\lambda)=\min\tcbra{\omega(\half), \omega(b\sqrt{\theta}) }\leq\omega(\half),$$ which is a contradiction.
	Therefore, we may assume $\omega(b) =\omega(b\sqrt{\theta}) =\omega(\half)$. Then we can write $b=\half u$ for some unit $u$ in $\CO_{F_0}$. Then \begin{flalign*}
		    \omega(\lambda)&= \omega( \half+\half u(1-\frac{2\pi}{t})) =\omega\trbra{(\half+u) -\frac{\pi}{t}u }.
	\end{flalign*} 
	Since $\omega(\pi)=1/2$, we have $\omega(\half+u)\neq \omega(\frac{\pi}{t}u)$. It implies that \begin{flalign*}
		   \omega(\lambda) =\min\tcbra{\omega(\half+u), \omega(\frac{\pi}{t}) } \leq \omega(\frac{\ol{\pi}}{t})
	\end{flalign*}
	Thus, we have $\omega(\lambda)=\omega(\frac{\ol{\pi}}{t})$. 
	
	(2) By construction, we have $\omega(\lambda)\geq \omega(\half)$. Write $\lambda=a+b\pi\in F$ for some $a,b\in F_0$. Then $\ol{\lambda}=a-b\pi$. Since $\lambda+\ol{\lambda}=1$, we have $a=\half$. As $\omega(\half)$ is even and $\omega(b\pi)$ is odd, they cannot be equal. We get \begin{flalign*}
		     \omega(\lambda) &=\omega\trbra{\half+b\pi }=\min\tcbra{\omega(\half), \omega(b\pi) } \leq \omega(\half).
	\end{flalign*}
	Thus, we have $\omega(\lambda)=\omega(\half)$. 
\end{proof}


Set \begin{flalign*}
	    \varepsilon\coloneqq \begin{cases}
	    	 t\quad &\text{in the (R-U) case},\\ 2 &\text{in the (R-P) case}.
	    \end{cases}
\end{flalign*} 
The hermitian form $h$ defines a symmetric $F_0$-bilinear form $s(-,-): V\times V\ra F_0$ and a quadratic form $q: V\ra F_0$ via \begin{flalign*}
	    s(x,y) \coloneqq  {\varepsilon}\inverse \Tr_{F/F_0} h(x,y) \text{\ and\ } 
	    q(x) \coloneqq  \half s(x,x), \text{\ for $x,y\in V$}.
\end{flalign*}  
Set $\sL\coloneqq \varepsilon\inverse \CO_{F_0}$, which is an invertible $\CO_{F_0}$-module. Then for $0\leq i\leq m$, we obtain induced forms \begin{flalign}
	     s : \Lambda_i\times \Lambda_i\lra \sL \text{\ and\ } q : \Lambda_i\lra \sL.  \label{sqform}
\end{flalign} 
It is straightforward to verify the following lemma.

\begin{lemma} \label{lem-hermm0}
   \begin{enumerate}
   	\item For $0\leq i\leq m$, the triple $(\Lambda_i,q,\sL)$ forms an $\sL$-valued hermitian quadratic module of rank $n$ over $\CO_{F_0}$ in the sense of Definition \ref{defn-hermqm}.
   	\item Define $$\phi: \Lambda_0\times\Lambda_0\ra \varepsilon\inverse\CO_{F_0},\ (x,y)\mapsto \varepsilon\inverse\Tr_{F/F_0}h(x,\pi\inverse y). $$ Then $(\Lambda_0,q,\sL,\phi)$ is a hermitian quadratic module with $\phi$.
   \end{enumerate}
\end{lemma}

Now we state two theorems on hermitian quadratic modules. The proofs will be given in the appendix.

\begin{thm}
	The functor $\ud{\Sim}(\Lambda_m)$ (resp. $\ud{\Sim}(\Lambda_0,\phi)$) is representable by an affine smooth group scheme over $\CO_{F_0}$ with generic fiber $\GU(V,h)$. Moreover, the scheme $\ud{\Sim}(\Lambda)$ (resp. $\ud{\Sim}(\Lambda_0,\phi)$) is isomorphic to the parahoric group scheme attached to $\Lambda_m$ (resp. $\Lambda_0$).
\end{thm}

\begin{thm}[{Theorem \ref{thmstand}, \ref{thmstand0}}]
	Let $R$ be an $\CO_{F_0}$-algebra. Let $(M,q,\sL)$ (resp. $(N,q,\sL,\phi)$) be a hermitian quadratic module over $R$ of rank $n$. Assume that $(M,q,\sL)$ (resp. $(N,q,\sL,\phi)$) is of type $\Lambda_m$ (\resp $\Lambda_0$) in the sense of Definition \ref{app-defn1} (\resp Definition \ref{app-defn2}). Then the hermitian quadratic module $(M,q,\sL)$ is \etale locally isomorphic to $(\Lambda_m,q,\varepsilon\inverse\CO_{F_0})\otimes_{\CO_{F_0}}R$ (\resp $(\Lambda_0,q,\varepsilon\inverse\CO_{F_0},\phi)\otimes_{\CO_{F_0}}R$) up to similitude. 
\end{thm}


\subsection{Naive local models} Let $I=\cbra{0}$ or $\cbra{m}$. Then $I$ corresponds to a special parahoric subgroup of $\GU(V,h)$. Let $\Lambda_I$ denote the corresponding lattice, which is either $\Lambda_0$ or $\Lambda_m$. 
Set $$\Lambda_I^h\coloneqq \cbra{x\in V\ |\ h(x,\Lambda_I)\subset\CO_F},\ \Lambda_I^s\coloneqq \cbra{x\in V\ |\ s(x,\Lambda_I)\subset \CO_{F_0}}.$$ 
The symmetric pairing $s$ on $V$ induces a perfect $\CO_{F_0}$-bilinear pairing \begin{flalign}
	   \Lambda_I\times \Lambda_I^s\ra \CO_{F_0}, \label{sym}
\end{flalign} which is still denotes by $s(-,-)$.  By Lemma \ref{lem-invdif}, one can check that \begin{flalign}
	   \Lambda^s =\begin{cases}
	   	     \Lambda^h\quad &\text{in the (R-U) case}, \\ \pi\inverse\Lambda^h &\text{in the (R-P) case}.
	   \end{cases} \label{eq-dualsh}
\end{flalign}
Note that \begin{flalign*}
	 \Lambda_0^h &=\CO_F\pair{\ol{\lambda}\inverse e_1,\ldots,\ol{\lambda}\inverse e_m, e_{m+1},e_{m+2},\ldots, e_n },\\ 
	 \Lambda_m^h &=\CO_F\pair{\ol{\lambda}\inverse e_1,\ldots,\ol{\lambda}\inverse e_m, e_{m+1},\ol{\pi} e_{m+2},\ldots,\ol{\pi} e_n }.
\end{flalign*} 
Using \eqref{eq-dualsh} and Lemma \ref{lem-lambda}, we have \begin{flalign*}
	    \Lambda^s_0\hookrightarrow\Lambda_0\hookrightarrow  \frac{\ol{\pi}}{t} \Lambda_0^s,\  \text{in the (R-U) case},\ \pi\Lambda^s_0\hookrightarrow\Lambda_0\hookrightarrow \frac{{\pi}}{2} \Lambda_0^s,\ \text{in the (R-P) case},
\end{flalign*}
and  \begin{flalign*}
	    \Lambda^s_m\hookrightarrow\Lambda_m\hookrightarrow  \frac{1}{t} \Lambda_m^s,\  \text{in the (R-U) case},\ \pi\Lambda^s_m\hookrightarrow\Lambda_m\hookrightarrow \frac{{1}}{2} \Lambda_m^s,\ \text{in the (R-P) case}.
\end{flalign*}
In summary, we have an inclusion of lattices  \begin{flalign*}
	     \Lambda_I\hookrightarrow \alpha\Lambda_I^s, \ \text{\ where\ } \alpha \coloneqq \begin{cases}
	     	{\ol{\pi}}/{\varepsilon}\quad &\text{if $I=\cbra{0}$},\\  {1}/{\varepsilon} &\text{if $I=\cbra{m}$}.
	     \end{cases}
\end{flalign*}

We define the \dfn{naive unitary local model} of type $I$ (and of signature $(n-1,1)$) as follows. 

\begin{defn} \label{def-naive}
	Let $\RM_I^\naive$ be the functor $$\RM_I^\naive: (\Sch/\CO_F)^\op \lra \Sets$$ which sends an $\CO_F$-scheme $S$ to the set of $\CO_S$-modules $\CF$ such that 
	\begin{enumerate}
		\item $\CF$ is an $\CO_F\otimes_{\CO_{F_0}}\CO_S$-submodule of $\Lambda_I\otimes_{\CO_{F_0}}\CO_S$ and as an $\CO_S$-module, it is a locally direct summand of rank $n$.
		\item (Kottwitz condition) The action of $\pi\otimes 1\in\CO_F\otimes_{\CO_{F_0}}\CO_S$ on $\CF$ has characteristic polynomial $$\det(T-\pi\otimes 1\ |\ \CF)=(T-\pi)(T-\ol{\pi})^{n-1}.$$
		\item Let $\CF^\perp$ be the orthogonal complement of $\CF$ in $\Lambda_I^s\otimes_{\CO_{F_0}}\CO_S$  with respect to the perfect pairing $$  (\Lambda_I\otimes_{\CO_{F_0}}\CO_S) \times (\Lambda_I^s\otimes_{\CO_{F_0}}\CO_S) \ra \CO_S$$ induced by \eqref{sym}. We require that the map $\Lambda_I\otimes_{\CO_{F_0}}\CO_S \ra \alpha \Lambda_I^s\otimes_{\CO_{F_0}}\CO_S$ induced by the inclusion $\Lambda_I\hookrightarrow \alpha\Lambda_I^s$ sends $\CF$ to $\alpha\CF^\perp,$ where $\alpha\CF^\perp$ denotes the image of $\CF^\perp$ under the isomorphism $\alpha:\Lambda_I^s\otimes_{\CO_{F_0}}\CO_S \simto \alpha\Lambda_I^s\otimes_{\CO_{F_0}}\CO_S$.
		\item $\CF$ is totally isotropic with respect to the form $(\Lambda_I\otimes_{\CO_{F_0}}\CO_S)\times (\Lambda_I\otimes_{\CO_{F_0}}\CO_S)\ra \sL\otimes_{\CO_{F_0}}\CO_S$ induced by $s$ in \eqref{sqform}, i.e., $s(\CF,\CF)=0$ in $\sL\otimes_{\CO_{F_0}}\CO_S$.
	\end{enumerate}
%
\end{defn}

\begin{lemma} \label{lem-generic}
	The functor $\RM_I^\naive$ is representable by a projective scheme over $\CO_F$ and the generic fiber is isomorphic to the $(n-1)$-dimensional projective space $\BP^{n-1}_F$ over $F$. 
\end{lemma}
\begin{proof}
	This is similar to \cite[1.5.3]{pappas2009local}. The representability follows by identifying $\RM_I^\naive$ with a closed subscheme of the Grassmannian $\Gr(n,\Lambda_I)_{\CO_F}$ classifying locally direct summands of rank $n$ in $\Lambda_I$. 
	
	As $\pi\otimes 1$ is a semisimple operator on $V\otimes_{F_0}F$, we have $$V\otimes_{F_0}F=V_\pi\oplus V_{\ol{\pi}},$$ where $V_\pi$ (resp. $V_{\ol{\pi}}$) denotes the $\pi$-eigenspace (resp. $\ol{\pi}$-eigenspace) of $\pi\otimes 1$. Both eigenspaces $V_\pi$ and $V_{\ol{\pi}}$ are $n$-dimensional $F$-vector spaces. We claim that $V_\pi$ is totally isotropic for the induced symmetric pairing, which is still denoted by $s(-,-)$, on $V\otimes_{F_0}F$. Indeed, for any $x,y\in V_\pi$, we have $(\pi\otimes 1)x=\pi x$ and $(\pi\otimes 1)y=\pi y$. Then \begin{flalign*}
		     s(x,y) &={\pi^{-2}}s(\pi x,\pi y)={\pi^{-2}}s\rbra{(\pi\otimes 1)x,(\pi\otimes 1)y } =(\pi_0/\pi^2)s(x,y).
	\end{flalign*} So $s(x,y)=0$. Similarly, we obtain that $V_{\ol{\pi}}$ is also totally isotropic. It implies that the induced pairing \begin{flalign}
		    s(-,-): V_\pi\times V_{\ol{\pi}}\ra F  \label{lemperf}
	\end{flalign}  is perfect.  
	
	Let $\BP^{n-1}_F$ be the projective space associated with $V_\pi$. For any $F$-algebra $R$, define \begin{flalign*}
		   \varphi: \RM_I^\naive(R) \lra \BP^{n-1}_F(R),\quad \CF \mapsto \ker(\pi\otimes 1-1\otimes\pi\ |\ \CF).
	\end{flalign*}
	By the Kottwitz condition for $\CF$, this is a well-defined map. Conversely, let $\CG\in\BP^{n-1}_F(R)$, i.e., $\CG$ is a direct summand of rank one of $V_\pi\otimes_{F}R$. The perfect pairing \eqref{lemperf} gives a (unique) direct summand $\CG'$ of rank $n-1$ of $V_{\ol{\pi}}\otimes_FR$ such that $s(\CG,\CG')=0$. Set $$\CF\coloneqq \CG\oplus\CG'\sset V\otimes_{F_0}R.$$ Then by our construction, we have $\CF\in \RM_I^\naive(R)$. This process defines an inverse map of $\varphi$. In particular, $\varphi$ is bijective, and hence the generic fiber of $\RM_I^\naive$ is isomorphic to $\BP^{n-1}_F$.
\end{proof}

Similar arguments as in \cite[Proposition 3.8]{pappas2000arithmetic} on the dimension of the special fiber of $\RM_I^\naive$ show that $\RM_I^\naive$ is not flat over $\CO_F$. 

\subsection{Local models} \label{subsec-localModels}
\begin{defn}
	The \dfn{local model} $\RM_I^\loc$ is defined to be the (flat) Zariski closure of the generic fiber of $\RM^\naive_I$ in $\RM^\naive_I$.  
\end{defn}
By construction, the scheme $\RM_I^\loc$ is a flat projective scheme of (relative) dimension $n-1$ over $\CO_F$. In the rest of the paper, we will prove Theorem \ref{thm-intromain}-\ref{intro-thm1.6} from the Introduction. The proof of Theorem \ref{thm-intromain} and \ref{thm16} will be divided into four cases, depending on the index set $I$ and the ramification types of $F/F_0$, see \S \ref{sec-0ru}-\ref{sec-mrp}. In the course of the proof, we also establish Theorem \ref{intro-thm1.6}. The proof of Theorem \ref{thm1.3} is given in \S \ref{sec-compa}.

\section{The case $I=\cbra{0}$ and (R-U)}\label{sec-0ru}
In this section, we will prove Theorem \ref{thm-intromain} in the case when $I=\cbra{0}$ and the quadratic extension $F/F_0$ is of (R-U) type. In particular, we have $$\pi^2-t\pi+\pi_0=0,$$ where $t\in \CO_{F_0}$ with $\pi_0|t\text{\ and\ }t|2$. Consider the following ordered $\CO_{F_0}$-basis of $\Lambda_0$ and $\Lambda_0^s$: \begin{alignat}{2}
	   \Lambda_0&: \frac{\ol{\pi}}{t}e_{m+2},\ldots,\frac{\ol{\pi}}{t}e_n,e_1,\ldots,e_m,e_{m+1},\frac{\pi_0}{t}e_{m+2},\ldots,\frac{\pi_0}{t}e_n,\pi e_1,\ldots,\pi e_m,\pi e_{m+1}, \label{basis-1}  \\ 
	   \Lambda_0^s&: e_{m+2},\ldots,e_n,\frac{t}{\pi}e_1,\ldots,\frac{t}{\pi}e_m,e_{m+1},\pi e_{m+2},\ldots,\pi e_n,t e_1,\ldots,te_m,\pi e_{m+1}. \label{basis-2}
\end{alignat}

\subsection{A refinement of $\RM_{\cbra{0}}^\naive$ in the (R-U) case}\label{subsec-refM1}
In this subsection, we will propose a refinement of the functor $\RM_{\cbra{0}}^\naive$. We first recall the ``strengthened spin condition" introduced by Smithling in \cite{smithling2015moduli}. 

\subsubsection{The strengthened spin condition} \label{subsubsec-spin}
Take $g_1,\ldots,g_{2n}$ to be the ordered $F$-basis 
\begin{flalign*}
	e_1\otimes 1-\pi e_1\otimes\pi\inverse,\ldots, e_n\otimes 1-\pi e_n\otimes\pi\inverse, \pi e_1\otimes  \frac{\pi}{t}-e_1\otimes\frac{\pi_0}{t},\ldots, \pi e_n\otimes\frac{\pi}{t} -e_n\otimes\frac{\pi_0}{t}
\end{flalign*}
of $V\otimes_{F_0}F$. Then with respect to the basis $(g_i)_{1\leq i\leq 2n}$, the symmetric pairing $s(-,-)\otimes_{F_0}F$ on $V\otimes_{F_0}F$ is represented by the $2n\times 2n$ matrix anti-diag($\theta,\ldots,\theta$). Recall $\theta=1-\frac{4\pi_0}{t^2}$. One can easily check that \begin{itemize}
	\item $(g_i)_{1\leq i\leq n}$ is a basis for $V_{\ol{\pi}}$ (the $\ol{\pi}$-eigenspace of the operator $\pi\otimes 1$ acting on $V\otimes_{F_0}F$),
	\item $(g_i)_{n+1\leq i\leq 2n}$ is a basis for $V_\pi$ (the $\pi$-eigenspace of the operator $\pi\otimes 1$ acting on $V\otimes_{F_0}F$).
\end{itemize}   
Take $f_1,\ldots,f_{2n}$ to be the ordered $\CO_F$-basis \begin{flalign*}
	   e_1\otimes 1,\ldots,e_{m+1}\otimes 1,  \frac{\ol{\pi}}{t}e_{m+2}\otimes 1,\ldots, \frac{\ol{\pi}}{t}e_n\otimes 1, \pi e_1\otimes 1,\ldots, \pi e_{m+1}\otimes 1, \frac{\pi_0}{t}e_{m+2}\otimes 1,\ldots, \frac{\pi_0}{t}e_n\otimes 1
\end{flalign*} of $\Lambda_0\otimes_{\CO_{F_0}}\CO_F$. This is the base change of the basis in \eqref{basis-1}, but in different order. We have 

 \begin{flalign}
	    (g_1,\ldots,g_{2n})=(f_1,\ldots,f_{2n})\begin{pmatrix}
	    	I_{m+1} &0 &-\frac{\pi_0}{t}I_{m+1} &0\\ 0 &\frac{t}{\pi}I_m &0 &-\pi I_m\\ -\frac{1}{\pi}I_{m+1} &0 &\frac{\pi}{t}I_{m+1} &0\\ 0 &-\frac{t}{\pi^2}I_m &0 &\frac{\pi^2}{{\pi_0}}I_m
	    \end{pmatrix}. \label{eq-egbasis}
\end{flalign}
As in \cite{smithling2015moduli}, we use the following convenient notations: \begin{itemize}
	\item For an integer $i$, we write $$i^\vee\coloneqq n+1-i,\quad i^*\coloneqq 2n+1-i.$$ For $S\sset\cbra{1,\ldots,2n}$ of cardinality $n$, we write $$S^*\coloneqq \cbra{i^*\ |\ i\in S},\quad S^\perp\coloneqq \cbra{1,\ldots,2n}\backslash S^*.$$
	    Let $\sigma_S$ be the permutation on $\cbra{1,\ldots,2n}$ sending $\cbra{1,\ldots,n}$ to $S$ in increasing order and sending $\cbra{n+1,\ldots,2n}$ to $\cbra{1,\ldots,2n}\backslash S$ in increasing order. Denote by $\sgn(\sigma_S)\in\cbra{\pm 1}$ the sign of $\sigma_S$. 
	\item Set $W\coloneqq \wedge^n(V\otimes_{F_0}F)$. For $S=\tcbra{i_1<\cdots<i_n} \sset\cbra{1,\ldots,2n}$ of cardinality $n$, we write $$e_S\coloneqq f_{i_1}\wedge\cdots\wedge f_{i_n}\in W, \text{\ similarly,\ }g_S\coloneqq g_{i_1}\wedge\cdots\wedge g_{i_n}\in W.$$
	    Note that $(e_S)_{\cbra{\# S=n}}$ (or $(g_S)_{\cbra{\# S=n}}$) is an $F$-basis of $W$.
	\item Set $$W_{\pm 1}\coloneqq \Span_F\cbra{g_S\pm\sgn(\sigma_S)g_{S^\perp}\ |\ S\sset\cbra{1,\ldots,2n} \text{\ and\ } \# S=n}.$$ This is a sub $F$-vector space of $W$. For any $\CO_F$-lattice $\Lambda$ in $V\otimes_{F_0}F$, set $$W(\Lambda)\coloneqq \wedge^n\rbra{\Lambda\otimes_{\CO_{F_0}}\CO_F },\ W(\Lambda)_{\pm 1}\coloneqq W_{\pm 1}\cap W(\Lambda).$$ Then $W(\Lambda)$ (resp. $W(\Lambda)_{\pm 1}$) is an $\CO_F$-lattice in $W$ (resp. $W_{\pm 1}$).	
	\item Set $$W^{n-1,1}\coloneqq \rbra{\wedge^{n-1}V_{\ol{\pi}}}\otimes_F(V_\pi),\ W^{n-1,1}_{\pm 1}\coloneqq W^{n-1,1}\cap W_{\pm 1},\ W(\Lambda)_{\pm 1}^{n-1,1}\coloneqq W_{\pm 1}^{n-1,1}\cap W(\Lambda). $$
\end{itemize} 
Then the strengthened spin condition states that 
\begin{enumerate}
	\item [] For any $\CO_F$-algebra $R$ and $\CF\in\RM^\naive_{\tcbra{0}}(R)$, the line $\wedge^n\CF\sset W(\Lambda_0)\otimes_{\CO_F}R$ is contained in $$\Im\rbra{W(\Lambda_0)^{n-1,1}_{- 1}\otimes_{\CO_F}R\ra W(\Lambda_0)\otimes_{\CO_F}R }.$$	  
\end{enumerate}

\subsubsection{The definition of the refinement}
\begin{defn}\label{defn-M}
	Let $\RM_{\cbra{0}}$ be the functor $$\RM_{\cbra{0}}: (\Sch/\CO_{F})^\op\lra \Sets$$ which sends an $\CO_F$-scheme $S$ to the set of $\CO_S$-modules $\CF$ such that 
	\begin{enumerate}[label=\textbf{LM\arabic*}]
		\item ($\pi$-stability condition) $\CF$ is an $\CO_F\otimes_{\CO_{F_0}}\CO_S$-submodule of $\Lambda_0\otimes_{\CO_{F_0}}\CO_S$ and as an $\CO_S$-module, it is a locally direct summand of rank $n$.
		\item (Kottwitz condition) The action of $\pi\otimes 1\in\CO_F\otimes_{\CO_{F_0}}\CO_S$ on $\CF$ has characteristic polynomial $$\det(T-\pi\otimes 1\ |\ \CF)=(T-\pi)(T-\ol{\pi})^{n-1}.$$
		\item Let $\CF^\perp$ be the orthogonal complement in $\Lambda_0^s\otimes_{\CO_{F_0}}\CO_S$ of $\CF$ with respect to the perfect pairing $$ s(-,-): (\Lambda_0\otimes_{\CO_{F_0}}\CO_S) \times (\Lambda_0^s\otimes_{\CO_{F_0}}\CO_S) \ra \CO_S.$$ We require the map $\Lambda_0\otimes_{\CO_{F_0}}\CO_S \ra (\frac{\ol{\pi}}{t} \Lambda_0^s)\otimes_{\CO_{F_0}}\CO_S$ induced by $\Lambda_0\hookrightarrow \frac{\ol{\pi}}{t}\Lambda_0^s$ sends $\CF$ to $\frac{\ol{\pi}}{t}\CF^\perp,$ where $\frac{\ol{\pi}}{t}\CF^\perp$ denotes the image of $\CF^\perp$ under the isomorphism $\frac{\ol{\pi}}{t}:\Lambda_0^s\otimes_{\CO_{F_0}}\CO_S \simto \frac{\ol{\pi}}{t}\Lambda_0^s\otimes_{\CO_{F_0}}\CO_S$.  \label{LMincl}
		\item (Hyperbolicity condition) The quadratic form $q: \Lambda_0\otimes_{\CO_{F_0}}\CO_S\ra \sL\otimes_{\CO_{F_0}}\CO_S$ induced by $q: \Lambda_0\ra \sL$ satisfies $q(\CF)=0$. 
		\item (Wedge condition) The action of $\pi\otimes 1-1\otimes\ol{\pi}\in \CO_F\otimes_{\CO_{F_0}}\CO_S$ satisfies $$\wedge^2(\pi\otimes 1-1\otimes \ol{\pi}\ |\ \CF )=0.$$
		\item (Strengthened spin condition) The line $\wedge^n\CF\sset W(\Lambda_0)\otimes_{\CO_F}\CO_S$ is contained in $$\Im\rbra{W(\Lambda_0)^{n-1,1}_{- 1}\otimes_{\CO_F}\CO_S\ra W(\Lambda_0)\otimes_{\CO_F}\CO_S }.$$	 
	\end{enumerate}
\end{defn} 

Then $\RM_{\cbra{0}}$ is representable by a projective $\CO_F$-scheme, which is a closed subscheme of $\RM^\naive_{\tcbra{0}}$. Note that over the generic fiber of $\RM_{\cbra{0}}$, the quadratic form $q$ is determined by $s$ via $q(x)=\half s(x,x)$. So, over the generic fiber, the hyperbolicity condition \textbf{LM4} is implied by the Condition (3) in $\RM_{\cbra{0}}^\naive$. Similarly as in \cite[1.5]{pappas2009local} and \cite[2.5]{smithling2015moduli}, we can deduce that the rest of the conditions of $\RM_{\cbra{0}}$ do not affect the generic fiber of $\RM^\naive_{\tcbra{0}}$, and hence $\RM_{\cbra{0}}$ and $\RM_{\cbra{0}}^\naive$ have the same generic fiber. 

Hence, we have closed immersions $$\RM^\loc_{\cbra{0}}\subset\RM_{\cbra{0}}\subset\RM^\naive_{\tcbra{0}}$$ of projective schemes over $\CO_F$, where all schemes have the same generic fiber.

\subsection{An affine chart $\RU_{\cbra{0}}$ around the worst point}\label{subsec4.2}
Set $$\CF_0\coloneqq (\pi\otimes 1)(\Lambda_0\otimes_{\CO_{F_0}}k).$$ Then we can check that $\CF_0\in\RM_{\cbra{0}}(k)$. We call it the \dfn{worst point} of $\RM_{\cbra{0}}$. 

With respect to the basis \eqref{basis-1}, the standard affine chart around $\CF_0$ in $\Gr(n,\Lambda_0)_{\CO_F}$ is the $\CO_F$-scheme of $2n\times n$ matrices $\left(\begin{smallmatrix}
	X\\ I_n
\end{smallmatrix}\right)$. We denote by $\RU_{\cbra{0}}$ the intersection of $\RM_{\cbra{0}}$ with the standard affine chart in $\Gr(n,\Lambda_0)_{\CO_F}$. The worst point $\CF_0$ of $\RM_{\cbra{0}}$ is contained in $\RU_{\cbra{0}}$ and corresponds to the closed point defined by $X=0$ and $\pi=0$. The conditions \textbf{LM1-6} yield the defining equations for $\RU_{\cbra{0}}$. We will analyze each condition in detail. A reader who is only interested in the affine coordinate ring of $\RU_{\cbra{0}}$ may proceed directly to Proposition \ref{prop-U0ring}.
In \S 4.3, we will use the group action on $\RM^\loc_{\tcbra{0}}$ to extend the geometric properties of $\RU_{\tcbra{0}}$ to the entire local model.   

\subsubsection{Condition \textbf{LM1}} \label{subsubsec-lm1}
Let $R$ be an $\CO_F$-algebra. With respect to the basis \eqref{basis-1}, the operator $\pi\otimes 1$ acts on $\Lambda_0\otimes_{\CO_{F_0}}R$ via the matrix \begin{flalign*}
	    \begin{pmatrix}
	    	0 &-\pi_0I_n\\ I_n &tI_n
	    \end{pmatrix}.
\end{flalign*}
Then the $\pi$-stability condition \textbf{LM1} on $\CF$ means there exists an $n\times n$ matrix $P\in M_{n\times n}(R)$ such that \begin{flalign*}
	   \begin{pmatrix}
	    	0 &-\pi_0I_n\\ I_n &tI_n
	    \end{pmatrix} \begin{pmatrix}
	    	X\\ I_n
	    \end{pmatrix} =\begin{pmatrix}
	    	X\\ I_n
	    \end{pmatrix}P.
\end{flalign*}
We obtain $P=X+tI_n$ and $X^2+tX+\pi_0I_n=0$. 

\subsubsection{Condition \textbf{LM2}} We have already shown that $\pi\otimes 1$ acts on $\CF$ via $X+tI_n$. Then the Kottwitz condition \textbf{LM2} translates to \begin{flalign*}
	   \det(T-(X+tI_n)) =(T-\pi)(T-\ol{\pi})^{n-1}.
\end{flalign*} Equivalently, \begin{flalign*}
	  \det(T-(X+\pi I_n)) =(T+\ol{\pi}-\pi)T^{n-1}.
\end{flalign*} Note that \begin{flalign*}
	  \det(T-(X+\pi I_n)) =\sum_{i=0}^n (-1)^i\tr(\wedge^i(X+\pi I_n)) T^{n-i}.
\end{flalign*}  Then by comparing the coefficients of $T^{n-i}$, the Kottwitz condition \textbf{LM2} becomes \begin{flalign}
	    \tr(X+\pi I_n)=\pi-\ol{\pi},\ \tr\rbra{\wedge^i(X+\pi I_n)}=0, \text{\ for $i\geq 2$}. \label{kott0ru}
\end{flalign} 

\subsubsection{Condition \textbf{LM3}} 

 With respect to the bases \eqref{basis-1} and \eqref{basis-2}, the perfect pairing $$s(-,-):(\Lambda_0\otimes_{\CO_{F_0}}R)\times (\Lambda_0^s\otimes_{\CO_{F_0}}R)\ra R  $$ and the map $\Lambda_0\otimes_{\CO_{F_0}}R\ra \frac{\ol{\pi}}{t}\Lambda_0^s\otimes_{\CO_{F_0}}R$ are represented respectively by the matrices \begin{flalign*}
	   S=\begin{pmatrix}
	   	  \frac{2}{t}H_{2m} &0 &H_{2m} &0 \\ 0 &\frac{2}{t} &0 &1\\ H_{2m} &0 &\frac{2\pi_0}{t}H_{2m} &0 \\ 0 &1 &0 &\frac{2\pi_0}{t}
	   \end{pmatrix} \text{\ and\ }  N=\begin{pmatrix}
	  	  I_m &0 &0 &0 &0 &0 \\ 0 &-I_m &0 &0 &-tI_m &0 \\ 0 &0 &0 &0 &0 &-t\\ 0 &0 &0 &I_m &0 &0 \\  0 &\frac{t}{\pi_0}I_m &0 &0 &\frac{t^2-\pi_0}{\pi_0}I_m &0 \\ 0 &0 &\frac{t}{\pi_0} &0 &0 &\frac{t^2}{\pi_0}
	  \end{pmatrix},
\end{flalign*} where $H_{2m}$ denotes the $2m\times 2m$ anti-diagonal unit matrix, and $I_m$ denotes the $m\times m$ identity matrix. 

Then the Condition \textbf{LM3} translates to $\matx^tS\rbra{N\matx }=0, \text{\ or equivalently,}$ 
 \begin{flalign}
	    \matx^t\begin{pmatrix}
		 0 &\frac{t^2-2\pi_0}{t\pi_0}H_m &0 &0 &\frac{t^2-3\pi_0}{\pi_0}H_m &0 \\ \frac{2}{t}H_m &0 &0 &H_m &0 &0 \\ 0 &0 &\frac{t}{\pi_0} &0 &0 &\frac{t^2-2\pi_0}{\pi_0}\\ 0 &H_m &0 &0 &\frac{t^2-2\pi_0}{t}H_m &0\\ H_m &0 &0 &\frac{2\pi_0}{t}H_m &0 &0 \\ 0 &0 &2 &0 &0 &t
	\end{pmatrix}\matx &=0. \label{eq-lm41} 
\end{flalign}

 Write $$X=\begin{pmatrix}
 	  A &B &E\\ C &D &F\\ G &H &x
 \end{pmatrix},$$ where $A,B,C,D\in M_{m\times m}(R)$, $E,F\in M_{m\times 1}(R)$, $G,H\in M_{1\times m}(R)$ and $x\in R$. Then Equation \eqref{eq-lm41} translates to 
 \setcounter{equation}{0}
 \renewcommand{\theequation}{\textbf{LM3}-\arabic{equation}}
 \begin{flalign}
 	 &\frac{2}{t}C^tH_mA+\frac{t^2-2\pi_0}{t\pi_0}A^tH_mC+\frac{t}{\pi_0}G^tG+H_mC+C^tH_m=0, \label{lm4-1} \\ 
 	 &\frac{2}{t}C^tH_mB+\frac{t^2-2\pi_0}{t\pi_0}A^tH_mD+\frac{t}{\pi_0}G^tH+H_mD+\frac{t^2-3\pi_0}{\pi_0}A^tH_m+\frac{t^2-2\pi_0}{t}H_m=0, \label{lm4-2} \\ 
 	 &\frac{2}{t}C^tH_mE+\frac{t^2-2\pi_0}{t\pi_0}A^tH_mF+\frac{t}{\pi_0}G^tx+H_mF+\frac{t^2-2\pi_0}{\pi_0}G^t=0, \label{lm4-3} \\ 
 	 &\frac{2}{t}D^tH_mA+\frac{t^2-2\pi_0}{t\pi_0}B^tH_mC+\frac{t}{\pi_0}H^tG+H_mA+D^tH_m+\frac{2\pi_0}{t}H_m=0, \label{lm4-4} \\ 
 	 &\frac{2}{t}D^tH_mB+\frac{t^2-2\pi_0}{t\pi_0}B^tH_mD+\frac{t}{\pi_0}H^tH+H_mB+\frac{t^2-3\pi_0}{\pi_0}B^tH_m=0, \label{lm4-5} \\ 
 	 &\frac{2}{t}D^tH_mE+\frac{t^2-2\pi_0}{t\pi_0}B^tH_mF+\frac{t}{\pi_0}xH^t+H_mE+\frac{t^2-2\pi_0}{\pi_0}H^t =0,  \label{lm4-6}\\ 
 	 &\frac{2}{t}F^tH_mA+\frac{t^2-2\pi_0}{t\pi_0}E^tH_mC+\frac{t}{\pi_0}xG+2G+F^tH_m=0, \label{lm4-7} \\
 	 &\frac{2}{t}F^tH_mB+\frac{t^2-2\pi_0}{t\pi_0}E^tH_mD+\frac{t}{\pi_0}xH+2H+\frac{t^2-3\pi_0}{\pi_0}E^tH_m=0, \label{lm4-8} \\ 
 	 &\frac{2}{t}F^tH_mE+\frac{t^2-\pi_0}{t\pi_0}E^tH_mF+\frac{t}{\pi_0}x^2+2x+\frac{t^2-2\pi_0}{\pi_0}x+t=0.  \label{lm4-9}
 \end{flalign}

 \setcounter{equation}{5}
 \renewcommand{\theequation}{\thesection.\arabic{equation}}
\subsubsection{Condition \textbf{LM4}}
Recall $\sL=t\inverse \CO_{F_0}$. With respect to the basis \eqref{basis-1}, the induced $(\sL\otimes_{\CO_{F_0}}R)$-valued symmetric pairing on $\Lambda_0\otimes_{\CO_{F_0}}R$ is represented by the matrix \begin{flalign*}
	  S_1=\begin{pmatrix}
	     0 &H_m &0 &0 &\frac{t^2-2\pi_0}{t}H_m &0\\ H_m &0 &0 &\frac{2\pi_0}{t}H_m &0 &0 \\ 0 &0 &2 &0 &0 &t\\ 0 &\frac{2\pi_0}{t}H_m &0 &0 &\pi_0H_m &0\\ \frac{t^2-2\pi_0}{t}H_m &0 &0 &\pi_0H_m &0 &0 \\ 0 &0 &t &0 &0 &2\pi_0
\end{pmatrix}. 
\end{flalign*} 
\textit{Convention}: Throughout the rest of the paper, we often encounter a matrix $M=(M_{ij})\in M_{\ell\times\ell}(R)$ whose diagonal entries are of the form $M_{ii}=2a_{ii}$ for some $a_{ii}\in R$. We then use $\half M_{ii}$ to denote $a_{ii}$. When we refer to ``half of the diagonal of $M$", we mean the row matrix consisting of the entries $\half M_{ii}$ for $1\leq i\leq\ell$.

The Condition \textbf{LM4} translates to \begin{gather*}
	 \matx^tS_1\matx=0 \text{\ and\ } \text{half of the diagonal of ${\matx}^t S_1\matx$ equals zero}. 
\end{gather*} 
One can check that the diagonal entries of $\smatx^tS_1\smatx$ are indeed divisible by $2$ in $R$. Equivalently, we obtain the following equations.
 \setcounter{equation}{0}
 \renewcommand{\theequation}{\textbf{LM4}-\arabic{equation}}
\begin{flalign}
	&C^tH_mA+A^tH_mC+2G^tG+\frac{2\pi_0}{t}H_mC+\frac{2\pi_0}{t}C^tH_m=0, \label{lm5-1} \\ &C^tH_mB+A^tH_mD+2G^tH+\frac{2\pi_0}{t}H_mD+\frac{t^2-2\pi_0}{t}A^tH_m+\pi_0H_m=0,\\ &C^tH_mE+A^tH_mF+2xG^t+\frac{2\pi_0}{t}H_mF+tG^t=0,\\ &D^tH_mA+B^tH_mC+2H^tG+\frac{t^2-2\pi_0}{t}H_mA+\frac{2\pi_0}{t}D^tH_m+\pi_0H_m=0,\\ &D^tH_mB+B^tH_mD+2H^tH+\frac{t^2-2\pi_0}{t}H_mB+\frac{t^2-2\pi_0}{t}B^tH_m=0,\\ &D^tH_mE+B^tH_mF+2xH^t+\frac{t^2-2\pi_0}{t}H_mE+tH^t=0,\\ &F^tH_mA+E^tH_mC+2xG+tG+\frac{2\pi_0}{t}F^tH_m=0,\\ &F^tH_mB+E^tH_mD+2xH+tH+\frac{t^2-2\pi_0}{t}E^tH_m=0,  \\ &F^tH_mE+E^tH_mF+2x^2+2tx+2\pi_0=0, \label{lm5-9} \\ &\text{half of the diagonal  of matrices in \textbf{LM4}-1,5,9 equals $0$. } \label{lm5-10}
\end{flalign}
 
\setcounter{equation}{5}
\renewcommand{\theequation}{\thesection.\arabic{equation}}

\subsubsection{Condition \textbf{LM5}}
We already know from \S \ref{subsubsec-lm1} that $\pi\otimes 1$ acts as right multiplication by $X+tI_n$ on $\CF$. Thus, the wedge condition \textbf{LM5} on $\CF$ translates to \begin{flalign*}
	   \wedge^2(X+\pi I_n)=0.
\end{flalign*}

\subsubsection{Condition \textbf{LM6}} \label{subsubsec-strspin}
We will use the same notations as in \S \ref{subsubsec-spin}. To find the equations induced by the strengthened spin condition \textbf{LM6} on $\CF$, we need to determine an $\CO_F$-basis of $W(\Lambda_0)^{n-1,1}_{-1}$.

\begin{defn}
   Let $S\sset \cbra{1,\ldots,2n}$ be a subset of cardinality $n$.
   \begin{enumerate}
   	\item We say $S$ is of type $(n-1,1)$ if $$\#(S\cap\cbra{1,\ldots,n})=n-1 \text{\ and\ } \#(S\cap\cbra{n+1,\ldots,2n})=1.$$ Such $S$ necessarily has the form $\tcbra{1,\ldots,\wh{j},\ldots,n,n+i}$ for some $i,j\in\cbra{1,\ldots,n}$.
   	\item Let $S$ be of type $(n-1,1)$. Denote by $i_S$ the unique element in $S\cap\cbra{n+1,\ldots,2n}$. Define $S\preccurlyeq S^\perp \text{\ if $i_S\leq i_{S^\perp}$}.$
   \end{enumerate}
\end{defn}
Set \begin{flalign*}
	&\CB \coloneqq  \cbra{S\subset\cbra{1,\ldots,2n}\ |\ \# S=n }, \quad \CB^{n-1,1}\coloneqq \cbra{S\in\CB\ |\text{\ $S$ is of type $(n-1,1)$}},  \\ &\CB_0 \coloneqq \cbra{S\in\CB^{n-1,1}\ |\text{\ $S\preccurlyeq S^\perp$} }.
\end{flalign*}
By construction, the $F$-vector space $W(\Lambda_0)_{-1}^{n-1,1}\otimes_{\CO_F}F$ equals $W^{n-1,1}_{-1}$, which is an $F$-subspace of $W$.
\begin{lemma}\label{lembasis}
	\begin{enumerate}
		\item The set $\cbra{e_S\ |\ S\in\CB} \rbra{\resp \cbra{g_S\ |\ S\in\CB}}$ is an $F$-basis of $W$.
		\item For $S\in\CB$, denote $$h_S\coloneqq g_S-\sgn(\sigma_S)g_{S^\perp}.$$ The set $\cbra{h_S\ |\ S\in\CB_0}$ is an $F$-basis of $W^{n-1,1}_{-1}$.
	\end{enumerate}
\end{lemma}
\begin{proof}
	(1) As $W=\wedge^n(V\otimes_{F_0}F)$ by definition, the statement is a standard fact about the wedge product of vector spaces.
	
	(2) By \cite[Lemma 4.2]{smithling2015moduli}, the $F$-space $W^{n-1,1}_{-1}$ is spanned by the set $\tcbra{h_S\ |\ S\in\CB^{n-1,1}}$. These $h_S$'s are not linearly independent over $F$. Indeed, for $S\in \CB^{n-1,1}$, we have $h_{S^\perp} =-\sgn(\sigma_S)h_{S}$ by using that $(S^\perp)^\perp=S$ and $\sgn(\sigma_S)=\sgn(\sigma_{S^\perp})$ (by \cite[Lemma 2.8]{smithling2015moduli}).  However, the set $\tcbra{h_S\ |\ S\in\CB_0}$ is $F$-linearly independent, since $\tcbra{g_S\ |\ S\in\CB}$ is $F$-linearly independent. So the set $\tcbra{h_S\ |\ S\in\CB_0}$ is an $F$-basis of $W_{-1}^{n-1,1}$. 
\end{proof}

\begin{defn}
	 Let $w=\sum_{S\in\CB}c_Se_S\in W$. The \dfn{worst term} of $w$ is defined to be $$WT(w)\coloneqq \sum_{S\in\CB(w)} c_Se_S,$$ where $\CB(w)\subset\CB$ consists of elements $S\in\CB$ such that $\omega(c_S)\leq \omega(c_T)$ for all $T\in \CB$.
\end{defn}

Recall $\sqrt{\theta}=1-{2\pi}/{t}\in\CO_F\cross$. Using \eqref{eq-egbasis}, we immediately obtain the following.
\begin{lemma}
    Let $S\in \CB^{n-1,1}$. Then exactly we have the following six cases. 
	\begin{enumerate}
		\item If $S=\tcbra{1,\ldots,\wh{i},\ldots,n,n+i}$ for some $i\leq m+1$, then $$WT(g_S)=(-1)^{i-1}\frac{t^{m-1}}{\pi^{3m-1}}e_{\tcbra{n+1,\ldots,2n}} .$$
		\item If $S=\tcbra{1,\ldots,\wh{i},\ldots,n,n+i}$ for some $m+1< i\leq n$, then $$WT(g_S)=(-1)^{i-1}\frac{t^{m-1}}{\pi^{3m-3}{\pi_0}}e_{\tcbra{n+1,\ldots,2n}} .$$
		\item If $S=\tcbra{1,\ldots,\wh{j},\ldots,n,n+i}$ for some $i,j\leq m+1$ with $i\neq j$, then $$WT(g_S)=-\sqrt{\theta}\frac{t^m}{\pi^{3m-1}}e_{\tcbra{i,n+1,\ldots,\wh{n+j},\ldots,2n}}.$$
		\item If $S=\tcbra{1,\ldots,\wh{j},\ldots,n,n+i}$ for some $i\leq m+1<j$, then $$ WT(g_S)=-\sqrt{\theta}\frac{t^{m-1}}{\pi^{3m-2}}e_{\tcbra{i,n+1,\ldots,\wh{n+j},\ldots,2n}}.$$
		\item If $S=\tcbra{1,\ldots,\wh{j},\ldots,n,n+i}$ for some $j\leq m+1<i$, then $$WT(g_S)=-\sqrt{\theta}\frac{t^{m+1}}{\pi^{3m-2}\pi_0}e_{\tcbra{i,n+1,\ldots,\wh{n+j},\ldots,2n}}.$$
		\item If $S=\tcbra{1,\ldots,\wh{j},\ldots,n,n+i}$ for some $i,j>m+1$ with $i\neq j$, then $$WT(g_S)=-\sqrt{\theta}\frac{t^m}{\pi^{3m-3}\pi_0} e_{\tcbra{i,n+1,\ldots,\wh{n+j},\ldots,2n}}.$$
	\end{enumerate}  
\end{lemma}

\begin{defn}
    For $S\in\CB^{n-1,1}$, the \dfn{weight vector} $\mathbf w_S\in\BZ^n$ attached to $S$ is defined to be an element of $\BZ^n$ such that the $i$-th coordinate of $\mathbf w_S$ is $\#(S\cap\cbra{i,n+i})$.
\end{defn}

Note that if $S\in\CB^{n-1,1}$, then $S=\tcbra{1,\ldots,\wh{j},\ldots,n,n+i}$ for some $1\leq i,j\leq n$. Moreover, we have $\sgn(\sigma_S)=(-1)^{i+j+1}$ (see \cite[Remark 4.3]{smithling2015moduli}) and $S^\perp=\tcbra{1,\ldots,\wh{i^\vee},\ldots,n,j^*}$. Similar arguments in \cite[Lemma 4.10]{smithling2015moduli} imply the following lemma.
\begin{lemma} \label{lem-wterm}
	Let $S\in\CB_0$. Then exactly we have the following nine cases.
	\begin{enumerate}
		\item $S=\tcbra{1,\ldots,\wh{m+1},\ldots,n,n+m+1}$. Then $S=S^\perp$, $\mathbf w_S=(1,\ldots,1)$, and $$WT(h_S)=WT(2g_S)=(-1)^m\frac{2t^{m-1}}{\pi^{3m-1}}e_{\tcbra{n+1,\ldots,2n}}.$$
		\item $S=\tcbra{1,\ldots,\wh{i^\vee},\ldots,n,n+i}$ for some $i<m+1$. Then $S=S^\perp$, $\mathbf w_S\neq (1,\ldots,1)$, and $$WT(h_S)=WT(2g_S)=-\sqrt{\theta}\frac{2t^{m-1}}{\pi^{3m-2}}e_{\tcbra{i,n+1,\ldots,\wh{i^*},\ldots,2n}}.$$
		\item $S=\tcbra{1,\ldots,\wh{i^\vee},\ldots,n,n+i}$ for some $i>m+1$. Then $S=S^\perp$, $\mathbf w_S\neq (1,\ldots,1)$, and $$WT(h_S)=WT(2g_S)=-\sqrt{\theta}\frac{2t^{m+1}}{\pi^{3m-2}\pi_0}e_{\tcbra{i,n+1,\ldots,\wh{i^*},\ldots,2n}}.$$
		\item $S=\tcbra{1,\ldots,\wh{i},\ldots,n,n+i}$ for some $i<m+1$. Then $S\neq S^\perp$, $\mathbf w_S=\mathbf w_{S^\perp}=(1,\ldots,1),$ and \begin{flalign*}
			   WT(h_S) &=WT(g_{\tcbra{1,\ldots,\wh{i},\ldots,n,n+i}}+g_{\tcbra{1,\ldots,\wh{i^\vee},\ldots,n,i^*}}) =(-1)^{i-1}\frac{t^m}{\pi^{3m-2}{\pi_0}}e_{\tcbra{n+1,\ldots,2n}}.
		\end{flalign*} 
		\item $S=\tcbra{1,\ldots,\wh{j},\ldots,n,n+i}$ for some $i<j^\vee<m+1$. Then $S\neq S^\perp$, $\mathbf w_S, \mathbf w_{S^\perp}$ and $(1,\ldots,1)$ are pairwise distinct and \begin{flalign*}
			   WT(h_S) &=WT(g_{\tcbra{1,\ldots,\wh{j},\ldots,n,n+i}}+(-1)^{i+j}g_{\tcbra{1,\ldots,\wh{i^\vee},\ldots,n,j^*}})\\ &=-\sqrt{\theta}\frac{t^{m-1}}{\pi^{3m-2}}e_{\tcbra{i,n+1,\ldots,\wh{n+j},\ldots,2n}}+(-1)^{i+j+1}\sqrt{\theta}\frac{t^{m-1}}{\pi^{3m-2}}e_{\tcbra{j^\vee,n+1,\ldots,\wh{i^*},\ldots,2n}} .
		\end{flalign*}
		\item $S=\tcbra{1,\ldots,\wh{m+1},\ldots,n,n+i}$ for some $i<m+1$. Then $S\neq S^\perp$, $\mathbf w_S, \mathbf w_{S^\perp}$ and $(1,\ldots,1)$ are pairwise distinct and \begin{flalign*}
			    WT(h_S) &= WT(g_{\tcbra{1,\ldots,\wh{m+1},\ldots,n,n+i}}-(-1)^{m+i}g_{\tcbra{1,\ldots,\wh{i^\vee},\ldots,n+m+1}})\\ &=(-1)^{m+i}\sqrt{\theta}\frac{t^{m-1}}{\pi^{3m-2}}e_{\tcbra{m+1,n+1,\ldots,\wh{i^*},\ldots,2n}}.
		\end{flalign*}
		\item $S=\tcbra{1,\ldots,\wh{j},\ldots,n,n+i}$ for some $i<m+1<j^\vee$. Then $S\neq S^\perp$, $\mathbf w_S, \mathbf w_{S^\perp}$ and $(1,\ldots,1)$ are pairwise distinct and \begin{flalign*}
			    WT(h_S) &=WT(g_{\tcbra{1,\ldots,\wh{j},\ldots,n,n+i}}-(-1)^{i+j+1}g_{\tcbra{1,\ldots,\wh{i^\vee},\ldots,n,j^*}})\\ &=-\sqrt{\theta}\frac{t^m}{\pi^{3m-1}}e_{\tcbra{i,n+1,\ldots,\wh{n+j},\ldots,2n}}-(-1)^{i+j}\sqrt{\theta}\frac{t^m}{\pi^{3m-3}\pi_0}e_{\tcbra{j^\vee,n+1,\ldots,\wh{i^*},\ldots,2n}}.
		\end{flalign*}
		\item $S=\tcbra{1,\ldots,\wh{j},\ldots,n,n+m+1}$ for some $j^\vee>m+1$. Then $S\neq S^\perp$, $\mathbf w_S,\mathbf w_{S^\perp}$ and $(1,\ldots,1)$ are pairwise distinct and \begin{flalign*}
			   WT(h_S) &= WT(g_{\tcbra{1,\ldots,\wh{j},\ldots,n,n+m+1}}-(-1)^{m+j+1}g_{\tcbra{1,\ldots,\wh{m+1},\ldots,n,j^*}})\\ &= -\sqrt{\theta}\frac{t^m}{\pi^{3m-1}}e_{\tcbra{m+1,n+1,\ldots,\wh{n+j},\ldots,2n}}.
		\end{flalign*}
		\item $S=\tcbra{1,\ldots,\wh{j},\ldots,n,n+i}$ for some $j^\vee>i>m+1$. Then $S\neq S^\perp$, $\mathbf w_S, \mathbf w_{S^\perp}$ and $(1,\ldots,1)$ are pairwise distinct and \begin{flalign*}
			    WT(h_S) &= WT(g_{\tcbra{1,\ldots,\wh{j},\ldots,n,n+i}}-(-1)^{i+j+1}g_{\tcbra{1,\ldots,\wh{i^\vee},\ldots,j^*}})\\ &=-\sqrt{\theta}\frac{t^{m+1}}{\pi^{3m-2}\pi_0}e_{\tcbra{i,n+1,\ldots,\wh{n+j},\ldots,2n}}+(-1)^{i+j+1}\sqrt{\theta}\frac{t^{m+1}}{\pi^{3m-2}\pi_0}e_{\tcbra{j^\vee,n+1,\ldots,\wh{i^*},\ldots,2n}}.
		\end{flalign*}
	\end{enumerate}
\end{lemma}

Let $w\in W_{-1}^{n-1,1}$. Recall that $\tcbra{h_S\ | S\in\CB_0}$ is an $F$-basis of $W_{-1}^{n-1,1}$ by Lemma \ref{lembasis}. Write $$w=\sum_{S\in\CB_0}a_Sh_S= \sum_{\mathbf w\in\BZ^n}\sum_{\substack{S\in\CB_0\\ \text{and $\mathbf w_S=\mathbf w$}} }a_Sh_S,\quad a_S\in F.$$
Then as in the proof of \cite[Proposition 4.12]{smithling2015moduli}, we have \begin{flalign*}
	w\in W(\Lambda_0)_{-1}^{n-1,1} \Longleftrightarrow \sum_{\substack{S\in\CB_0\\ \text{and $\mathbf w_S=\mathbf w$}}}a_Sh_S\in W(\Lambda_0)_{-1}^{n-1,1}, \text{\ for each $\mathbf w\in\BZ^n$}
\end{flalign*} 
We have two distinct situations for $\mathbf w$:
\begin{enumerate}[label=\textit{Case \arabic*:}]
	\item $\mathbf w\neq (1,\ldots,1)$. Then there exists at most one $S\in\CB_0$ such that $\mathbf w_S=\mathbf w$.
	\item $\mathbf w=(1,\ldots,1)$. Then $S$ is necessarily of the form $$S_i\coloneqq \cbra{1,\ldots,\wh{i},\ldots,n,n+i}$$ for some $1\leq i\leq m+1$. For any $1\leq i<m+1$, we have 
	    	\begin{flalign*}
	    		  h_{S_i} &= g_{S_i}+g_{S_{i^\vee}}\\ &=(-1)^ig_1\wedge\cdots\wedge\wh{g_i}\wedge\cdots\wedge\wh{g_{i^\vee}}\wedge\cdots\wedge g_n\wedge\rbra{g_i\wedge g_{i^*}-g_{i^\vee}\wedge g_{n+i} }\\ &= (-1)^ig_1\wedge\cdots\wedge\wh{g_i}\wedge\cdots\wedge\wh{g_{i^\vee}}\wedge\cdots\wedge g_n\\  &\quad \wedge\trbra{-tf_i\wedge f_{i^\vee}+\frac{t^2-\pi_0}{\pi_0} f_i\wedge f_{i^*}-2f_{i^\vee}\wedge f_{n+i}-\frac{t}{\pi_0}f_{n+i}\wedge f_{i^*} },
	    	   \end{flalign*}
	    	   and
	    	\begin{flalign*}
	    		   h_{S_{m+1}} &=2g_{S_{m+1}}\\ &=-2\cdot g_1\wedge\cdots\wedge \wh{g_i}\wedge\cdots\wedge \wh{g_{i^\vee}}\wedge\cdots\wedge g_n\wedge g_{n+m+1}\wedge \rbra{g_i\wedge g_{i^\vee} }\\ &= -2\cdot g_1\wedge\cdots\wedge \wh{g_i}\wedge\cdots\wedge \wh{g_{i^\vee}}\wedge\cdots\wedge g_n\wedge g_{n+m+1} \\ &\quad \wedge\trbra{\frac{t}{{\pi}}f_i\wedge f_{i^\vee}-\frac{t}{\pi^2}f_i\wedge f_{i^*}+\frac{t}{\pi^2}f_{i^\vee}\wedge f_{n+i}+\frac{t}{\pi^3}f_{n+i}\wedge f_{i^*} }.
	    	\end{flalign*}	    
	Define 
	     \begin{flalign*}
		    \wt{h}_{S_i}\coloneqq \begin{cases}
		    	  2\ol{\pi} h_{S_{i}}+(-1)^{m+i}th_{S_{m+1}}\quad &\text{if $i\neq m+1$},\\ h_{S_{m+1}} &\text{if $i=m+1$}.
		    \end{cases}
	    \end{flalign*}
	    Then for $1\leq i< m+1$, terms of $\wt{h}_{S_i}$ do not contain (multiples of) $$WT(h_{S_{m+1}})=(-1)^m\frac{2t^{m-1}}{\pi^{3m-1}}e_{\tcbra{n+1,\ldots,2n}},$$ and  \begin{flalign}
	  WT(\wt{h}_{S_i}) &=-\sqrt{\theta}\frac{2t^m\pi_0}{\pi^{3m}}e_{\tcbra{i,n+1,\ldots,\wh{n+i},\ldots,2n}}-\sqrt{\theta}\frac{2t^m}{\pi^{3m-2}}e_{\tcbra{i^\vee,n+1,\ldots,\wh{i^*},\ldots,2n}}.  \label{eq-newSi}
       \end{flalign} 
\end{enumerate}
For $S$ with $\mathbf w_S\neq (1,\ldots,1)$, we set $\wt{h}_S\coloneqq h_S$. By Lemma \ref{lembasis}, the set $\tcbra{\wt{h}_S\ |\ S\in\CB_0}$ forms an $F$-basis of $W_{-1}^{n-1,1}$. Previous analysis on $\mathbf w$ together with similar arguments in \cite[Proposition 4.12]{smithling2015moduli} imply the following lemma. 

\begin{lemma}\label{lem-Wbasis}
    For each $S\in\CB_0$, pick $b_S\in F$ such that the worst term $WT(b_S\wt{h}_S)$ is a sum of terms of the form $u_Te_T$ for some unit $u_T\in\CO_F\cross$ and $T\in \CB$. Then the set $\tcbra{b_S\wt{h}_S\ |\ S\in\CB_0}$ forms an $\CO_F$-basis of the $\CO_F$-module $W(\Lambda_0)_{-1}^{n-1,1}$.
\end{lemma}


For the matrix $\smatx$ corresponding to $\CF$, denote by $v\in\wedge^n\CF$ the wedge product of $n$-columns of the matrix in the order from left to right. Then the strengthened spin condition \textbf{LM6} on $\CF$ amounts to that $$v\in \Im\rbra{W(\Lambda_0)^{n-1,1}_{- 1}\otimes_{\CO_F}R\ra W(\Lambda_0)\otimes_{\CO_F}R }.$$

Write $v=\sum_{S\in\CB}a_Se_S$ for some $a_S\in R$. By Lemma \ref{lem-Wbasis}, we have \begin{flalign}
	  v=\sum_{S\in\CB}a_Se_S=\sum_{S\in\CB_0} c_Sb_S\wt{h}_S \label{eq-vwedge}
\end{flalign}  for some $c_S\in R$. By comparing the coefficients of both sides in Equation \eqref{eq-vwedge}, we will obtain the defining equations of the condition \textbf{LM6} on the chart $\RU_{\cbra{0}}$.

Recall $$X=\begin{pmatrix}
 	  A &B &E\\ C &D &F\\ G &H &x
 \end{pmatrix},$$ where $A,B,C,D\in M_{m\times m}(R)$, $E,F\in M_{m\times 1}(R)$, $G,H\in M_{1\times m}(R)$ and $x\in R$. In the following, we use $a_{ij}$ to denote the $(i,j)$-entry of the matrix $A$. We use similar notations for other block matrices in $X$. 
For $1\leq i<m+1$, comparing the coefficients of $e_{\tcbra{n+1,\ldots,2n}}$ and $e_{S_i}=e_{\tcbra{1,\ldots,\wh{i},\ldots,n,n+i }}$ in \eqref{eq-vwedge}, we obtain \begin{flalign*}
	   c_{S_{m+1}}(-1)^m b_{S_{m+1}} \frac{2t^{m-1}}{\pi^{3m-1}} &=1,\\ c_{S_{m+1}}b_{S_{m+1}}(-1)^{m+i}\frac{2t^{m-1}}{\pi^{3m-2}}+c_{S_i}b_{S_i}\rbra{-\sqrt{\theta}\frac{2t^m\pi_0}{\pi^{3m}} } &=(-1)^{1+i}d_{ii},\\ c_{S_{m+1}}b_{S_{m+1}}(-1)^{m+i}\frac{2t^{m-1}}{\pi^{3m-2}}+c_{S_i}b_{S_i}\rbra{-\sqrt{\theta}\frac{2t^m}{\pi^{3m-2}} } &=(-1)^{1+i}a_{m+i-i,m+1-i}.
\end{flalign*}
Hence, \begin{flalign}
	  d_{ii}=\frac{\pi_0}{\pi^2}a_{m+1-i,m+1-i}+t\sqrt{\theta}. \label{eq-AD1}
\end{flalign}
For $1\leq i,j<m+1$ and $i\neq j$, comparing the coefficients of $e_{\tcbra{1,\ldots,\wh{j},\ldots,n,n+i}}$ and $e_{\tcbra{j^\vee,n+1,\ldots,\wh{i^*},\ldots,2n } } $, we obtain \begin{flalign*}
	   c_{\cbra{1,\ldots,\wh{j},\ldots,n,n+i} }b_{\cbra{1,\ldots,\wh{j},\ldots,n,n+i} }\rbra{-\sqrt{\theta}\frac{t^m}{\pi^{3m-1}} } &=(-1)^{1+j}d_{ij},\\ c_{\cbra{1,\ldots,\wh{j},\ldots,n,n+i} }b_{\cbra{1,\ldots,\wh{j},\ldots,n,n+i} }\rbra{(-1)^{1+i+j}\sqrt{\theta}\frac{t^m}{\pi^{3m-3}\pi_0} } &= (-1)^{1+i}a_{m+1-j,m+1-i}.
\end{flalign*}
Hence, \begin{flalign}
	d_{ij}=\frac{\pi_0}{\pi^2}a_{m+1-j,m+1-i}.  \label{eq-AD2}
\end{flalign} 
Combining \eqref{eq-AD1} and \eqref{eq-AD2}, we obtain \begin{flalign*}
	  D=\frac{\pi_0}{\pi^2}H_mA^tH_m +t\sqrt{\theta}I_m.  
\end{flalign*} Here the matrix $H_mA^tH_m$ is the reflection of $A$ over its anti-diagonal. Equivalently, \begin{flalign}
	   D+\pi I_m=\frac{\ol{\pi}}{\pi}H_m(A+\pi I_m)^tH_m.  \label{eq-AD}
\end{flalign} Similarly, we can obtain \begin{flalign}
	   B=H_mB^tH_m,\ C=H_mC^tH_m,\  E=\frac{t}{\ol{\pi}}H_mH^t,\ F=\frac{t}{\pi}H_mG^t. \label{eq-BCEF}
\end{flalign}

Write \begin{flalign*}
	\wt{H}_{2m}\coloneqq  \begin{pmatrix}
	      0 &H_m\\ \frac{\ol{\pi}}{\pi}H_m &0
	 \end{pmatrix},\ X_1\coloneqq  \begin{pmatrix}
	   	    A &B\\ C &D
	   \end{pmatrix}.
\end{flalign*}  Combining \eqref{eq-AD} and \eqref{eq-BCEF}, we have \begin{flalign}
	   \wt{H}_{2m}\rbra{X_1+\pi I_{2m} }=(X_1+\pi I_{2m})^t\wt{H}_{2m}^t.  \label{eq-X1}
\end{flalign}

\subsubsection{A simplification of equations}
First we can see under the wedge condition $\wedge^2(X+\pi I_n)=0$, the Kottwitz condition  \eqref{kott0ru} becomes \begin{flalign}
	   \tr(X+\pi I_n) =\pi-\ol{\pi}.   \label{trXpi}
\end{flalign}
Next we claim that the equation \begin{flalign}
	   X^2+tX+\pi_0I_n=0 \label{X2=0}
\end{flalign} of Condition \textbf{LM1} is implied by the Kottwitz condition \textbf{LM2} and the wedge condition \textbf{LM5}. To justify the claim, we need an easy but useful lemma. 
\begin{lemma}\label{lem-2wedge}
	Let $X$ be an $n\times n$ matrix. Then $X^2\equiv (\tr X)X$ modulo $(\wedge^2X)$.  
\end{lemma}
\begin{proof}
	The $(i,j)$-entry of the matrix $X^2-{\tr(X)}X$ is \begin{flalign*}
	    &\sum_{k=1}^nX_{ik}X_{kj}- \sum_{k=1}^nX_{kk}X_{ij} =\sum_{k=1}^n\rbra{X_{ik}X_{kj}- X_{kk}X_{ij} },
   \end{flalign*} which is a sum of $2$-minors of $X$.
\end{proof}

By Lemma \ref{lem-2wedge} and the wedge condition \textbf{LM5},  the equation \eqref{X2=0} \begin{flalign*}
	  X^2+tX+\pi_0 I_n = (X+\pi I_n)^2+(t-2\pi)(X+\pi I_n)=0
\end{flalign*}  is equivalent to \begin{flalign*}
	  \tr(X+\pi I_n) (X+\pi I_n)+(t-2\pi)(X+\pi I_n) =\rbra{\tr(X+\pi I_n)+\ol{\pi}-{\pi} }(X+\pi I_n)=0,
\end{flalign*} which is implied by the Kottwitz condition \eqref{trXpi}.

Next we look at the Condition \textbf{LM3}. For the equation \eqref{lm4-1}, we have \begin{flalign*}
	    &\frac{2}{t}C^tH_mA+\frac{t^2-2\pi_0}{t\pi_0}A^tH_mC+\frac{t}{\pi_0}G^tG+H_mC+C^tH_m \\ = &\frac{2}{t}C^tH_m(A+\pi I_m)+\frac{t^2-2\pi_0}{t\pi_0}(A+\pi I_m)^tH_mC-\frac{2\pi}{t}C^tH_m-\frac{t^2-2\pi_0}{t\pi_0}\pi H_mC\\ &\quad +\frac{t}{\pi_0}G^tG+H_mC+C^tH_m\\ =& \frac{2}{t}C^tH_m(A+\pi I_m)+\frac{t^2-2\pi_0}{t\pi_0}(A+\pi I_m)^tH_mC +\frac{t}{\pi_0}G^tG+\sqrt{\theta}C^tH_m+\frac{\pi}{\ol{\pi}}\sqrt{\theta}H_mC.
\end{flalign*} 
A similar argument as in the proof of Lemma \ref{lem-2wedge} implies that \begin{flalign*}
	 C^tH_m(A+\pi I_m)\equiv (A+\pi I_m)^tH_mC \text{\ modulo ($\wedge^2(X+\pi I_m)$). }
\end{flalign*} Hence, the equation \eqref{lm4-1} gives the same restriction on $\RU_{\cbra{0}}$ as the equation \begin{flalign*}
	    \frac{t}{\pi_0}(A+\pi I_m)^tH_mC+\frac{t}{\pi_0}G^tG+\sqrt{\theta}C^tH_m+\frac{\pi}{\ol{\pi}}\sqrt{\theta}H_mC=0.
\end{flalign*}
By \eqref{eq-BCEF}, we further obtain \begin{gather}
	\frac{t}{\pi_0}(A+\pi I_m)^tH_mC+\frac{t}{\pi_0}G^tG+\frac{t}{\ol{\pi}} \sqrt{\theta}H_mC=0, \label{eq-AC} \\ (A+\pi I_m)^tH_mC=(C(A+\pi I_m))^tH_m. \notag
\end{gather}
Again as in Lemma \ref{lem-2wedge}, the matrix $C(A+\pi I_m)$ is equivalent to $\tr(A+\pi I_m) C$. Thus, the equation \eqref{eq-AC} is equivalent to \begin{flalign*}
	    \frac{t}{\pi_0}\tr(A+\pi I_m) C^tH_m +\frac{t}{\pi_0}G^tG+\frac{t}{\ol{\pi}}\sqrt{\theta}H_mC=0.
\end{flalign*}
Equivalently, \begin{flalign}
	\frac{t}{\pi_0}\rbra{(\tr(A+\pi I_m)+\pi\sqrt{\theta})H_mC +G^tG}=0.  \label{lm7-1}
\end{flalign}

Similarly, under the wedge condition \textbf{LM5} and the strengthened spin condition \textbf{LM6}, one can verify that the   equation \eqref{lm4-2} can be simplified to \begin{flalign}
	   \frac{t}{\pi_0}\rbra{(\tr(A+\pi I_m)+\pi\sqrt{\theta})H_m(D+\pi I_m) +G^tH}=0; \label{lm7-2} 
\end{flalign}
the equation \eqref{lm4-3} is trivial; the equation \eqref{lm4-4} is equivalent to \eqref{lm4-2};  the equation \eqref{lm4-5} is equivalent to \begin{flalign}
	   \frac{t}{\pi_0}\rbra{(\frac{\ol{\pi}}{\pi} \tr(A+\pi I_m)+\pi\sqrt{\theta})H_mB +H^tH}=0;  \label{lm7-3}
\end{flalign} the rest of the equations are trivial. 

Set $$X_1\coloneqq \begin{pmatrix}
	A &B\\ C &D
\end{pmatrix}, X_2\coloneqq  \begin{pmatrix}
	E \\ F
\end{pmatrix}, X_3\coloneqq \begin{pmatrix}
	G &H
\end{pmatrix}, X_4\coloneqq x. $$ 
Then $X=\begin{psmallmatrix}
	X_1 &X_2\\ X_3 &X_4
\end{psmallmatrix}$, and equations \eqref{lm7-1}, \eqref{lm7-2}, \eqref{lm7-3} translate to \begin{flalign*}
	    \frac{t}{\pi_0}\rbra{\trbra{\tr(A+\pi I_m)+\pi\sqrt{\theta} }\wt{H}_{2m}(X_1+\pi I_{2m})+X_3^tX_3}=0.
\end{flalign*}

Using similar arguments, one can check that under the wedge condition \textbf{LM5} and the strengthened spin condition \textbf{LM6}, equations \eqref{lm5-1} to \eqref{lm5-9} are implied by the Condition \textbf{LM3}, and the equation \eqref{lm5-10} is equivalent to \begin{flalign*}
	    \text{the diagonal of $\trbra{\tr(A+\pi I_m)+\pi\sqrt{\theta} }\wt{H}_{2m}(X_1+\pi I_{2m})+X_3^tX_3$ equals $0$. }
\end{flalign*} 

Denote by $\CO_F[X]$ the polynomial ring over $\CO_F$ whose variables are entries of the matrix $X$. Then we can view the affine chart $\RU_{\cbra{0}}\subset\RM_{\cbra{0}}$ as a closed subscheme of $\Spec \CO_F[X]$. In summary, we have shown the following.
\begin{prop}\label{prop-U0ring}
	The scheme $\RU_{\cbra{0}}$ is a closed subscheme \footnote{In fact, we expect that $\RU_{\cbra{0}}=\RU_{\cbra{0}}'$. This amounts to saying that the equations obtained by comparing coefficients of $e_S$ in \eqref{eq-vwedge} for $S$ not of type $(n-1,1)$ are implied by relations in $\CI$.  } of $\RU_{\cbra{0}}'\coloneqq \Spec\CO_F[X]/\CI$, where $\CI$ is the ideal of $\CO_F[X]$ generated by: \begin{flalign*}
	    &\tr(X+\pi I_n)-\pi+\ol{\pi}, \ 
	    \wedge^2(X+\pi I_n),\
	    \wt{H}_{2m}\rbra{X_1+\pi I_{2m} }-(X_1+\pi I_{2m})^t\wt{H}_{2m}^t ,\\
	    &E-\frac{t}{\ol{\pi}}H_mH^t,\ F-\frac{t}{\pi}H_mG^t, \
	    \frac{t}{\pi_0}\rbra{\trbra{\tr(A+\pi I_m)+\pi\sqrt{\theta} }\wt{H}_{2m}(X_1+\pi I_{2m})+X_3^tX_3}, \\
	    &\text{the diagonal of $\trbra{\tr(A+\pi I_m)+\pi\sqrt{\theta} }\wt{H}_{2m}(X_1+\pi I_{2m})+X_3^tX_3$. }
\end{flalign*} 
\end{prop}

Set $$ \wt{X}_1\coloneqq X_1+\pi I_{2m},\ \wt{A}\coloneqq A+\pi I_m,\ \wt{X}\coloneqq \begin{pmatrix}
	  \wt{X}_1 \\ X_3
\end{pmatrix}.$$ 
 As $X_2$ and $X_4$ are determined by $X_1$ and $X_3$ by relations in $\CI$, we obtain the following proposition.
\begin{prop}
	The scheme $\RU'_{\cbra{0}}=\Spec\CO_F[X]/\CI$ is isomorphic to $\Spec\CO_F[\wt{X}]/\wt{\CI}$, where $\wt{\CI}$ is the ideal of $\CO_F[\wt{X}]$ generated by: \begin{flalign*}
	    &\wedge^2(\wt{X}),\
	    \wt{H}_{2m}\wt{X}_1-\wt{X}_1^t\wt{H}_{2m}^t,\
	    \frac{t}{\pi_0}\rbra{({\tr(\wt{A})+\pi\sqrt{\theta}) }\wt{H}_{2m}\wt{X}_1 +X_3^tX_3}, \\
	    &\text{the diagonal of $\trbra{\tr(\wt{A})+\pi\sqrt{\theta} }\wt{H}_{2m}\wt{X}_1 +X_3^tX_3$. }
\end{flalign*}
\end{prop}

\begin{defn}
	Denote by $\RU^\rfl_{\cbra{0}}$ the closed subscheme of $\RU'_{\cbra{0}}=\Spec\CO_F[\wt{X}]/\CI$ defined by the ideal $\wt{\CI}^\rfl\subset\CO_F[\wt{X}]$ that is generated by: 
	\begin{flalign*}
	    &\wedge^2(\wt{X}),\  
	    \wt{H}_{2m}\wt{X}_1-\wt{X}_1^t\wt{H}_{2m}^t ,\   
	    {({\tr(\wt{A})+\pi\sqrt{\theta}) }\wt{H}_{2m}\wt{X}_1 +X_3^tX_3}. 
\end{flalign*}
Note that the ideal $\wt{\CI}^\rfl$ contains $\wt{\CI}$.
\end{defn}

\subsubsection{Geometric properties of $\RU_{\cbra{0}}$ and $\RU^\rfl_{\cbra{0}}$}

In the following, we write $\CR^\rfl$ for the ring $\CO_F[\wt{X}]/\wt{\CI}^\rfl$ and $\CR$ for the ring $\CO_F[\wt{X}]/\wt{\CI}$. 
\begin{lemma} \label{lem-moduli}
	If $\omega(\pi_0)=\omega(t)$, then $\CR=\CR^\rfl$. 
\end{lemma}
\begin{proof}
	Note that $\omega(\pi_0)=\omega(t)$ if and only if $t/\pi_0$ is a unit in $\CO_F$. By comparing the lists of generators of $\wt{\CI}$ and $\wt{\CI}^\rfl$, we immediately see that $\wt{\CI}=\wt{\CI}^\rfl$, and hence $\CR=\CR^\rfl$. 
\end{proof}
\begin{remark} \label{rmk-t=pi0}
	Since $\pi_0|t\text{\ and\ }t|2$, the condition $\omega(t)=\omega(\pi_0)$ clearly holds if $F_0/\BQ_2$ is unramified. More generally, by applying Proposition \ref{quad} (4) to $F_0$, we have $\omega(t)=\omega(\pi_0)$ if and only if $\theta\in U_{2e-1}- U_{2e}$. Namely, given a quadratic extension $F$ of $F_0$ with a uniformizer $\pi$ satisfying an Eisenstein equation $\pi^2-t\pi+\pi_0=0$, the condition $\omega(t)=\omega(\pi_0)$ holds if and only if $F$ is of the form $F_0(\sqrt{\theta})$ for some unit $\theta\in U_{2e-1}-U_{2e}$. We will count the number of such extensions $F$ in the following. 
	
	 We have a short exact sequence \begin{flalign}
		    0\ra \frac{U_{2e}}{U^2\cap U_{2e}} \ra \frac{U_{2e-1}}{U^2\cap U_{2e-1}}\ra \frac{U_{2e-1}}{U_{2e}(U^2\cap U_{2e-1})}\ra 0.  \label{eq-ses}
	\end{flalign}
	We claim that $U^2\cap U_{2e-1}\sset U_{2e}$. For any $x\in U^2\cap U_{2e-1}$, we can find $a\in \CO_{F_0}$ and $u\in U$ such that $x=1+\pi_0^{2e-1}a=u^2$. We want to show $\omega(a)\geq 1$. Set $b=u-1$. Then $b(b+2)=\pi_0^{2e-1}a$. If $\omega(b)< e=\omega(2)$, then $\omega(b+2)=\omega(b)$ and \begin{flalign*}
		     \omega(\pi_0^{2e-1}a)= \omega\rbra{b(b+2)} =2\omega(b).
	\end{flalign*}
	As $2e-1$ is odd, this forces $\omega(a)$ to be odd and in particular $\omega(a)\geq 1$. If $\omega(b)\geq e$, then  \begin{flalign*}
		    \omega(\pi_0^{2e-1}a)=\omega\rbra{b(b+2)} \geq \omega(b)+\omega(2) \geq 2e.
	\end{flalign*}
	Again we have $\omega(a)\geq 1$. This proves the claim.
	
	Then we have $U_{2e}(U^2\cap U_{2e-1})=U_{2e}$ and by the short exact sequence \eqref{eq-ses}, \begin{flalign*}
		   \left\lvert\frac{U_{2e-1}}{U^2\cap U_{2e-1}}\right\rvert &= \left\lvert\frac{U_{2e}}{U^2\cap U_{2e}}\right\rvert \left\lvert\frac{U_{2e-1}}{U_{2e}}\right\rvert =2\cdot 2^f=2^{1+f},
	\end{flalign*}
	where $f$ denotes the residue degree of $F_0/\BQ_2$. Note that there are two elements in $\frac{U_{2e-1}}{U^2\cap U_{2e-1}}$ defining the trivial extension and the unramified quadratic extension of $F_0$. Thus, we have $2^{1+f}-2$ ramified quadratic extensions of $F_0$ of type (R-U) with $\omega(t)=\omega(\pi_0)$. 
\end{remark}

By \eqref{eq-AD}, we have $$\tr(\wt{X}_1)=\tr(\wt{A})+\tr(\wt{D})=\frac{t}{\pi}\tr(\wt{A}). $$ So we can rewrite $\CR^\rfl$ as \begin{flalign*}
	    \CR^\rfl &= \frac{\CO_F[\begin{psmallmatrix}
	     	\wt{X}_1\\ X_3
	     \end{psmallmatrix} ]}{\rbra{\wedge^2\begin{psmallmatrix}
	     	\wt{X}_1\\ X_3
	     \end{psmallmatrix}, \wt{H}_{2m}\wt{X}_1-\wt{X}_1^t\wt{H}_{2m}^t, (\frac{\pi}{t}\tr(\wt{X}_1)+\pi\sqrt{\theta})\wt{H}_{2m}\wt{X}_1+X_3^tX_3   }}.
\end{flalign*}
Let $Y\coloneqq \wt{H}_{2m}\wt{X}_1$. Then $\wt{X}_1=\frac{\pi}{\ol{\pi}}\wt{H}_{2m}Y$ and \begin{flalign*}
	  \CR^\rfl &\simeq \frac{\CO_F[\begin{psmallmatrix}
	     	Y\\ X_3
	     \end{psmallmatrix} ]}{\rbra{\wedge^2\begin{psmallmatrix}
	     	\frac{\pi}{\ol{\pi}}\wt{H}_{2m}Y \\ X_3
	     \end{psmallmatrix}, Y-Y^t, (\frac{\pi^2}{t\ol{\pi}}\tr(\wt{H}_{2m}Y)+\pi\sqrt{\theta})Y+X_3^tX_3   }} \\ &= \frac{\CO_F[\begin{psmallmatrix}
	     	Y\\ X_3
	     \end{psmallmatrix} ]}{\rbra{\wedge^2\begin{psmallmatrix}
	     	Y \\ X_3
	     \end{psmallmatrix}, Y-Y^t, (\frac{\pi}{2\ol{\pi}}\tr({H}_{2m}Y)+\pi\sqrt{\theta})Y+X_3^tX_3   }}.
\end{flalign*}
For $1\leq i,j\leq 2m$, we denote by $y_{ij}$ the $(i,j)$-entry of $Y$ and by $x_i$ the $(1,i)$-entry of $X_3$.

\begin{lemma}\label{lem-smoothoutside}
	The scheme $\RU^\rfl_{\cbra{0}}$ is irreducible of Krull dimension $n$ and smooth over $\CO_F$ on the complement of the worst point, which is the closed point defined by $Y=X_3=\pi=0$. 
\end{lemma}
\begin{proof}
	 For $1\leq \ell\leq 2m$, consider the principal open subscheme $D(y_{\ell\ell})$ of $\RU^\rfl_{\cbra{0}}$, i.e., the locus where $y_{\ell\ell}$ is invertible. Then one can easily verify that $D(y_{\ell\ell})$ is isomorphic to the closed subscheme of $$\Spec \CO_F[y_{ij},x_i\ |\ 1\leq i,j\leq 2m]$$ defined by the ideal generated by the relations \begin{flalign*}
		     y_{ij} =y_{ji},\ y_{ij} =y_{\ell\ell}\inverse y_{\ell i}y_{\ell j},\ x_i =y_{\ell\ell}\inverse x_\ell y_{\ell i},\ -x_\ell^2 =(\frac{\pi}{\ol{\pi}} \sum_{i=1}^m y_{\ell i}y_{\ell,n-i}) +\pi\sqrt{\theta}y_{\ell\ell}.
	\end{flalign*} Hence, the scheme $D(y_{\ell \ell})$ is isomorphic to \begin{flalign*}
		   \Spec \frac{\CO_F[x_\ell,y_{\ell 1},\ldots, y_{\ell\ell},\ldots,y_{\ell,2m},y_{\ell\ell}\inverse]}{(x_\ell^2+ (\frac{\pi}{\ol{\pi}} \sum_{i=1}^m y_{\ell i}y_{\ell,n-i}) +\pi\sqrt{\theta}y_{\ell\ell} )}.
	\end{flalign*} By the Jacobian criterion, $D(y_{\ell\ell})$ is smooth over $\CO_F$ of Krull dimension $n$.  Note that the worst point is defined (set-theoretically) by the ideal generated by $\pi$ and $y_{\ell\ell}$ for $1\leq \ell\leq 2m$. Since the generic fiber of $\RU^\rfl_{\cbra{0}}$ is smooth, we obtain that $\RU^\rfl_{\cbra{0}}$ is smooth over $\CO_F$ on the complement of the worst point. As the generic fiber and all $D(y_{\ell\ell})$, for $1\leq\ell\leq 2m$,  are irreducible, we conclude that $\RU^\rfl_{\cbra{0}}$ is irreducible. 
\end{proof}

\begin{lemma} \label{lem-spfiberCM}
	The scheme $\RU^\rfl_{\cbra{0}}$ is Cohen-Macaulay.  
\end{lemma}
\begin{proof}
	Let $\CS$ denote the polynomial ring $\CO_F[y_{ii}\ |\ 1\leq i\leq 2m]$. Then we have an obvious ring homomorphism $\CS\ra \CR^\rfl$. By the wedge condition \textbf{LM5} and $Y=Y^t$, for $1\leq i,j\leq 2m$, we have   \begin{flalign*}
	    y_{ij}^2 =y_{ij}y_{ji}=y_{ii}y_{jj} \text{\ and\ } x_i^2 =-(\frac{\pi}{\ol{\pi}} \sum_{\ell=1}^m y_{i \ell}y_{i,n-\ell}) -\pi\sqrt{\theta}y_{ii}
\end{flalign*} In particular, we deduce that $\CR^\rfl$ is integral (also of finite type) over $\CS$, and hence $\CR^\rfl$ is a finitely generated $\CS$-module. Since $\CS$ is a domain of the same Krull dimension as $\CR^\rfl$, the map $\CS\ra \CR^\rfl$ is necessarily injective. By \cite[Corollary 18.17]{eisenbud2013commutative}, to show $\CR^\rfl$ is Cohen-Macaulay, it suffices to show that $\CR^\rfl$ is a flat $\CS$-module. Equivalently, we need to show that the induced morphism $$\psi: \Spec \CR^\rfl\ra \Spec \CS\simeq \BA^{2m}$$ is flat.  Let $P_0$ be the closed point in $\Spec \CS$ corresponding to the maximal ideal $\fm_0\coloneqq (\pi, y_{11},\ldots,y_{2m,2m})$. Then $\psi$ maps the worst point of $\Spec \CR^\rfl$ to $P_0$ and the preimage of $\Spec \CS[y_{\ell\ell}\inverse]$ is the scheme $D(y_{\ell\ell})$ considered in the proof of Lemma \ref{lem-smoothoutside}. As $D(y_{\ell\ell})$ is smooth over $\CO_F$, by miracle flatness (see \cite[Theorem 18.16 b.]{eisenbud2013commutative}), the restriction $\psi|_{D(y_{\ell\ell})}$ is flat. Similarly, we obtain that $\psi$ restricted to the generic fiber of $\RU^\rfl_{\cbra{0}}$ is flat. It remains to show that $\psi$ is flat at the worst point, i.e., the localization map $\CS_{\fm_0}\ra \CR^\rfl_{\fm_0}$ is flat. The local ring $\CS_{\fm_0}$ has residue field $k$. Let $K$ denote the fraction field of $\CS_{\fm_0}$. By an application of Nakayama's lemma (see \cite[Chapter II, Lemma 8.9]{hartshorne2013algebraic}), we are reduced to show that \begin{flalign}
	   \dim_K(\CR^\rfl_{\fm_0}\otimes_{\CS_{\fm_0}}K)=\dim_k(\CR^\rfl_{\fm_0}\otimes_{\CS_{\fm_0}}k). \label{eq-dimen}
\end{flalign}  Note that $K$ is the field $F(y_{11},\ldots,y_{2m,2m})$ of rational functions. By the following Lemma \ref{lembas}, we have the desired equality \eqref{eq-dimen} of dimensions. 
\end{proof}

\begin{lemma}\label{lembas}
	The $K$-vector space (resp. $k$-vector space) $\CR^\rfl_{\fm_0}\otimes_{\CS_{\fm_0}}K$ (resp. $\CR^\rfl_{\fm_0}\otimes_{\CS_{\fm_0}}k$) has a $K$-basis (resp. $k$-basis) consisting of (images of) monomials \begin{flalign*}
	      x_i^\alpha y_{i_1j_1}^{\beta_1}y_{i_2j_2}^{\beta_2} \cdots y_{i_\ell j_\ell}^{\beta_\ell},
\end{flalign*} where $\alpha,\beta_i\in\cbra{0,1}$, $0\leq \ell\leq m$, and $1\leq i< i_1<j_1<i_2<j_2<\cdots<i_\ell<j_\ell\leq 2m$. Let $S$ denote the set of these monomials. Then the cardinality $\# S$ equals $2^{2m}$. In particular, \begin{flalign}
	   \dim_K(\CR^\rfl_{\fm_0}\otimes_{\CS_{\fm_0}}K)=\dim_k(\CR^\rfl_{\fm_0}\otimes_{\CS_{\fm_0}}k)=2^{2m}.
\end{flalign}
\end{lemma}
\begin{proof}
	We first count the cardinality of $S$. For an integer $0\leq \ell\leq  m$, the number of monomials of the form $x_i y_{i_1j_1}^{\beta_1}y_{i_2j_2}^{\beta_2} \cdots y_{i_\ell j_\ell}^{\beta_\ell}$ in $S$ is the number of tuples $(i,i_1,j_1,\ldots,i_\ell, j_\ell)$ such that $1\leq i< i_1<j_1<i_2<j_2<\cdots<i_\ell<j_\ell\leq 2m$. It is well-known that the number is $\binom{2m}{2\ell+1}$. Here, we set $\binom{2m}{2\ell +1}=0$ if $\ell=m$. Similarly, the number of monomials of the form $y_{i_1j_1}^{\beta_1}y_{i_2j_2}^{\beta_2} \cdots y_{i_\ell j_\ell}^{\beta_\ell}$ in $S$ is $\binom{2m}{2\ell}$. Hence, we obtain that \begin{flalign*}
		    \# S &= \sum_{\ell=0}^m\binom{2m}{2\ell+1} +\sum_{\ell=0}^m\binom{2m}{2\ell} = \sum_{i=0}^{2m}\binom{2m}{i}=2^{2m}.
	\end{flalign*}

	Let $x_i^\alpha x_j^{\alpha'} y_{i_1j_1}^{\beta_1}y_{i_2j_2}^{\beta_2} \cdots y_{i_\ell j_\ell}^{\beta_\ell}$ be a general monomial in $\CR^\rfl_{\fm_0}\otimes_{\CS_{\fm_0}}K$. As $y_{ij}^2=y_{ij}y_{ji} =y_{ii}y_{jj}$ in $\CR^\rfl_{\fm_0}$, we may assume $\beta_i$ for $1\leq i\leq \ell$ lies in $\cbra{0,1}$. As $$-X_3^tX_3=(\frac{\pi}{2\ol{\pi}}\tr({H}_{2m}Y)+\pi\sqrt{\theta})Y$$ in $\CR^\rfl_{\fm_0}$, we see $x_ix_j$ can be expressed by entries in $Y$. Hence, we may assume $\alpha'=0$ and $\alpha\in \cbra{0,1}$. We claim that the monomial $x_i^\alpha y_{i_1j_1}y_{i_2j_2} \cdots y_{i_\ell j_\ell}$ for $\alpha\in\cbra{0,1}$ is generated by elements in $S$. By the wedge condition and $Y=Y^t$, it is straightforward to check that the product $x_r y_{ij}y_{pq}$ only depends on the indices $\cbra{r, i,j,p,q}$, namely, changing the order of indices gives the same product in $\CR^\rfl_{\fm_0}$. Since $y_{ii}\in K$, we may assume $1\leq i< i_1<j_1<i_2<j_2<\cdots<i_\ell<j_\ell\leq 2m$, and hence we may assume $0\leq\ell\leq m$. Thus, the $K$-vector space $\CR^\rfl_{\fm_0}\otimes_{\CS_{\fm_0}}K$ is generated by (images of) the elements in $S$.
	Now it suffices to show that these elements are $K$-linearly independent.
	
	Note that the ring $\CR^\rfl_{\fm_0}\otimes_{\CS_{\fm_0}}K$ corresponds to the generic point of $\Spec \CR^\rfl$. Since $y_{11}$ is invertible over $\CR^\rfl_{\fm_0}\otimes_{\CS_{\fm_0}}K$, the ring $\CR^\rfl_{\fm_0}\otimes_{\CS_{\fm_0}}K$ is in fact the function field of $D(y_{11})$ in the proof of Lemma \ref{lem-smoothoutside} (take $\ell=1$), and we can identify the map \begin{flalign*}
		    K \lra \CR^\rfl_{\fm_0}\otimes_{\CS_{\fm_0}}K 
	\end{flalign*} with the field extension \begin{flalign*}
		   K=F(y_{11},\ldots,y_{2m,2m}) \lra \frac{K[y_{12},y_{13},\ldots,y_{1,2m},x_1]}{\rbra{y_{12}^2-y_{11}y_{22},\ldots, y_{1,2m}^2-y_{11}y_{2m,2m}, x_1^2+ (\frac{\pi}{\ol{\pi}} \sum_{i=1}^m y_{1 i}y_{1,n-i}) +\pi\sqrt{\theta}y_{11}  }}.
	\end{flalign*}
	We can see that this is a compositum of successive quadratic extensions. In particular, \begin{flalign*}
		    \dim_K(\CR^\rfl_{\fm_0}\otimes_{\CS_{\fm_0}}K) =2^{2m}.
	\end{flalign*} As $\# S=2^{2m}$, elements in $S$ are $K$-linearly independent, i.e., elements in $S$ form a $K$-basis of $\CR^\rfl_{\fm_0}\otimes_{\CS_{\fm_0}}K$.
	
	Similar arguments (just note that now $y_{ii}=0$ in $k$) as before imply that $\CR^\rfl_{\fm_0}\otimes_{\CS_{\fm_0}}k$ is generated by (images of) elements in $S$. Hence, \begin{flalign*}
		    \dim_k(\CR^\rfl_{\fm_0}\otimes_{\CS_{\fm_0}}k)\leq \# S=\dim_K(\CR^\rfl_{\fm_0}\otimes_{\CS_{\fm_0}}K).
	\end{flalign*} On the other hand, by Nakayama's lemma, we always have $$\dim_k(\CR^\rfl_{\fm_0}\otimes_{\CS_{\fm_0}}k)\geq \dim_K(\CR^\rfl_{\fm_0}\otimes_{\CS_{\fm_0}}K).$$
	This completes the proof of the lemma. 
\end{proof}

\begin{corollary}
	The scheme $\RU^\rfl_{\cbra{0}}$ is normal and flat over $\CO_F$. The geometric special fiber $\RU_{\cbra{0}}^\rfl\otimes_{\CO_F}\ol{k}$ is reduced and irreducible. 
\end{corollary}
\begin{proof}
	As $\RU^\rfl_{\cbra{0}}$ is smooth over $\CO_F$ on the complement of a closed point and Cohen-Macaulay by Lemma \ref{lem-smoothoutside} and \ref{lem-spfiberCM}, the normality of $\RU^\rfl_{\cbra{0}}$ follows from the Serre's criterion for normality (see \cite[031S]{stacks-project}). 
	By Lemma \ref{lem-smoothoutside}, the scheme $\RU_{\cbra{0}}^\rfl\otimes_{\CO_F}\ol{k}$ is smooth over $\ol{k}$ on the complement of the worst point. The proof of Lemma \ref{lem-smoothoutside} also implies that $\RU_{\cbra{0}}^\rfl\otimes_{\CO_F}\ol{k}$ is irreducible of dimension $n-1$. As $\RU^\rfl_{\cbra{0}}$ is Cohen-Macaulay and $\Spec\CO_F$ is regular, then $\RU^\rfl_{\cbra{0}}$ is flat over $\CO_F$ by the miracle flatness (see \cite[Theorem 18.16 b.]{eisenbud2013commutative}).   

    It remains to show that $\RU_{\cbra{0}}^\rfl\otimes_{\CO_F}\ol{k}$ is reduced. Using Serre's criterion for reducedness (see \cite[031R]{stacks-project}), it is enough to verify that this scheme satisfies Serre's conditions ($R_0$) and ($S_1$). Since the singular locus of $\RU_{\cbra{0}}^\rfl\otimes_{\CO_F}\ol{k}$ consists of a singleton, condition $(R_0)$ is satisfied.
	Since $\RU^\rfl_{\cbra{0}}$ is Cohen-Macaulay and $\pi$ is not a zero divisor (followed by the flatness of $\RU_{\cbra{0}}^\rfl\otimes_{\CO_F}\ol{k}$ over $\CO_F$), the scheme $\RU_{\cbra{0}}^\rfl\otimes_{\CO_F}\ol{k}$ is also Cohen-Macaulay (see \cite[Chapter II, Theorem 8.21A (d)]{hartshorne2013algebraic}). In particular, it satisfies ($S_1$).  
\end{proof}

\begin{lem} \label{prop-topflat}
 The schemes $\RU_{\cbra{0}}$ and $\RU^\rfl_{\cbra{0}}$ have the same underlying topological space. 
\end{lem}
\begin{proof}
	(1) Since $\RU_{\cbra{0}}^\rfl$ is flat over $\CO_F$, the scheme $\RU_{\cbra{0}}^\rfl$ is the Zariski closure of its generic fiber. Then we have closed immersions $$\RU^\rfl_{\cbra{0}}\hookrightarrow \RU_{\cbra{0}}\hookrightarrow \RU_{\cbra{0}}'$$ where all schemes have the same generic fiber. Then it suffices to prove that the special fibers of $\RU_{\cbra{0}}^\rfl$ and $\RU'_{\cbra{0}}$ have the same underlying topological space. Since $\RU^\rfl_{\cbra{0}}\otimes_{\CO_F} k$ is reduced, we are reduced to show that $\CI^\rfl\otimes_{\CO_F}k$ is contained in the radical of $\CI\otimes_{\CO_F}k$.
	  
	If $\omega(\pi_0)=\omega(t)$, then the assertion follows from Lemma \ref{lem-moduli}. We may assume $t/\pi_0$ is not a unit. In this case, we have   \begin{flalign*}
		  \CI\otimes_{\CO_F}k &= {\rbra{\wedge^2\begin{psmallmatrix}
	     	Y\\ X_3
	     \end{psmallmatrix}, Y-Y^t, \text{the diagonal of ($\frac{\tr(H_{2m}Y)}{2}Y+X_3^tX_3$)} } },\\ 
	     \CI^\rfl\otimes_{\CO_F}k &= {\rbra{\wedge^2\begin{psmallmatrix}
	     	Y\\ X_3
	     \end{psmallmatrix}, Y-Y^t, {\frac{\tr(H_{2m}Y)}{2}Y+X_3^tX_3} } }.
	\end{flalign*} 
	Let $M$ denote the matrix $\frac{\tr(H_{2m}Y)}{2}Y+X_3^tX_3$. Then for $1\leq i,j\leq 2m$, the $(i,j)$-entry $M_{ij}$ of $M$ is \begin{flalign*}
		   \alpha y_{ij}+x_ix_j, \quad \alpha\coloneqq {\tr(H_{2m}Y)}/{2}.
	\end{flalign*}
	Since $\Char(k)=2$, we obtain $M_{ij}^2 = \alpha^2y_{ij}^2+x_i^2x_j^2$. Therefore, we have \begin{flalign*}
		  M_{ij}^2 -M_{ii}M_{jj} &= \alpha^2(y_{ij}^2-y_{ii}y_{jj})-\alpha x_i^2y_{jj}-\alpha x_j^2y_{ii}\\ &=\alpha^2(y_{ij}^2-y_{ii}y_{jj})-x_i^2M_{jj}-x_j^2M_{ii}+2x_i^2x_j^2\\ &= \alpha^2(y_{ij}^2-y_{ii}y_{jj})-x_i^2M_{jj}-x_j^2M_{ii} \in \wt{\CI}\otimes_{\CO_F}k
	\end{flalign*}
	In particular, any $M_{ij}^2$ for $1\leq i,j\leq 2m$ lies in $\wt{\CI}\otimes_{\CO_F}k$. Hence, $\CI^\rfl\otimes_{\CO_F}k$ is contained in the radical of $\CI\otimes_{\CO_F}k$. This finishes the proof.
\end{proof}

In summary, we have proven the following.
\begin{prop}
	\label{prop-Ufl}
    \begin{enumerate}
    	\item The scheme $\RU^\rfl_{\cbra{0}}$ is flat over $\CO_F$ of relative dimension $n-1$. In particular, $\RU^\rfl_{\cbra{0}}$ is isomorphic to an open subscheme of the local model $\RM_{\cbra{0}}^\loc$ containing the worst point. Furthermore, $\RU^\rfl_{\cbra{0}}$ is normal, Cohen-Macaulay, and smooth over $\CO_F$ on the complement of the worst point. The special fiber $\RU_{\cbra{0}}^\rfl\otimes_{\CO_F}k$ is (geometrically) reduced and irreducible. 
    	\item $\RU_{\cbra{0}}$ and $\RU^\rfl_{\cbra{0}}$ have the same underlying topological space. 
    	\item If $\omega(\pi_0)=\omega(t)$, then $\RU_{\cbra{0}}=\RU^\rfl_{\cbra{0}}$.
    \end{enumerate} 
\end{prop}

\subsection{Global results}
Recall that $(\Lambda_0,q,\sL,\phi)$ is a hermitian quadratic module with $\phi$ over $\CO_{F_0}$ by Lemma \ref{lem-hermm0}. Let $\sH_{\cbra{0}}\coloneqq \ud{\Sim}((\Lambda_0,q,\sL,\phi))$ be the group scheme over $\CO_{F_0}$ of similitudes preserving $\phi$ of $(\Lambda_0,q,\sL,\phi)$. By Theorem \ref{thmsimm0}, $\sH_{\cbra{0}}$ is an affine smooth group scheme over $\CO_{F_0}$. 

\begin{lemma}\label{lemGstable}
	The group scheme $\sH_{\cbra{0}}$ acts on $\RM^\naive_{\cbra{0}}$ and $\RM_{\cbra{0}}$.
\end{lemma}
\begin{proof}
	It suffices to show the result for $\RM_{\cbra{0}}$. Let $R$ be an $\CO_F$-algebra. Let $g=(\varphi,\gamma) \in \sH_{\cbra{0}}(R)$ be a similitude preserving $\phi$. For $\CF\in \RM_{\cbra{0}}$, we define $g\CF\coloneqq \varphi(\CF)\sset \Lambda_0\otimes_{\CO_{F_0}}R$. We need to show that $g\CF\in \RM_{\cbra{0}}(R)$. It is clear that $g\CF$ satisfies conditions \textbf{LM1,2,4}. Recall that $\phi: \Lambda_0\times \Lambda_0\ra t\inverse\CO_{F_0}$ is defined by $(x,y)\mapsto t\inverse\Tr_{F/F_0} h(x,\pi\inverse y)$. We also use $\phi$ to denote the base change to $\Lambda_0\otimes_{\CO_{F_0}}R$. Then we see that $\CF$ satisfies \textbf{LM3} if and only if $\phi(\CF,\CF)=0$. As $g$ preserves $\phi$, we have that $$\phi(g\CF,g\CF)=\gamma \phi(\CF,\CF)=0.$$ So $g\CF$ satisfies \textbf{LM3}.   As $g$ is $\CO_F\otimes_{\CO_{F_0}}R$-linear by definition, we obtain that $$(\pi\otimes 1-1\otimes\ol{\pi})\circ g=g\circ (\pi\otimes 1-1\otimes\ol{\pi}).$$ By the functoriality of the wedge product of linear maps, we have $$\wedge^2(\pi\otimes 1-1\otimes\ol{\pi}\ |\ g\CF)=\wedge^2(g\circ (\pi\otimes 1-1\otimes \ol{\pi})\ |\ \CF)=\wedge^2(g)\circ \wedge^2(\pi\otimes 1-1\otimes\ol{\pi}\ |\ \CF)=0. $$
	Therefore, $g\CF$ satisfies the wedge condition $\textbf{LM5}$. Since $\sH_{\cbra{0}}$ is smooth over $\CO_{F_0}$, using a similar argument of \cite[Lemma 7.1]{rapoport2018regular}, we can show that the $R$-submodule $$ \Im\rbra{W(\Lambda_0)^{n-1,1}_{-1}\otimes_{\CO_F}R\ra W(\Lambda_0)\otimes_{\CO_F}R}$$ of $W(\Lambda_0)\otimes_{\CO_F}R$ is stable under the natural action of $\sH_{\cbra{0}}(R)$ on $W(\Lambda_0)\otimes_{\CO_F}R=\wedge^n(\Lambda_0\otimes_{\CO_{F_0}}R)$. It follows that $g\CF$ satisfies the strengthened spin condition \textbf{LM6}.  
\end{proof}

\begin{lemma}\label{lem-Gorbit}
	Let $\ol{k}$ be the algebraic closure of the residue field $k$. Then $\RM_{\cbra{0}}\otimes_{\CO_F}\ol{k}$ has two $\sH_{\cbra{0}}\otimes_{\CO_{F_0}}\ol{k}$-orbits, one of which consists of the worst point. 
\end{lemma}
\begin{proof}
	By Lemma \ref{lemGstable}, the special fiber $\RM_{\cbra{0}}\otimes_{\CO_F}\ol{k}$ has an action of $\sH_{\cbra{0}}\otimes_{\CO_{F_0}}\ol{k}$. Let $\CF\in\RM_{\cbra{0}}(\ol{k})$. In particular, the subspace $\CF\sset(\Lambda_0\otimes_{\CO_{F_0}}\ol{k})$ is  an $n$-dimensional $\ol{k}$-vector space. The wedge condition in this case becomes $\wedge^2(\pi\otimes 1\ |\ \CF)=0$. Therefore, the image $(\pi\otimes 1)\CF$ is at most one dimensional. We have the following two cases.
	
	Suppose $(\pi\otimes 1)\CF=0$. Then $\CF=(\pi\otimes 1)(\Lambda_0\otimes_{\CO_{F_0}}\ol{k})$, namely, $\CF$ is the worst point.
	
	Suppose $(\pi\otimes 1)\CF$ is one dimensional. Then there exists a vector $v\in\CF$ such that $(\pi\otimes 1)v$ generates $(\pi\otimes 1)\CF$. For simplicity, write $\pi$ for $\pi\otimes 1$. Recall the $\ol{k}$-bilinear form \begin{flalign*}
		 \phi(-,-): (\Lambda_0\otimes_{\CO_{F_0}}\ol{k})\times (\Lambda_0\otimes_{\CO_{F_0}}\ol{k}) &\lra \sL\otimes_{\CO_{F_0}}\ol{k}\\ (x,y) &\mapsto s(x,\pi\inverse y)=t\inverse \Tr h(x,\pi\inverse y),
	\end{flalign*} where $\pi\inverse$ is the induced isomorphism $\Lambda_0\otimes_{\CO_{F_0}}\ol{k}\simto ({\pi\inverse}\Lambda_0)\otimes_{\CO_{F_0}}\ol{k}$. We can identify $\sL\otimes_{\CO_{F_0}}\ol{k}$ with $\ol{k}$ by sending $t\inverse \otimes 1$ to $1$. 
	Denote by $N\coloneqq \ol{k}\pair{e_{m+1},\pi e_{m+1}}$ the submodule of $\Lambda_0\otimes_{\CO_{F_0}}\ol{k}$. Then one can check that the radical of $\phi$ is contained in $N$. We claim that $\pi v$ is not in $N$. Otherwise, after rescaling, we may assume $v=e_{m+1}\otimes1+\pi v_1$ for some $v_1\in \Lambda_0\otimes_{\CO_{F_0}}\ol{k}$. Then for the quadratic form $$q: \Lambda_0\otimes_{\CO_{F_0}}\ol{k}\lra \sL\otimes_{\CO_{F_0}}\ol{k}\simeq\ol{k} ,$$ we have  \begin{flalign*}
		     q(v) &=q(e_{m+1}\otimes 1+\pi v_1)=q(e_{m+1}\otimes 1)+s(e_{m+1}\otimes 1,\pi v_1)+q(\pi v_1).
	\end{flalign*}
	One can check that $q(e_{m+1}\otimes 1)=1$ and $s(e_{m+1}\otimes 1,\pi v_1)=q(\pi v_1)=0$. Hence $q(v)\neq 0$. This contradicts the hyperbolicity condition \textbf{LM4} that $q(\CF)=0$. In particular, we obtain that $\pi v$ is not in the radical of $\phi$. Thus, we can find $w\in \Lambda_0\otimes_{\CO_{F_0}}\ol{k}$ such that $\phi(w,\pi v)\neq 0$ in $\ol{k}$. By rescaling, we may assume $\phi(w,\pi v)=1$. 
	Note that for $a\in\ol{k}$, \begin{flalign*}
		q(w+av) &=q(w)+as(w,v)+a^2q(v)\\ &=q(w)+a\phi(w,\pi v)+0,\quad \text{since $q(v)=0$,}\\ &=q(w)+a.
	\end{flalign*} 
	Replacing $w$ by $w-q(w)v$, we may assume $q(w)=0$. Put $b\coloneqq  -\phi(w,v)$. One can check that $\phi(w+b\ol{\pi}w)=0$. Replacing $w$ by $w+b\ol{\pi}w$, we have \begin{flalign*}
		\text{$q(w)=q(v)=0$, $\phi(w,v)=0$ and $\phi(w,\pi v)=1$.}
	\end{flalign*}  
	Let $W_1\coloneqq \pair{v,\pi v,w,\pi w}$, the $\ol{k}$-subspace of $\Lambda_0\otimes_{\CO_{F_0}}\ol{k}$ generated by $v,\pi v,w,\pi w$. Then $\phi$ restricts to a perfect pairing on $W_1$. Now we can write \begin{flalign}
		    \Lambda_0\otimes_{\CO_{F_0}}\ol{k} =W_1\oplus W, \label{spaceW}
	\end{flalign} where $W$ is the orthogonal complement of $W_1$ with respect to $\phi$ whose dimension is $2n-4$ over $\ol{k}$. Note that the Condition \ref{LMincl} in Definition \ref{defn-M} of $\RM_{\cbra{0}}$ implies that $\phi(\CF,\CF)=0$, and hence $\CF\cap \pair{w,\pi w}=0$. Since $\pair{v,\pi v}\sset\CF$ and $\phi(\CF,\CF)=0$, we obtain that the $\ol{k}$-dimension of $\CF\cap W$ is $n-2$ and $\CF\cap W$ is contained in $\pi W=\ker(\pi\ |\ W)$. Therefore, $\CF\cap W=\pi W$ for dimension reasons.
     By \eqref{spaceW}, we have \begin{flalign*}
		    \disc'(\phi) = \disc(\phi|_{W_1})\disc'(\phi|_W).
	\end{flalign*}
	Here $\disc'(\phi)$ is the divided discriminant in the sense of Definition \ref{app-defn2}, and we view it as an element in $\ol{k}$ by using a basis of $\Lambda_0\otimes_{\CO_{F_0}}\ol{k}$. By Example \ref{exlambda0}, we have $\disc'(\phi)\in\ol{k}\cross$.
	Since $\phi$ is perfect on $W_1$, we obtain that $\disc(\phi|_{W_1})\in\ol{k}\cross$, and hence $\disc'(\phi|_W)\in \ol{k}\cross$. So $W$ is a hermitian quadratic module of type $\Lambda_0$ over $\ol{k}$ in the sense of Definition \ref{app-defn2}. Set $v_1\coloneqq v$ and $v_n\coloneqq w$. By applying Theorem \ref{thmstand0} to $W$, we deduce that there is an $\CO_F\otimes_{\CO_{F_0}}\ol{k}$-basis $\tcbra{v_i: 1\leq i\leq n}$ of $\Lambda_0\otimes_{\CO_{F_0}}\ol{k}$ with the property that $q(v_{m+1})$ generates $R$, $q(v_i)=0$, $\phi(v_i,v_j)=0$ and $\phi(v_i,\pi v_j)=\delta_{i,n+1-j}$ for all $1\leq i<j\leq n$. With respect to this basis, we have $$\CF=\pair{v,\pi v}\oplus(\CF\cap W)= \pair{v,\pi v}\oplus(\pi W)=\pair{v_1,\pi v_1,\pi v_{i}, 2\leq i\leq n-1}.$$ 
	This shows that points $\CF\in\RM_{\cbra{0}}(\ol{k})$ with $\dim_{\ol{k}} \pi\CF=1$ are in the same $\sH_{\cbra{0}}(\ol{k})$-orbit.
\end{proof}

As $\RU^\rfl_{\cbra{0}}$ is flat over $\CO_F$, we may view $\RU^\rfl_{\cbra{0}}$ as an open subscheme of $\RM^\loc_{\cbra{0}}$ containing the worst point. By Lemma \ref{lem-Gorbit}, the  $\sH_{\cbra{0}}$-translation of $\RU^\rfl_{\cbra{0}}$ covers $\RM^\loc_{\cbra{0}}$. By Proposition \ref{prop-Ufl}, we have shown Theorem \ref{intro-thm1.6}, and Theorem \ref{thm-intromain}, \ref{thm16} in the case $I=\cbra{0}$ and (R-U).

\section{The case $I=\cbra{0}$ and (R-P)}\label{sec-0rp}
In this section, we consider the case when $F/F_0$ is of (R-P) type. In particular, we have $$\pi^2+\pi_0=0 \text{\ and $\ol{\pi}=-\pi$}. $$  Consider the following ordered $\CO_{F_0}$-basis of $\Lambda_0$ and $\Lambda_0^s$: \begin{alignat}{2}
	    \Lambda_0&: \half e_{m+2},\ldots,\half e_n,e_1,\ldots,e_m,e_{m+1},\frac{\pi}{2}e_{m+2},\ldots,\frac{\pi}{2}e_n,\pi e_1,\ldots,\pi e_m,\pi e_{m+1}, \label{rpbasis-1}  \\  \Lambda_0^s&: \pi\inverse e_{m+2},\ldots, \pi\inverse e_n,  \frac{2}{\pi}e_1,\ldots,\frac{2}{\pi}e_m, \pi\inverse e_{m+1}, e_{m+2},\ldots, e_n,2 e_1,\ldots,2e_m, e_{m+1}. \label{rpbasis-2}
\end{alignat}
Recall $(\Lambda_0,q,\sL)$ is a hermitian quadratic module for $\sL=\frac{1}{2}\CO_{F_0}$.

\subsection{A refinement of $\RM^\naive_{\tcbra{0}}$ in the (R-P) case}
\begin{defn}\label{rpdefn-M}
	Let $\RM_{\cbra{0}}$ be the functor $$\RM_{\cbra{0}}: (\Sch/\CO_{F})^\op\lra \Sets$$ which sends an $\CO_F$-scheme $S$ to the set of $\CO_S$-modules $\CF$ such that 
	\begin{enumerate}[label=\textbf{LM\arabic*}]
		\item ($\pi$-stability condition) $\CF$ is an $\CO_F\otimes_{\CO_{F_0}}\CO_S$-submodule of $\Lambda_0\otimes_{\CO_{F_0}}\CO_S$ and as an $\CO_S$-module, it is a locally direct summand of rank $n$.
		\item (Kottwitz condition) The action of $\pi\otimes 1\in\CO_F\otimes_{\CO_{F_0}}\CO_S$ on $\CF$ has characteristic polynomial $$\det(T-\pi\otimes 1\ |\ \CF)=(T-\pi)(T-\ol{\pi})^{n-1}.$$
		\item Let $\CF^\perp$ be the orthogonal complement in $\Lambda_0^s\otimes_{\CO_{F_0}}\CO_S$ of $\CF$ with respect to the perfect pairing $$ s(-,-): (\Lambda_0\otimes_{\CO_{F_0}}\CO_S) \times (\Lambda_0^s\otimes_{\CO_{F_0}}\CO_S) \ra \CO_S.$$ We require the map $\Lambda_0\otimes_{\CO_{F_0}}\CO_S \ra (\frac{{\pi}}{2} \Lambda_0^s)\otimes_{\CO_{F_0}}\CO_S$ induced by $\Lambda_0\hookrightarrow \frac{\pi}{2}\Lambda_0^s$ sends $\CF$ to $\frac{{\pi}}{2}\CF^\perp,$ where $\frac{{\pi}}{2}\CF^\perp$ is the image of $\CF^\perp$ under the isomorphism $\frac{{\pi}}{2}:\Lambda_0^s\otimes_{\CO_{F_0}}\CO_S \simto \frac{{\pi}}{2}\Lambda_0^s\otimes_{\CO_{F_0}}\CO_S$.  \label{rpLMincl}
		\item (Hyperbolicity condition) The quadratic form $q: \Lambda_0\otimes_{\CO_{F_0}}\CO_S\ra \sL\otimes_{\CO_{F_0}}\CO_S$ induced by $q: \Lambda_0\ra \sL$ satisfies $q(\CF)=0$. 
		\item (Wedge condition) The action of $\pi\otimes 1-1\otimes\ol{\pi}\in \CO_F\otimes_{\CO_{F_0}}\CO_S$ on $\CF$ satisfies $$\wedge^2(\pi\otimes 1-1\otimes \ol{\pi}\ |\ \CF )=0.$$
	\end{enumerate}
\end{defn} 

\begin{remark}
    In this case, the strengthened spin condition appears to be implied by conditions \textbf{LM1-5} (although we do not have a proof). Since it does not play a role in our computation of the flat closure of $\RM_{\cbra{0}}$, we omit it from the definition of $\RM_{\cbra{0}}$. A similar remark applies to the case $I=\cbra{m}$ and (R-U) in \S \ref{sec-mru}.  
\end{remark}
As in the (R-U) case, the functor $\RM_{\cbra{0}}$ is representable and we have closed immersions $$\RM_{\cbra{0}}^\loc\subset\RM_{\cbra{0}}\subset\RM^\naive_{\tcbra{0}}$$ of projective schemes over $\CO_F$, where all schemes have the same generic fiber.

\subsection{An affine chart $\RU_{\cbra{0}}$ around the worst point}
Set $$\CF_0\coloneqq (\pi\otimes 1)(\Lambda_0\otimes_{\CO_{F_0}}k).$$ Then we can check that $\CF_0\in\RM_{\cbra{0}}(k)$. We call it the worst point of $\RM_{\cbra{0}}$. 
With respect to the basis \eqref{rpbasis-1}, the standard affine chart around $\CF_0$ in $\Gr(n,\Lambda_0)_{\CO_F}$ is the $\CO_F$-scheme of $2n\times n$ matrices $\left(\begin{smallmatrix}
	X\\ I_n
\end{smallmatrix}\right)$. As in \S \ref{subsec4.2}, we can obtain an affine neighborhood $\RU_{\cbra{0}}$ in $\RM_{\cbra{0}}$ containing $\CF_0$. We will analyze the conditions \textbf{LM1-5} and obtain the affine coordinate ring of $\RU_{\cbra{0}}$ in Proposition \ref{prop-U0rpring}. 


\subsubsection{Condition \textbf{LM1}} \label{rpsubsubsec-lm1}
Let $R$ be an $\CO_F$-algebra. With respect to the basis \eqref{rpbasis-1}, the operator $\pi\otimes 1$ acts on $\Lambda_0\otimes_{\CO_{F_0}}R$ via the matrix \begin{flalign*}
	    \begin{pmatrix}
	    	0 &-\pi_0I_n\\ I_n &0
	    \end{pmatrix}.
\end{flalign*}
Then the $\pi$-stability condition \textbf{LM1} on $\CF$ means there exists an $n\times n$ matrix $P\in M_{n\times n}(R)$ such that \begin{flalign*}
	   \begin{pmatrix}
	    	0 &-\pi_0I_n\\ I_n &0
	    \end{pmatrix} \begin{pmatrix}
	    	X\\ I_n
	    \end{pmatrix} =\begin{pmatrix}
	    	X\\ I_n
	    \end{pmatrix}P.
\end{flalign*}
We obtain $P=X$ and $X^2+\pi_0I_n=0$. 

\subsubsection{Condition \textbf{LM2}} We have already shown that $\pi\otimes 1$ acts on $\CF$ via right multiplication of $X$. Then as in the (R-U) case, the Kottwitz condition \textbf{LM2} translates to 
\begin{flalign}
	    \tr(X+\pi I_n)=\pi-\ol{\pi}=2\pi,\ \tr\rbra{\wedge^i(X+\pi I_n)}=0, \text{\ for $i\geq 2$}. \label{kott}
\end{flalign} 

\subsubsection{Condition \textbf{LM3}} 
 With respect to the bases \eqref{rpbasis-1} and \eqref{rpbasis-2}, the perfect pairing $$s(-,-):(\Lambda_0\otimes_{\CO_{F_0}}R)\times (\Lambda_0^s\otimes_{\CO_{F_0}}R)\ra R  $$ and the map $\Lambda_0\otimes_{\CO_{F_0}}R \ra \frac{{\pi}}{2}\Lambda_0^s\otimes_{\CO_{F_0}}R$ are represented respectively by the matrices \begin{flalign*}
	   S=\begin{pmatrix}
	   	  0 &0 &H_{2m} &0 \\ 0 &0 &0 &1\\ -H_{2m} &0 &0 &0 \\ 0 &-1 &0 &0
	   \end{pmatrix} \text{\ and\ } N= \begin{pmatrix}
     	  I_{2m} &0 &0 &0\\ 0 &2 &0 &0\\ 0 &0 &I_{2m} &0\\ 0 &0 &0 &2
     \end{pmatrix}.
\end{flalign*}  
Then the Condition \textbf{LM3} translates to $
	   \matx^tS\rbra{N\matx }=0,$ or equivalently, \begin{flalign}
	\matx^t\begin{pmatrix}
		 0 &0 &H_{2m} &0\\ 0 &0 &0 &2\\ -H_{2m} &0 &0 &0\\ 0 &-2 &0 &0
	\end{pmatrix}\matx &=0. \label{rpeq-lm41} 
\end{flalign}

 Write $$X=\begin{pmatrix}
 	  X_1 &X_2\\ X_3 &x
 \end{pmatrix},$$ where $X_1\in M_{2m}(R)$, $X_2\in M_{2m\times 1}(R)$, $X_3\in M_{1\times 2m}(R)$ and $x\in R$. Then \eqref{rpeq-lm41} translates to \begin{flalign*}
		                 \begin{pmatrix}
		                 	X_1^tH_{2m}-H_{2m}X_1 &2X_3^t-H_{2m}X_2\\ X_2^tH_{2m}-2X_3 &0
		                 \end{pmatrix} =0.
	                \end{flalign*}
	                
\subsubsection{Condition \textbf{LM4}}
Recall $\sL=\frac{1}{2}\CO_{F_0}$. With respect to the basis \eqref{rpbasis-1}, the induced $\sL\otimes_{\CO_{F_0}}R$-valued symmetric pairing on $\Lambda_0\otimes_{\CO_{F_0}}R$ is represented by the matrix \begin{flalign*}
	  S_1=\begin{pmatrix}
	     H_{2m} &0 &0 &0\\ 0 &2 &0 &0\\ 0 &0 &\pi_0H_{2m} &0\\ 0 &0 &0 &2\pi_0
\end{pmatrix}. 
\end{flalign*} 
The Condition \textbf{LM4} translates to \begin{gather*}
	 \matx^tS_1\matx=0 \text{\ and\ } \text{half of the diagonal  of $\matx^tS_1\matx$ equals zero}. 
\end{gather*} 
One can check that the diagonal entries of $\smatx^tS_1\smatx$ are indeed divisible by $2$ in $R$. Equivalently, we obtain \begin{flalign*}
		             &\begin{pmatrix}
		             	 X_1^tH_{2m}X_1+2X_3^tX_3+\pi_0H_{2m} &X_1^tH_{2m}X_2+2xX_3^t\\ X_2^tH_{2m}X_1+2xX_3 &X_2^tH_{2m}X_2+2x^2+2\pi_0
		             \end{pmatrix} =0, \\ &\text{half of the diagonal  
		             of $X_1^tH_{2m}X_1+2X_3^tX_3+\pi_0H_{2m}$ equals $0$}, \\ &\half\rbra{X_2^tH_{2m}X_2+2x^2+2\pi_0}=0.
	          \end{flalign*}
	          
\subsubsection{Condition \textbf{LM5}}
As $\pi\otimes 1$ acts as right multiplication by $X$ on $\CF$, the wedge condition on $\CF$ translates to \begin{flalign*}
	   \wedge^2(X+\pi I_n)=0.
\end{flalign*}

\subsubsection{A simplification of equations} As in the (R-U) case, we can simplify the above equations and obtain the following proposition.

\begin{prop}\label{prop-U0rpring}
	The scheme $\RU_{\cbra{0}}=\Spec\CO_F[X]/\CI$, where $\CI$ is the ideal generated by: 
     \begin{align*}
	     &\tr(X+\pi I_n)-2\pi,\  \wedge^2(X+\pi I_n),\ X_1^tH_{2m}-H_{2m}X_1, \ 2X_3^t-H_{2m}X_2, \\  &(\tr(X_1+\pi I_{2m})-2\pi)H_{2m} ({X_1}+\pi I_{2m})+2X_3^tX_3,\\ &\text{half of the diagonal  of $(\tr(X_1+\pi I_{2m})-2\pi)H_{2m} ({X_1}+\pi I_{2m})+2X_3^tX_3$}.
     \end{align*}
\end{prop}
  
Set $$ \wt{X}_1\coloneqq X_1+\pi I_{2m},\ \wt{X}\coloneqq \begin{pmatrix}
	  \wt{X}_1 \\ X_3
\end{pmatrix}.$$ 
Then we have the following proposition.
\begin{prop}
	The scheme $\RU_{\cbra{0}}$ is isomorphic to $\Spec\CO_F[\wt{X}]/\wt{\CI}$, where $\wt{\CI}$ is the ideal in $\CO_F[\wt{X}]$ generated by:  \begin{align*}
	    &\wedge^2(\wt{X}),\
	    {H}_{2m}\wt{X}_1-\wt{X}_1^t{H}_{2m} ,\
	    (\tr(\wt{X}_1)-2\pi)H_{2m} \wt{X}_1+2X_3^tX_3, \\
	    &\text{half of the diagonal  of $(\tr(\wt{X}_1)-2\pi)H_{2m} \wt{X}_1+2X_3^tX_3$. }
\end{align*}  
\end{prop}

\begin{defn}
	Denote by $\RU^\rfl_{\cbra{0}}$ the closed subscheme of $\RU_{\cbra{0}}=\Spec\CO_F[\wt{X}]/\wt{\CI}$ defined by the ideal $\wt{\CI}^\rfl\subset\CO_F[\wt{X}]$ generated by:  \begin{align*}
	    &\wedge^2(\wt{X}),\  
	    {H}_{2m}\wt{X}_1-\wt{X}_1^t{H}_{2m},\   
	    (\half{\tr(\wt{X}_1)}-\pi)H_{2m} \wt{X}_1+X_3^tX_3. 
\end{align*}
    Note that $\tr(\wt{X}_1)$ is divisible by $2$ by the relation ${H}_{2m}\wt{X}_1=\wt{X}_1^t{H}_{2m}$.
\end{defn}

\subsection{Global results}
As in the (R-U) case, we can prove the following proposition.
\begin{prop}
	\label{rpprop-Ufl}
    \begin{enumerate}
    	\item The scheme $\RU^\rfl_{\cbra{0}}$ is flat over $\CO_F$ of relative dimension $n-1$. In particular, $\RU^\rfl_{\cbra{0}}$ is isomorphic to an open subscheme of $\RM_{\cbra{0}}^\loc$ containing the worst point. Furthermore, $\RU^\rfl_{\cbra{0}}$ is normal, Cohen-Macaulay, and smooth over $\CO_F$ on the complement of the worst point. The special fiber $\RU_{\cbra{0}}^\rfl\otimes_{\CO_F}k$ is (geometrically) reduced and irreducible. 
    	\item $\RU_{\cbra{0}}$ and $\RU^\rfl_{\cbra{0}}$ have the same underlying topological space. 
    \end{enumerate} 
\end{prop}

Similar arguments as in the proof of Lemma \ref{lem-Gorbit} imply that the special fiber $\RM_{\cbra{0}}\otimes_{\CO_F}\ol{k}$ has only two $\sH_{\cbra{0}}(\ol{k})$-orbits. Together with Proposition \ref{rpprop-Ufl}, we can deduce Theorem \ref{thm-intromain} and \ref{thm16} in the case $I=\cbra{0}$ and (R-P).


\section{The case $I=\cbra{m}$ and (R-U)} \label{sec-mru}
In this section, we consider the case when $F/F_0$ is of (R-U) type and $I=\cbra{m}$. In particular, we have $$\pi^2-t\pi+\pi_0=0, $$ where $t\in \CO_{F_0}$ with $\pi_0|t\text{\ and\ }t|2$.  Consider the following ordered $\CO_{F_0}$-basis of $\Lambda_m$ and $\Lambda_m^s$: \begin{alignat}{2}
	   \Lambda_m&: \frac{\ol{\pi}}{t} e_{m+2},\ldots,\frac{\ol{\pi}}{t} e_n,\pi\inverse e_1,\ldots,\pi\inverse e_m,e_{m+1},\frac{\pi_0}{t}e_{m+2},\ldots,\frac{\pi_0}{t}e_n, e_1,\ldots,e_m,\pi e_{m+1}, \label{mbasis-1}  \\ \Lambda_m^s&: \ol{\pi} e_{m+2},\ldots, \ol{\pi} e_n,\frac{t}{\pi}e_1,\ldots,\frac{t}{\pi}e_m,  e_{m+1}, \pi_0e_{m+2},\ldots, \pi_0e_n, te_1,\ldots,te_m, \pi e_{m+1}. \label{mbasis-2}
\end{alignat}
Recall $(\Lambda_m,q,\sL)$ is a hermitian quadratic module for $\sL=t\inverse \CO_{F_0}$.

\subsection{A refinement of $\RM^\naive_{\tcbra{m}}$ in the (R-U) case}
\begin{defn}\label{mdefn-M}
	Let $\RM_{\cbra{m}}$ be the functor $$\RM_{\cbra{m}}: (\Sch/\CO_{F})^\op\lra \Sets$$ which sends an $\CO_F$-scheme $S$ to the set of $\CO_S$-modules $\CF$ such that 
	\begin{enumerate}[label=\textbf{LM\arabic*}]
		\item ($\pi$-stability condition) $\CF$ is an $\CO_F\otimes_{\CO_{F_0}}\CO_S$-submodule of $\Lambda_m\otimes_{\CO_{F_0}}\CO_S$ and as an $\CO_S$-module, it is a locally direct summand of rank $n$.
		\item (Kottwitz condition) The action of $\pi\otimes 1\in\CO_F\otimes_{\CO_{F_0}}\CO_S$ on $\CF$ has characteristic polynomial $$\det(T-\pi\otimes 1\ |\ \CF)=(T-\pi)(T-\ol{\pi})^{n-1}.$$
		\item Let $\CF^\perp$ be the orthogonal complement in $\Lambda_m^s\otimes_{\CO_{F_0}}\CO_S$ of $\CF$ with respect to the perfect pairing $$ s(-,-): (\Lambda_m\otimes_{\CO_{F_0}}\CO_S) \times (\Lambda_m^s\otimes_{\CO_{F_0}}\CO_S) \ra \CO_S.$$ We require that the map $\Lambda_m\otimes_{\CO_{F_0}}\CO_S \ra (t\inverse \Lambda_m^s)\otimes_{\CO_{F_0}}\CO_S$ induced by the inclusion $\Lambda_m\hookrightarrow t\inverse\Lambda_m^s$ sends $\CF$ to $t\inverse \CF^\perp,$ where $t\inverse \CF^\perp$ is the image of $\CF^\perp$ under the isomorphism $t\inverse:\Lambda_m^s\otimes_{\CO_{F_0}}\CO_S \simto t\inverse\Lambda_m^s\otimes_{\CO_{F_0}}\CO_S$.  
		\item (Hyperbolicity condition) The quadratic form $q: \Lambda_m\otimes_{\CO_{F_0}}\CO_S\ra \sL\otimes_{\CO_{F_0}}\CO_S$ induced by $q: \Lambda_m\ra\sL$ satisfies $q(\CF)=0$. 
		\item (Wedge condition) The action of $\pi\otimes 1-1\otimes\ol{\pi}\in \CO_F\otimes_{\CO_{F_0}}\CO_S$ on $\CF$ satisfies $$\wedge^2(\pi\otimes 1-1\otimes \ol{\pi}\ |\ \CF )=0.$$
	\end{enumerate}
\end{defn} 
Then $\RM_{\cbra{m}}$ is representable and we have closed immersions $$\RM^\loc_{\cbra{m}}\subset\RM_{\cbra{m}}\subset\RM^\naive_{\tcbra{m}}$$ of projective schemes over $\CO_F$, where all schemes have the same generic fiber.

\subsection{An affine chart $\RU_{\cbra{m}}$ around the worst point}\label{subsec6.2}
Set $$\CF_0\coloneqq (\pi\otimes 1)(\Lambda_m\otimes_{\CO_{F_0}}k).$$ Then we can check that $\CF_0\in\RM_{\cbra{m}}(k)$. We call it the worst point of $\RM_{\cbra{m}}$. 
With respect to the basis \eqref{mbasis-1}, the standard affine chart around $\CF_0$ in $\Gr(n,\Lambda_m)_{\CO_F}$ is the $\CO_F$-scheme of $2n\times n$ matrices $\left(\begin{smallmatrix}
	X\\ I_n
\end{smallmatrix}\right)$. 
As in the case $I=\tcbra{0}$, we can obtain an affine neighborhood $\RU_{\cbra{m}}$ in $\RM_{\cbra{m}}$ containing $\CF_0$. We will analyze the conditions \textbf{LM1-5} and obtain the affine coordinate ring of $\RU_{\cbra{m}}$ in Proposition \ref{prop-Umruring}. 

\subsubsection{Condition \textbf{LM1}} 
Let $R$ be an $\CO_F$-algebra. With respect to the basis \eqref{mbasis-1}, the operator $\pi\otimes 1$ acts on $\Lambda_m\otimes_{\CO_{F_0}}R$ via the matrix \begin{flalign*}
	    \begin{pmatrix}
	    	0 &-\pi_0I_n\\ I_n &tI_n
	    \end{pmatrix}.
\end{flalign*}
Then the $\pi$-stability condition \textbf{LM1} on $\CF$ means there exists an $n\times n$ matrix $P\in M_{n\times n}(R)$ such that \begin{flalign*}
	   \begin{pmatrix}
	    	0 &-\pi_0I_n\\ I_n &tI_n
	    \end{pmatrix} \begin{pmatrix}
	    	X\\ I_n
	    \end{pmatrix} =\begin{pmatrix}
	    	X\\ I_n
	    \end{pmatrix}P.
\end{flalign*}
We obtain $P=X+tI_n$ and $X^2+tX+\pi_0I_n=0$. 

\subsubsection{Condition \textbf{LM2}} We have already shown that $\pi\otimes 1$ acts on $\CF$ via right multiplication of $X+t I_n$. Then the Kottwitz condition \textbf{LM2} translates to 
\begin{flalign}
	    \tr(X+\pi I_n)=\pi-\ol{\pi},\ \tr\rbra{\wedge^i(X+\pi I_n)}=0, \text{\ for $i\geq 2$}. 
\end{flalign} 

\subsubsection{Condition \textbf{LM3}} 
 With respect to the bases \eqref{mbasis-1} and \eqref{mbasis-2}, the perfect pairing $$s(-,-):(\Lambda_m\otimes_{\CO_{F_0}}R)\times (\Lambda_m^s\otimes_{\CO_{F_0}}R)\ra R $$ and the map $\Lambda_m\otimes_{\CO_{F_0}}R \ra \frac{{1}}{t}\Lambda_m^s\otimes_{\CO_{F_0}}R$  are represented respectively by the matrices \begin{flalign*}
	   S=\begin{pmatrix}
	   	  \frac{2}{t}H_{2m} &0 &H_{2m} &0 \\ 0 &\frac{2}{t} &0 &1\\ H_{2m} &0 &\frac{2\pi_0}{t}H_{2m} &0 \\ 0 &1 &0 &\frac{2\pi_0}{t}
	   \end{pmatrix} \text{\ and\ } N= \begin{pmatrix}
     	  I_{2m} &0 &0 &0\\ 0 &t &0 &0\\ 0 &0 &I_{2m} &0\\ 0 &0 &0 &t
     \end{pmatrix}.
\end{flalign*} 
Then the Condition \textbf{LM3} translates to $
	   \matx^tS\rbra{N\matx }=0,$ or equivalently, \begin{flalign}
	\matx^t\begin{pmatrix}
		 \frac{2}{t}H_{2m} &0 &H_{2m} &0\\ 0 &2 &0 &t\\ H_{2m} &0 &\frac{2\pi_0}{t}H_{2m} &0\\ 0 &t &0 &2\pi_0
	\end{pmatrix}\matx &=0. \label{meq-lm41}  
\end{flalign}
It amounts to the following equation.
\begin{flalign}
	(\frac{2}{t}X^t+I_n) \begin{pmatrix}
		H_{2m} &0 \\ 0 &t
	\end{pmatrix}X +X^t\begin{pmatrix}
		H_{2m} &0 \\ 0 &t
	\end{pmatrix} +\begin{pmatrix}
		\frac{2\pi_0}{t} H_{2m} &0 \\ 0 &2\pi_0
	\end{pmatrix} =0.  \label{eq-m1}
\end{flalign}
Note that the $\pi$-stability condition \textbf{LM1} on $\CF$ implies \begin{flalign*}
	  \frac{2}{t}(X^t)^2+2X^t+\frac{2\pi_0}{t}I_n=0, \ \text{and hence\ } \trbra{\frac{2}{t}X^t+I_n}^2=(1-\frac{4\pi_0}{t^2})I_n=\theta I_n. \label{eq-theta}
\end{flalign*}
Multiplying $\frac{2}{t}X^t+I_n$ on both sides of \eqref{eq-m1}, we can obtain \begin{flalign*}
	    \begin{pmatrix}
		H_{2m} &0 \\ 0 &t
	\end{pmatrix}X =X^t\begin{pmatrix}
		H_{2m} &0 \\ 0 &t
	\end{pmatrix}.
\end{flalign*} 
Write $$X=\begin{pmatrix}
 	  X_1 &X_2\\ X_3 &x
 \end{pmatrix},$$ where $X_1\in M_{2m}(R)$, $X_2\in M_{2m\times 1}(R)$, $X_3\in M_{1\times 2m}(R)$ and $x\in R$. Equivalently, we obtain  \begin{gather*}
	    H_{2m}X_1=X_1^tH_{2m},\ H_{2m}X_2=tX_3^t.
\end{gather*}

\subsubsection{Condition \textbf{LM4}}
Recall $\sL=\frac{1}{t}\CO_{F_0}$. With respect to the basis \eqref{mbasis-1}, the induced $\sL\otimes_{\CO_{F_0}}R$-valued symmetric pairing on $\Lambda_m\otimes_{\CO_{F_0}}R$ is represented by the matrix \begin{flalign}
	  S_1=\begin{pmatrix}
		 \frac{2}{t}H_{2m} &0 &H_{2m} &0\\ 0 &2 &0 &t\\ H_{2m} &0 &\frac{2\pi_0}{t}H_{2m} &0\\ 0 &t &0 &2\pi_0
	\end{pmatrix}.   \label{S1mRU}
\end{flalign} 
The Condition \textbf{LM4} translates to \begin{gather*}
	 \matx^tS_1\matx=0 \text{\ and} \ \text{half of the diagonal  of $\matx^tS_1\matx$ equals zero}. 
\end{gather*} 
Equivalently, we obtain \begin{align*}
		             &\begin{pmatrix}
		             	\frac{2}{t} X_1^tH_{2m}X_1+2X_3^tX_3+H_{2m}X_1+X_1^tH_{2m}+\frac{2\pi_0}{t}H_{2m} &\frac{2}{t}X_1^tH_{2m}X_3+2xX_3^t+H_{2m}X_2+tX_3^t \\ \frac{2}{t}X_2^tH_{2m}X_1+2xX_3+tX_3+X_2^tH_{2m} &\frac{2}{t}X_2^tH_{2m}X_2+2x^2+2tx+2\pi_0
		             \end{pmatrix} =0, \\ &\text{half of the diagonal of $ \frac{2}{t} X_1^tH_{2m}X_1+2X_3^tX_3+H_{2m}X_1+X_1^tH_{2m}+\frac{2\pi_0}{t}H_{2m}$ equals $0$}, \\ &\half\trbra{\frac{2}{t}X_2^tH_{2m}X_2+2x^2+2tx+2\pi_0}=0.
	          \end{align*}
	          
\subsubsection{Condition \textbf{LM5}}
As $\pi\otimes 1$ acts as right multiplication by $X+tI_n$ on $\CF$, the wedge condition \textbf{LM5} on $\CF$ translates to \begin{flalign*}
	   \wedge^2(X+\pi I_n)=0.
\end{flalign*}

\subsubsection{A simplification of equations} As in the case $I=\cbra{0}$, we can simplify the above equations and obtain the following.
\begin{prop}\label{prop-Umruring}
	The scheme $\RU_{\cbra{m}}=\Spec\CO_F[X]/\CI$, where $\CI$ is the ideal generated by: \begin{align*}
	     &\tr(X+\pi I_n)-\pi+\ol{\pi},\ \wedge^2(X+\pi I_n),\ X_1^tH_{2m}-H_{2m}X_1, \ tX_3^t-H_{2m}X_2, \\  &\text{half of the diagonal  of $( \frac{2}{t} \tr(X_1+\pi I_{2m})+2\sqrt{\theta} )H_{2m} ({X_1}+\pi I_{2m})+2X_3^tX_3$}.
         \end{align*} 
\end{prop}

Set $$ \wt{X}_1\coloneqq X_1+\pi I_{2m},\ \wt{X}\coloneqq \begin{pmatrix}
	  \wt{X}_1 \\ X_3
\end{pmatrix}.$$ Then we have the following proposition.

\begin{prop}
	The scheme $\RU_{\cbra{m}}$ is isomorphic to $\Spec\CO_F[\wt{X}]/\wt{\CI}$, where $\wt{\CI}$ is the ideal generated by  \begin{align*}
	    &\wedge^2(\wt{X}),\
	    {H}_{2m}\wt{X}_1-\wt{X}_1^t{H}_{2m} ,\
	    \text{half of the diagonal  of $( \frac{2}{t} \tr(\wt{X}_1)+2\sqrt{\theta} )H_{2m} \wt{X}_1+2X_3^tX_3$. }
        \end{align*}
\end{prop}

\begin{defn}
	Denote by $\RU_{\cbra{m}}^\rfl$ the closed subscheme of $\RU_{\cbra{m}}=\Spec\CO_F[\wt{X}]/\wt{\CI}$ defined by the ideal $\wt{\CI}^\rfl\subset\CO_F[\wt{X}]$ generated by  
 \begin{flalign*}
	    &\wedge^2(\wt{X}),\  
	    {H}_{2m}\wt{X}_1-\wt{X}_1^t{H}_{2m} ,\   
	    (\frac{\tr(\wt{X}_1)}{t}+\sqrt{\theta} )H_{2m} \wt{X}_1+X_3^tX_3. 
\end{flalign*}
Note that $\wt{\CI}\subset\wt{\CI}^\rfl$.
\end{defn}

\subsection{Global results}
We first give the results for the schemes $\RU_{\cbra{m}}$ and $\RU_{\cbra{m}}^\rfl$.
\begin{prop}
	\label{mprop-Ufl}
    \begin{enumerate}
    	\item $\RU_{\cbra{m}}^\rfl$ is smooth over $\CO_F$ of relative dimension $n-1$. The special fiber is geometrically reduced and irreducible.     	
        \item $\RU_{\cbra{m}}$ and $\RU_{\cbra{m}}^\rfl$ have the same underlying topological space. 
    \end{enumerate} 
\end{prop}
\begin{proof}
	The proof of (2) is similar as that of Lemma \ref{prop-topflat}. Now we prove the smoothness of $\RU_{\cbra{m}}^\rfl$. We use the notation as in the proof of Lemma \ref{lem-smoothoutside}. In particular, \begin{flalign*}
	  \CR^\rfl &= \frac{\CO_F[\begin{psmallmatrix}
	     	Y\\ X_3
	     \end{psmallmatrix} ]}{\rbra{\wedge^2\begin{psmallmatrix}
	     	Y \\ X_3
	     \end{psmallmatrix}, Y-Y^t, (\frac{1}{t}\tr({H}_{2m}Y)+\sqrt{\theta})Y+X_3^tX_3   }}.
\end{flalign*}
Then one can similarly show that $D(y_{\ell\ell})$ for $1\leq \ell\leq 2m$ is smooth over $\CO_F$. Let $z\coloneqq \frac{1}{t}\tr(H_{2m}Y)+\sqrt{\theta}$. Note that \begin{flalign*}
	     Y=-z\inverse X_3^tX_3.
\end{flalign*}  in $\CR^\rfl[z\inverse]$.
   Thus, $Y$ is determined by $X_3$ and the scheme $D(z)=\Spec \CR^\rfl[z\inverse]\simeq \Spec \CO_F[X_3]$ is smooth over $\CO_F$. By Lemma \ref{lemcover}, we conclude that $\RU_{\cbra{m}}^\rfl$ is smooth over $\CO_F$. The special fiber is geometrically reduced by the smoothness. It is geometrically irreducible because the geometric special fibers of $D(z)$ and $D(y_{\ell\ell})$ for $1\leq \ell \leq 2m$ are irreducible.
\end{proof}

\begin{lem}\label{lemcover}
    The scheme $\RU_{\cbra{m}}^\rfl$ is covered by $D(z)$ and $D(y_{\ell\ell})$ for $1\leq \ell\leq 2m$.
\end{lem}
\begin{proof}
    It suffices to show that the closed subscheme $Z$ defined by the ideal $(z,y_{11},\ldots,y_{2m,2m})$ is empty. Suppose that $P$ is a $\bar k$-point $Z$. Then $P$ is determined by the elements $(y_{ij},x_k)_{1\leq i,j\leq 2m, 1\leq k\leq 2m}$ in $\bar k$, where $Y=(y_{ij})$ and $X_3=(x_k)$ satisfy the relations in $\wt{\CI}^\rfl$, and $z=y_{\ell\ell}=0$ for $1\leq \ell\leq 2m$. Since we have \begin{flalign*}
        y_{ij}=y_{ji} \text{\ and\ } y_{ii}y_{jj}=y_{ij}y_{ji},
    \end{flalign*}
    we obtain $y_{ij}^2=0$ in $\bar k$ for $1\leq i,j\leq 2m$. Hence, $Y=0$. It follows that \[z=\frac{1}{t}\tr(H_{2m}Y)+\sqrt{\theta} =\sqrt{\theta}\neq 0, \] which is a contradiction. 
\end{proof}

Recall that $(\Lambda_m,q,\sL)$ is a hermitian quadratic module over $\CO_{F_0}$ for $\sL=\frac{1}{t}\CO_{F_0}$. Let $$\sH_{\cbra{m}}\coloneqq \ud{\Sim}((\Lambda_m,q,\sL))$$ be the group scheme over $\CO_{F_0}$ of similitude automorphisms of $(\Lambda_m,q,\sL)$. By Theorem \ref{thmsimm}, $\sH_{\cbra{m}}$ is an affine smooth group scheme over $\CO_{F_0}$. As in Lemma \ref{lemGstable}, the group scheme $\sH_{\cbra{m}}$ acts on $\RM_{\cbra{m}}$.

\begin{lemma}\label{lem-Gorbitm}
	Let $\ol{k}$ be the algebraic closure of the residue field $k$. Then $\RM_{\cbra{m}}\otimes_{\CO_F}\ol{k}$ has two $\sH_{\cbra{m}}\otimes_{\CO_{F_0}}\ol{k}$-orbits, one of which consists of the worst point. 
\end{lemma}
\begin{proof}
	Let $\CF\in\RM_{\cbra{m}}(\ol{k})$. In particular, the subspace $\CF\sset(\Lambda_m\otimes_{\CO_{F_0}}\ol{k})$ is  an $n$-dimensional $\ol{k}$-vector space. The wedge condition \textbf{LM5} in this case becomes $\wedge^2(\pi\otimes 1\ |\ \CF)=0$. Therefore, the image $(\pi\otimes 1)\CF$ is at most one dimensional. We have the following two cases.
	
	Suppose $(\pi\otimes 1)\CF=0$. Then $\CF=(\pi\otimes 1)(\Lambda_m\otimes_{\CO_{F_0}}\ol{k})$, namely, $\CF$ is the worst point.
	
	Suppose $(\pi\otimes 1)\CF$ is one dimensional. Then there exists a vector $v\in\CF$ such that $(\pi\otimes 1)v$ generates $(\pi\otimes 1)\CF$. For simplicity, write $\pi$ for $\pi\otimes 1$. Let $f: (\Lambda_m\otimes_{\CO_{F_0}}\ol{k})\times(\Lambda_m\otimes_{\CO_{F_0}}\ol{k})\ra \sL\simeq\ol{k}$ denote the associated symmetric pairing on $\Lambda_m\otimes_{\CO_{F_0}}\ol{k}$. As in the proof of Lemma \ref{lem-Gorbit}, we see that $\pi v$ is not in the radical of the paring $f$, because $q(v)=0$. Then we can find some $w\in \Lambda_m\otimes_{\CO_{F_0}}\ol{k}$ such that $f(w,\pi v)\neq 0$ in $\ol{k}$. By rescaling, we may assume that $f(w,\pi v)=1$. Similar arguments in Lemma \ref{lem-Gorbit} imply that after some linear transformations, we may assume 
	\begin{flalign*}
		\text{$q(w)=q(v)=f(w,v)=0$ and $f(w,\pi v)=1$.}
	\end{flalign*}  
	Let $W_1\coloneqq \pair{v,\pi v,w,\pi w}$. Then $f$ restricts to a perfect symmetric pairing on $W_1$. Now we can write \begin{flalign}
		    \Lambda_m\otimes_{\CO_{F_0}}\ol{k} =W_1\oplus W, \label{spaceWm}
	\end{flalign} where $W$ is the orthogonal complement of $W_1$ with respect to $f$ whose dimension is $2n-4$ over $\ol{k}$. Since $q(\CF)=0$, we have $\CF\cap \pair{w,\pi w}=0$.  Hence, we obtain that the $\ol{k}$-dimension of $\CF\cap W$ is $n-2$ and $\CF\cap W\sset\pi W=\ker(\pi\ |\ W)$. Therefore, $\CF\cap W=\pi W$ for dimension reasons. 
    Note that the space $W$ carries a structure of hermitian quadratic module. By \eqref{spaceWm}, we have  \begin{flalign*}
		    \disc'(q) = \disc(q|_{W_1})\disc'(q|_W).
	\end{flalign*}
	Here $\disc'(q)$ is the divided discriminant in the sense of Definition \ref{app-defn1}, and we view it as an element in $\ol{k}$ by using a basis of $\Lambda_m\otimes_{\CO_{F_0}}\ol{k}$. By Example \ref{app-exlambm}, we have $\disc'(q)\in\ol{k}\cross$.
	Since $\phi$ is perfect on $W_1$, we obtain that $\disc(\phi|_{W_1})\in\ol{k}\cross$, and hence $\disc'(q|_W)\in \ol{k}\cross$. 
	In particular, $W$ is a hermitian quadratic module of type $\Lambda_m$ over $\ol{k}$ in the sense of Definition \ref{app-defn1}. Applying Theorem \ref{thmstand} to $W$ and using similar arguments as in the proof Lemma \ref{lem-Gorbit}, we can conclude that points $\CF\in\RM_{\cbra{m}}(\ol{k})$ with $\dim_{\ol{k}} \pi\CF=1$ are in the same orbit under the action of $\sH_{\cbra{m}}\otimes_{\CO_{F_0}}\ol{k}$.
\end{proof}

As $\RU^\rfl_{\cbra{m}}$ is flat over $\CO_F$, we may view $\RU^\rfl_{\cbra{m}}$ as an open subscheme of $\RM^\loc_{\cbra{m}}$ containing the worst point. By Lemma \ref{lem-Gorbitm}, the  $\sH_{\cbra{m}}$-translation of $\RU^\rfl_{\cbra{m}}$ covers $\RM^\loc_{\cbra{m}}$. Together with Proposition \ref{mprop-Ufl}, we have proven Theorem \ref{thm-intromain} and \ref{thm16} in the case $I=\cbra{m}$ and (R-U).


\section{The case $I=\cbra{m}$ and (R-P)}\label{sec-mrp}
In this section, we consider the case when $F/F_0$ is of (R-P) type and $I=\cbra{m}$. In particular, we have $$\pi^2+\pi_0=0 \text{\ and\ }\pi+\ol{\pi}=0. $$ 
Consider the following ordered $\CO_{F_0}$-basis of $\Lambda_m$ and $\Lambda_m^s$: \begin{alignat}{2}
	   \Lambda_m&: \half e_{m+2},\ldots,\half e_n,\pi\inverse e_1,\ldots,\pi\inverse e_m,e_{m+1},\frac{\pi}{2}e_{m+2},\ldots,\frac{\pi}{2}e_n, e_1,\ldots,e_m,\pi e_{m+1}, \label{rpmbasis-1}  \\  \Lambda_m^s&: e_{m+2},\ldots,  e_n,  \frac{2}{\pi}e_1,\ldots,\frac{2}{\pi}e_m, \pi\inverse e_{m+1}, \pi e_{m+2},\ldots, \pi e_n, 2e_1,\ldots,2e_m,  e_{m+1}. \label{rpmbasis-2}
\end{alignat}
Recall $(\Lambda_m,q,\sL)$ is a hermitian quadratic module for $\sL=2\inverse \CO_{F_0}$.

\subsection{A refinement of $\RM_{\cbra{m}}^\naive$ in the (R-P) case}
\begin{defn}\label{rpmdefn-M}
	Let $\RM_{\cbra{m}}$ be the functor $$\RM_{\cbra{m}}: (\Sch/\CO_{F})^\op\lra \Sets$$ which sends an $\CO_F$-scheme $S$ to the set of $\CO_S$-modules $\CF$ such that 
	\begin{enumerate}[label=\textbf{LM\arabic*}]
		\item ($\pi$-stability condition) $\CF$ is an $\CO_F\otimes_{\CO_{F_0}}\CO_S$-submodule of $\Lambda_m\otimes_{\CO_{F_0}}\CO_S$ and as an $\CO_S$-module, it is a locally direct summand of rank $n$.
		\item (Kottwitz condition) The action of $\pi\otimes 1\in\CO_F\otimes_{\CO_{F_0}}\CO_S$ on $\CF$ has characteristic polynomial $$\det(T-\pi\otimes 1\ |\ \CF)=(T-\pi)(T-\ol{\pi})^{n-1}.$$
		\item Let $\CF^\perp$ be the orthogonal complement in $\Lambda_m^s\otimes_{\CO_{F_0}}\CO_S$ of $\CF$ with respect to the perfect pairing $$ s(-,-): (\Lambda_m\otimes_{\CO_{F_0}}\CO_S) \times (\Lambda_m^s\otimes_{\CO_{F_0}}\CO_S) \ra \CO_S.$$ We require the map $\Lambda_m\otimes_{\CO_{F_0}}\CO_S \ra (2\inverse \Lambda_m^s)\otimes_{\CO_{F_0}}\CO_S$ induced by $\Lambda_m\hookrightarrow 2\inverse\Lambda_m^s$ sends $\CF$ to $2\inverse\CF^\perp,$ where $2\inverse\CF^\perp$ denotes the image of $\CF^\perp$ under the isomorphism $2\inverse:\Lambda_m^s\otimes_{\CO_{F_0}}\CO_S \simto 2\inverse \Lambda_m^s\otimes_{\CO_{F_0}}\CO_S$.  
		\item (Hyperbolicity condition) The quadratic form $q: \Lambda_m\otimes_{\CO_{F_0}}\CO_S\ra \sL\otimes_{\CO_{F_0}}\CO_S$ induced by $q:\Lambda_m\ra \sL$ satisfies $q(\CF)=0$. 
		\item (Wedge condition) The action of $\pi\otimes 1-1\otimes\ol{\pi}\in \CO_F\otimes_{\CO_{F_0}}\CO_S$ satisfies $$\wedge^2(\pi\otimes 1-1\otimes \ol{\pi}\ |\ \CF )=0.$$
		\item (Strengthened spin condition) The line $\wedge^n\CF\sset W(\Lambda_m)\otimes_{\CO_F}\CO_S$ is contained in $$\Im\rbra{W(\Lambda_m)^{n-1,1}_{- 1}\otimes_{\CO_F}\CO_S\ra W(\Lambda_m)\otimes_{\CO_F}\CO_S }.$$	
		     Here we use similar notations as in \S \ref{subsubsec-spin}.
	\end{enumerate}
\end{defn} 
Then $\RM_{\cbra{m}}$ is representable and we have closed immersions $$\RM^\loc_{\cbra{m}}\subset\RM_{\cbra{m}}\subset\RM_{\cbra{m}}^\naive$$ of projective schemes over $\CO_F$, where all schemes have the same generic fiber.

\subsection{An affine chart $\RU_{\cbra{m}}$ around the worst point}
Set $$\CF_0\coloneqq (\pi\otimes 1)(\Lambda_m\otimes_{\CO_{F_0}}k).$$ Then we can check that $\CF_0\in\RM_{\cbra{m}}(k)$. We call it the worst point of $\RM_{\cbra{m}}$. 
With respect to the basis \eqref{mbasis-1}, the standard affine chart around $\CF_0$ in $\Gr(n,\Lambda_m)_{\CO_F}$ is the $\CO_F$-scheme of $2n\times n$ matrices $\left(\begin{smallmatrix}
	X\\ I_n
\end{smallmatrix}\right)$. 
As in \S \ref{subsec6.2}, we can obtain an affine neighborhood $\RU_{\cbra{m}}$ in $\RM_{\cbra{m}}$ containing $\CF_0$. We will analyze the conditions \textbf{LM1-6} and obtain the affine coordinate ring of $\RU_{\cbra{m}}$ in Proposition \ref{prop-Umrpring}. 


\subsubsection{Condition \textbf{LM1}} 
Let $R$ be an $\CO_F$-algebra. With respect to the basis \eqref{rpmbasis-1}, the operator $\pi\otimes 1$ acts on $\Lambda_m\otimes_{\CO_{F_0}}R$ via the matrix \begin{flalign*}
	    \begin{pmatrix}
	    	0 &-\pi_0I_n\\ I_n &0
	    \end{pmatrix}.
\end{flalign*}
Then the $\pi$-stability condition \textbf{LM1} on $\CF$ means there exists an $n\times n$ matrix $P\in M_{n\times n}(R)$ such that \begin{flalign*}
	   \begin{pmatrix}
	    	0 &-\pi_0I_n\\ I_n &0
	    \end{pmatrix} \begin{pmatrix}
	    	X\\ I_n
	    \end{pmatrix} =\begin{pmatrix}
	    	X\\ I_n
	    \end{pmatrix}P.
\end{flalign*}
We obtain $P=X$ and $X^2+\pi_0I_n=0$. 

\subsubsection{Condition \textbf{LM2}} We have already shown that $\pi\otimes 1$ acts on $\CF$ via right multiplication by $X$. Then the Kottwitz condition \textbf{LM2} translates to 
\begin{flalign}
	    \tr(X+\pi I_n)=\pi-\ol{\pi}=2\pi,\ \tr\rbra{\wedge^i(X+\pi I_n)}=0, \text{\ for $i\geq 2$}. 
\end{flalign} 

\subsubsection{Condition \textbf{LM3}} 
 With respect to the bases \eqref{rpmbasis-1} and \eqref{rpmbasis-2}, the perfect pairing $$s(-,-):(\Lambda_m\otimes_{\CO_{F_0}}R)\times (\Lambda_m^s\otimes_{\CO_{F_0}}R)\ra R  $$ and the map $\Lambda_m\otimes_{\CO_{F_0}} \ra \frac{{1}}{2}\Lambda_m^s\otimes_{\CO_{F_0}}R$  are represented respectively by the matrices \begin{flalign*}
	   S=\begin{pmatrix}
	   	  0 &0 &J_{2m} &0 \\ 0 &0 &0 &1\\ -J_{2m} &0 &0 &0 \\ 0 &-1 &0 &0
	   \end{pmatrix} \text{\ and\ } N= \begin{pmatrix}
     	  I_{2m} &0 &0 &0\\ 0 &0 &0 &-2\pi_0\\ 0 &0 &I_{2m} &0\\ 0 &2 &0 &0
     \end{pmatrix},
	 \end{flalign*}  
	 \text{where $J_{2m}\coloneqq \begin{pmatrix}
	0 &H_m\\ -H_m &0
\end{pmatrix}$}. 

Then the Condition \textbf{LM3} translates to $\matx^tS\rbra{N\matx }=0,$ or equivalently,
\begin{flalign}
	    \matx^t\begin{pmatrix}
		 0 &0 &J_{2m} &0\\ 0 &2 &0 &0\\ -J_{2m} &0 &0 &0\\ 0 &0 &0 &2\pi_0
	\end{pmatrix}\matx &=0. \label{rpmeq-lm41}  
\end{flalign}
 Write $$X=\begin{pmatrix}
 	  X_1 &X_2\\ X_3 &x
 \end{pmatrix},$$ where $X_1\in M_{2m}(R)$, $X_2\in M_{2m\times 1}(R)$, $X_3\in M_{1\times 2m}(R)$ and $x\in R$. The Equation \eqref{rpmeq-lm41} translates to \begin{flalign*}
	   \begin{pmatrix}
	   	  2X_3^tX_3+X_1^tJ_{2m}-J_{2m}X_1 &2xX_3^t-J_{2m}X_2\\ 2xX_3+X_2^tJ_{2m} &2x^2+2\pi_0
	   \end{pmatrix} =0.
\end{flalign*}

\subsubsection{Condition \textbf{LM4}}
Recall $\sL=\frac{1}{2}\CO_{F_0}$. With respect to the basis \eqref{rpmbasis-1}, the induced $\sL\otimes_{\CO_{F_0}}R$-valued symmetric pairing on $\Lambda_m\otimes_{\CO_{F_0}}R$ is represented by the matrix \begin{flalign}
	  S_1=\begin{pmatrix}
		 0 &0 &J_{2m} &0\\ 0 &2 &0 &0\\ -J_{2m} &0 &0 &0\\ 0 &0 &0 &2\pi_0
	\end{pmatrix}. \label{S1mRP}
\end{flalign} 
The Condition \textbf{LM4} translates to \begin{gather*}
	 \matx^tS_1\matx=0 \text{\ and\ } \text{half of the diagonal of $\matx^tS_1\matx$ equals zero}. 
\end{gather*} 
Equivalently, we obtain \begin{flalign*}
	   &\begin{pmatrix}
	   	  2X_3^tX_3+X_1^tJ_{2m}-J_{2m}X_1 &2xX_3^t-J_{2m}X_2\\ 2xX_3+X_2^tJ_{2m} &2x^2+2\pi_0
	   \end{pmatrix} =0, \\ &x^2+\pi_0=0,\\ &\text{half of the diagonal of $2X_3^tX_3+X_1^tJ_{2m}-J_{2m}X_1$ equals zero}.
\end{flalign*}

\subsubsection{Condition \textbf{LM5}}
Since $\pi\otimes 1$ acts as right multiplication by $X$ on $\CF$, the wedge condition \textbf{LM5} on $\CF$ translates to \begin{flalign*}
	   \wedge^2(X+\pi I_n)=0.
\end{flalign*}

\subsubsection{Condition \textbf{LM6}}
As in \S \ref{subsubsec-strspin}, the strengthened spin condition \textbf{LM6} in this case implies that \begin{flalign*}
	    X_1=J_{2m}X_1^tJ_{2m},\ 2\pi X_3^t=J_{2m}X_2.
\end{flalign*}

\subsubsection{A simplification of equations} 
As in the case $I=\cbra{0}$, we can simplify the above equations and obtain the following.
\begin{prop}\label{prop-Umrpring}
  The scheme $\RU_{\cbra{m}}$ is a closed subscheme of $\RU'_{\cbra{m}}\coloneqq \Spec\CO_F[X]/\CI$, where $\CI$ is the ideal generated by: \begin{align*}
	     &\tr(X+\pi I_n)-2\pi,\  \wedge^2(X+\pi I_n),\  X_1^tJ_{2m}+J_{2m}X_1, \ 2\pi X_3^t-J_{2m}X_2, \\  &\text{half of the diagonal of $2X_3^tX_3+X_1^tJ_{2m}-J_{2m}X_1$}.
        \end{align*} 
\end{prop}

Set $$ \wt{X}_1\coloneqq X_1+\pi I_{2m},\ \wt{X}\coloneqq \begin{pmatrix}
	  \wt{X}_1 \\ X_3
\end{pmatrix}.$$ 
As $X_2$ and $x$ are determined by $X_1$ and $X_3$ by relations in $\CI$, we obtain the following proposition.

\begin{prop}
   The scheme $\RU_{\cbra{m}}'$ is isomorphic to $\Spec\CO_F[\wt{X}]/\wt{\CI}$, where $\wt{\CI}$ is the ideal generated by:  
   \begin{align*}
	    &\wedge^2(\wt{X}),\
	    {J}_{2m}\wt{X}_1+ \wt{X}_1^t{J}_{2m} ,\
	    \text{half of the diagonal of $2X_3^tX_3+\wt{X}_1^tJ_{2m}-J_{2m}\wt{X}_1$. }
\end{align*}
\end{prop}

\begin{defn}
	Denote by $\RU_{\cbra{m}}^\rfl$ the closed subscheme of $\RU'_{\cbra{m}}=\Spec\CO_F[\wt{X}]/\wt{\CI}$ defined by the ideal $\CI^\rfl\sset\CO_F[\wt{X}]$ generated by:  
 \begin{align*}
	    &\wedge^2(\wt{X}),\  
	    {J}_{2m}\wt{X}_1+\wt{X}_1^t{J}_{2m} ,\    
	    X_3^tX_3+\wt{X}_1^tJ_{2m}. 
\end{align*}
Note that $\wt{\CI}\subset\wt{\CI}^\rfl$.
\end{defn}

\subsection{Global results}
We first give results for the schemes $\RU_{\cbra{m}}$ and $\RU_{\cbra{m}}^\rfl$.
\begin{prop}
	\label{rpmprop-Ufl}
    \begin{enumerate}
    	\item $\RU_{\cbra{m}}^\rfl$ is smooth over $\CO_F$ of relative dimension $n-1$ with geometrically integral special fiber.     	
        \item $\RU_{\cbra{m}}$ and $\RU_{\cbra{m}}^\rfl$ have the same underlying topological space. 
    \end{enumerate} 
\end{prop}
\begin{proof}
	The proof of (2) is similar as that of Lemma \ref{prop-topflat}. Now we prove the smoothness of $\RU_{\cbra{m}}^\rfl$. It is clear from the expression of $\wt{\CI}^\rfl$ that $\wt{X}_1$ is determined by $X_3$, and hence, \begin{flalign*}
		     \CO_F[\wt{X}]/\wt{\CI}^\rfl \simeq \Spec\CO_F[X_3] \simeq {\BA}_{\CO_F}^{n-1},
	\end{flalign*} which is smooth over $\CO_F$ of relative dimension $n-1$. The special fiber of $\RU_{\cbra{m}}^\rfl$ is isomorphic to $\BA_k^{n-1}$, which is geometrically integral. 
\end{proof}

As $\RU^\rfl_{\cbra{m}}$ is flat over $\CO_F$, we may view $\RU^\rfl_{\cbra{m}}$ as an open subscheme of $\RM^\loc_{\cbra{m}}$ containing the worst point.
Then as in Lemma \ref{lem-Gorbitm}, we can show that the special fiber $\RM_{\cbra{m}}\otimes_{\CO_F}\ol{k}$ has only two orbits under the action of $\sH_{\cbra{m}}\otimes_{\CO_{F_0}}\ol{k}$. Together with Proposition \ref{rpmprop-Ufl}, we deduce Theorem \ref{thm-intromain} and \ref{thm16} in the case $I=\cbra{m}$ and (R-P).


\section{Comparison with the $v$-sheaf local models} \label{sec-compa}
In this section, we will relate the local model $\RM^\loc_I$ for $I=\cbra{0}$ or $\cbra{m}$ to the v-sheaf local models considered in \cite[\S 21.4]{scholze2020berkeley} and \cite{anschutz2022p}. Let $G$ be \dfn{any} connected reductive group over a complete discretely valued field $L/\BQ_p$, where $p$ is \dfn{any} prime. Let $\CO_L$ be the ring of integers of $L$. Let $\sG$ be a parahoric group scheme over $\CO_L$ of $G$. Then we can form the Beilinson-Drinfeld Grassmannian $\Gr_\sG$, which is a v-sheaf over $\CO_L$. We have the following properties. 
\begin{thm} \label{thm-Gr}
	\begin{enumerate}
		\item The structure morphism $\Gr_\sG\lra \Spd\CO_L$ is ind-proper and ind-representable in spatial diamonds. The generic fiber of $\Gr_\sG$ can be naturally identified with the $B_\dR^+$-affine Grassmannian $\Gr_G$.
         \item If $\sG\hookrightarrow \sH$ is a closed immersion of parahoric group schemes, then the induced morphism $\Gr_\sG\ra \Gr_\sH$ is a closed immersion.
	\end{enumerate}
\end{thm}
\begin{proof}
	See \cite[Proposition 20.3.6, Proposition 20.5.4, Theorem 21.2.1]{scholze2020berkeley}, or \cite[Theorem 4.9, Lemma 4.10]{anschutz2022p}.
\end{proof}

 Recall that the $B_\dR^+$-affine Grassmannian $\Gr_G$ is a union of (open) Schubert diamonds $\Gr_{G,\cbra{\mu}}^\circ$ indexed by geometric conjugacy classes $\cbra{\mu}$ of cocharacters of $G$. Let $\Gr_{G,\cbra{\mu}}$ denote the v-closure of $\Gr_{G,\cbra{\mu}}^\circ$. If $\cbra{\mu}$ is minuscule with reflex field $E$, then $\Gr_{G,\cbra{\mu}}$ is representable by a projective scheme over $E$ (see \cite[Proposition 19.4.2]{scholze2020berkeley}). More precisely, $\Gr_{G,\cbra{\mu}}$ is the associated diamond of the flag variety $\sFl_{G,\cbra{\mu}}\coloneqq  G/P_{\cbra{\mu}}$, see \cite[Proposition 19.4.1]{scholze2020berkeley} for the normalization of the parabolic subgroup $P_{\cbra{\mu}}$. 
 
\begin{defn}
	Let $\Gr_{\sG,\CO_E}$ be the base change of $\Gr_\sG$. The v-sheaf local model $\RM^v_{\sG,\cbra{\mu}}$ is defined to be the v-closure of $\Gr_{G,\cbra{\mu}}$ inside $\Gr_{\sG,\CO_E}$.
\end{defn}

Recall that given a scheme $X$ proper over $\CO_E$, there is a functorially associated v-sheaf $X^\Diamond$ over $\Spd\CO_E$. For details of the definition, we refer to \cite[\S 2.2]{anschutz2022p}. We have the following representability result of the v-sheaf local models.

\begin{thm}[Scholze-Weinstein Conjecture]  \label{thm-vLM}
	Assume $\cbra{\mu}$ is minuscule. Then there exists a unique (up to unique isomorphism) flat, projective and normal $\CO_E$-scheme $\BM_{\sG,\cbra{\mu}}$ with a closed immersion \begin{flalign*}
		      \BM_{\sG,\cbra{\mu}}^\Diamond\hookrightarrow \Gr_{\sG}\otimes_{\CO_L}\CO_E
	\end{flalign*}
	prolonging $\sFl_{G,\cbra{\mu}}^\Diamond\simto \Gr_{G,\cbra{\mu}}\sset\Gr_{G}\otimes_LE$. In particular, $\BM_{\sG,\cbra{\mu}}^\Diamond=\RM^v_{\sG,\cbra{\mu}}$.
\end{thm}
\begin{proof}
	See \cite[Theorem 1.1]{anschutz2022p} and \cite[Corollary 1.4]{gleason2022tubular}.
\end{proof}

Now we return to the situation in \S \ref{subsec-notation}. In particular, $G$ is the unitary similitude group $\GU(V,h)$ over $F_0$ attached to a split hermitian $F/F_0$-vector space $(V,h)$ of dimension $n=2m+1\geq 3$, and there is an $F$-basis $(e_i)_{1\leq i\leq n}$ of $V$ such that $h(e_i,e_j)=\delta_{i,n+1-j}$ for $1\leq i,j\leq n$. Let $\sG$ be the (special) parahoric group scheme corresponding to the index set $I=\cbra{0}$ or $\cbra{m}$. Let $T$ be the maximal torus of $G$ consisting of diagonal matrices with respect to the basis $(e_i)_{1\leq i\leq n}$. Under the isomorphism \begin{flalign*}
	     G_F\simeq \GL_{n,F}\times\BG_{m,F},
\end{flalign*} we can identify $X_*(T)$ with $\BZ^n\times\BZ$. Let $\mu\coloneqq \mu_{n-1,1}\in X_*(T)$ be the (minuscule) cocharacter corresponding to $$(1,0^{(n-1)},1)\in \BZ^n\times\BZ.$$ We write $0^{(n-1)}$ for a list of $n-1$ copies of $0$. As $n-1\neq 1$, the reflex field $E$ of $\cbra{\mu}$ equals $F$ (see \cite[\S 1.5]{pappas2009local})). Let $\RM^\loc$ denote the local model $\RM_I^\loc$ for $I=\cbra{0}$ or $\cbra{m}$ considered in \S 3.3.
\begin{thm} \label{prop-M=Mv}
	The scheme $\RM^\loc$ is isomorphic to $\BM_{\sG,\cbra{\mu}}$ in Theorem \ref{thm-vLM}.
\end{thm}
\begin{proof}
	We have shown that the scheme $\RM^\loc$ is normal, flat and projective over $\CO_F$. By the uniqueness part of Theorem \ref{thm-vLM}, it suffices to show  that \begin{flalign*}
		     \RM_{\sG,\cbra{\mu}}^v =\RM^{\loc,\Diamond}.
	\end{flalign*} 
	
	By our concrete description of $\sG$ in Corollary \ref{app-coroparam} and \ref{app-coropara0}, we have a closed immersion over $\CO_{F_0}$ \begin{flalign}
		    \sG\hookrightarrow \GL(\Lambda)\simeq\GL_{2n}  \label{cl}
	\end{flalign} prolonging the closed immersion $G\hookrightarrow \GL_{F_0}(V)\simeq\GL_{2n,F_0}$, where $\Lambda$ is either $\Lambda_0$ or $\Lambda_m$ depending on what $\sG$ is. Let $T'$ be the maximal torus of $\GL_{2n,F_0}$ consisting of diagonal matrices. Then the map $G\hookrightarrow \GL_{F_0}(V)$ transports $\cbra{\mu_{n-1,1}}$ to the geometric conjugacy class $\cbra{\mu_n}$ of cocharacters of $T'$. Here, $\mu_n$ corresponds to $(1^{(n)},0^{(n)})\in X_*(T')\simeq \BZ^{2n}.$ 
	By Theorem \ref{thm-Gr} (2), the closed immersion \eqref{cl} induces a closed immersion \begin{flalign*}
		    \RM^v_{\sG,\cbra{\mu}} \hookrightarrow \RM^v_{\GL_{2n},\cbra{\mu_n}}\otimes_{\CO_{F_0}}\CO_F =\Gr(n,2n)_{\CO_F}^\Diamond,
	\end{flalign*} and we may identify $\RM^v_{\sG,\cbra{\mu}}$ with the v-closure of $\sFl_{G,\cbra{\mu}}^\Diamond$ inside $\Gr(n,2n)_{\CO_F}^\Diamond$.
	
	By Lemma \ref{lem-generic}, we can identify the generic fiber $\RM^\loc\otimes_{\CO_F}F$ with $\BP^{n-1}_F\simeq \sFl_{G,\cbra{\mu}}$, and there exists a closed immersion $$\sFl_{G,\cbra{\mu}}\hookrightarrow \sFl_{\GL_{2n},\cbra{\mu_n},F}=\Gr(n,2n)_F$$ induced by the embedding $G\hookrightarrow \GL_{F_0}(V)$. By our construction of $\RM^\loc$, the scheme $\RM^\loc$ is the Zariski closure of $\sFl_{G,\cbra{\mu}}$ along $\sFl_{G,\cbra{\mu}}\hookrightarrow \sFl_{\GL_{2n},\cbra{\mu_n},F}\hookrightarrow \Gr(n,2n)_{\CO_F}$.
	Applying the diamond functor, we see that $\RM^{\loc,\Diamond}$ is the v-closure of $\sFl_{G,\cbra{\mu}}^\Diamond$ inside $\Gr(n,2n)_{\CO_F}^\Diamond$. Hence, we have $\RM_{\sG,\cbra{\mu}}^v =\RM^{\loc,\Diamond}$.
\end{proof}

\begin{remark}
	The proof of the above proposition also gives another proof of the representability of the v-sheaf local model $\RM^v_{\sG,\cbra{\mu}}$ in our setting.
\end{remark}

\appendix 

\section{Normal forms of hermitian quadratic modules} \label{appB}
Let us keep the notations as in \S \ref{sec-main}. In this appendix, we will show that, under certain conditions, hermitian quadratic modules \etale locally have a normal form up to similitude. This is a variant of \cite[Theorem 3.16]{rapoport1996period} in our setting. Such a result will be important when we relate the local models to Shimura varieties.  

In the following, we let $$\Nilp\coloneqq \Nilp_{\CO_{F_0}}$$ denote the category of noetherian\footnote{If $R$ is noetherian, then a finitely generated $R$-module $M$ is projective if and only if there exists a finite Zariski open cover $\tcbra{\Spec R_i}_{i\in I}$ of $\Spec R$ such that $M_{R_i}$ is free. } $\CO_{F_0}$-algebras such that $\pi_0$ is nilpotent. We set $t\coloneqq \pi+\ol{\pi}$. In particular, $t=0$ if $F/F_0$ is of (R-P) type. For an $\CO_{F_0}$-algebra $R$ and $a\in \CO_{F}$, we will simply use $a$ to denote the element $a\otimes 1$ in $\CO_F\otimes_{\CO_{F_0}}R$, if there is no confusion. For a hermitian quadratic module $(M,q,\sL)$, we will use $f$ to denote the associated symmetric pairing on $M$, as in Definition \ref{defn-hermqm}. 


\subsection{Hermitian quadratic modules of type $\Lambda_m$}
The results in this subsection are essentially contained in \cite[\S 9]{anschutz2018extending}, with some modifications to the proof. 

\begin{lemma}[{cf. \cite[Lemma 9.6]{anschutz2018extending}}] \label{lem-trick}
	Let $R\in\Nilp$. Let $(M,q,R)$ be an $R$-valued hermitian quadratic module over $R$. Assume there exist $v,w\in M$ such that $f(v,\pi w)=1$ in $R$. Then there exist $v',w'$ in the $R$-submodule spanned by $\cbra{v,w,\pi v,\pi w} $ such that \begin{flalign*}
		     q(v')=q(w')=f(v',w')=0\ \text{and\ }f(v',\pi w')=1.
	\end{flalign*}
\end{lemma}
\begin{proof}
	For $r\in R$, we have \begin{flalign*}
		 q(v+r\pi w) &=q(v)+rf(v,\pi w)+r^2\pi_0q(w) =(\pi_0q(w))r^2+r+q(v),
	\end{flalign*}
	which can be viewed as a quadratic function of $r$. As $4\pi_0$ is nilpotent on $R$ by assumption, there exists a sufficiently large integer $N$ such that the sum
	\begin{flalign*}
		      1-2\pi_0q(v)q(w) +2\pi_0^2q(v)^2q(w)^2 +\cdots +(-1)^N \binom{1/2}{N}4^N\pi_0^Nq(v)^Nq(w)^N
	\end{flalign*}
	in $R$ is a square root of ${1-4\pi_0q(v)q(w)}$.
	Note that $\binom{1/2}{N}4^N$ lies in $R$ by a direct computation of the $2$-adic valuation. In particular, \begin{flalign*}
		    r_0\coloneqq \frac{-1+\rbra{1-4\pi_0q(v)q(w)}^{1/2}}{2\pi_0q(w)}\in R, 
	\end{flalign*} 
	and it is a solution for the quadratic equation $q(v+r\pi w)=0$. Replacing $v$ by $v+r_0\pi w$, we may assume $q(v)=0$. Similarly, we may assume $q(w)=0$ by replacing $w$ by $w+r\ol{\pi}v$ for suitable $r$ in $R$. 
	
	Set $r_1\coloneqq \rbra{1-f(x,y)f(v,\pi^2w)}\inverse$ and $r_2\coloneqq -r_1f(v,w)$. Note that \begin{flalign*}
		   f(v,\pi^2w) &=f(v,(t\pi-\pi_0)w) =tf(v,\pi w)-\pi_0f(v,w)=t-\pi_0f(v,w)
	\end{flalign*} 
	is nilpotent in $R$, so $r_1$ indeed exists in $R$. 
	Set $v'\coloneqq r_1v+r_2\ol{\pi}v$. 
	Then the straightforward computation implies that \begin{flalign*}
		    f(v',w) &= r_1f(v,w)+r_2f(\ol{\pi}v,w) =r_1f(v,w)+r_2f(v,\pi w) =r_1f(v,w)+r_2=0
	\end{flalign*}
	and \begin{flalign*}
		  f(v',\pi w) &=r_1f(v,\pi w)+r_2f(\ol{\pi}v,\pi w) =r_1+r_2f(v,\pi^2w) =1.
	\end{flalign*}
\end{proof}

\begin{lemma}\label{lem-perfect}
	Let $R$ be an $\CO_{F_0}$-algebra and $M$ be a finite free $\CO_F\otimes_{\CO_{F_0}}R$-module of rank $d\geq 1$. Suppose $b: M\times M\ra R$ is a perfect $R$-bilinear pairing. 
	Then there exists $v, w\in M$ such that $b(v,\pi w)=1$.
\end{lemma}
\begin{proof}
	By assumption, we may choose an $R$-basis $\cbra{v_1,\ldots,v_{2d}}$ of $M$ such that $v_{d+i}=\pi v_i$ for $1\leq i\leq d$. This basis yields a dual basis $\cbra{v_1^\vee,\ldots,v_{2d}^\vee}$ of $M^\vee\coloneqq \Hom_R(M,R)$ such that $v_i^\vee(v_j)=b(v_i,v_j)=\delta_{ij}$.
	Since $b$ is perfect, we can find elements $\cbra{w_1,\ldots,w_{2d}}$ in $M$ such that \begin{flalign*}
          b(w_i,v_j)=v_i^\vee(v_j)=\delta_{ij}
	\end{flalign*} for $1\leq i,j\leq 2d$. Set $v\coloneqq w_{d+1}$ and $w\coloneqq v_1$. Then we have \begin{flalign*}
		    b(v,\pi w)=b(w_{d+1},v_{d+1})=v_{d+1}^\vee(v_d)=1.
	\end{flalign*}
\end{proof}

\begin{lemma} \label{lem-orthogo}
	Let $R$ be an $\CO_{F_0}$-algebra and $M$ be a finite free $\CO_F\otimes_{\CO_{F_0}}R$-module of rank $d\geq 1$. Suppose $b:M\times M\ra R$ is an $R$-bilinear pairing on $M$ such that \begin{flalign}
		   b(\pi m_1,m_2)=b(m_1,\ol{\pi}m_2) \label{bpairing}
	\end{flalign}  for any $m_1$ and $m_2$ in $M$. Let $N$ be a free $(\CO_F\otimes_{\CO_{F_0}}R)$-submodule of $M$ such that $b$ restricts to a perfect pairing on $N$. Denote by $N^\perp\coloneqq \tcbra{m\in M\ |\  b(m,n)=0 \text{\ for any $n\in N$}}$ the (left) orthogonal complement of $N$ with respect to $b$.  
	
	Then $N^\perp$ is a projective $(\CO_F\otimes_{\CO_{F_0}}R)$-module and $M=N\oplus N^\perp$ as $\CO_F\otimes_{\CO_{F_0}}R$-modules.
\end{lemma}
\begin{proof}
	 By construction, we have an exact sequence of $R$-modules \begin{flalign}
	 	0\ra N^\perp\xrightarrow{\alpha} M\xrightarrow{\beta} \Hom_R(N,R), \label{aexact}
	 \end{flalign} where $\alpha$ denotes the inclusion map and $\beta$ denotes the map $m\mapsto (n\mapsto b(m,n))$ for $m\in M$ and $n\in N$.
	By \eqref{bpairing}, the $R$-submodule $N^\perp$ is also an $\CO_F\otimes_{\CO_{F_0}}R$-submodule. For any $\varphi\in \Hom_R(N,R)$, define $\pi \varphi\in \Hom_R(N,R)$ by setting $(\pi\varphi)(n)\coloneqq \varphi(\ol{\pi}n)$ for $n\in N$. This endows $\Hom_R(N,R)$ with the structure of an $\CO_F\otimes_{\CO_{F_0}}R$-module, and the exact sequence \eqref{aexact} becomes an exact sequence of $\CO_F\otimes_{\CO_{F_0}}R$-modules. Since $b$ is perfect on $N$, the map $\beta$ is surjective with a section $\Hom_R(N,R)\ra N\sset M$. It follows that $M=N\oplus N^\perp$ as $\CO_F\otimes_{\CO_{F_0}}R$-modules and $N^\perp$ is projective. 
\end{proof}

\begin{lemma}[{cf. \cite[Lemma 9.2]{anschutz2018extending}}] \label{lem-affine}
	Let $R$ be an $\CO_{F_0}$-algebra and let $M$ be a free $\CO_F\otimes_{\CO_{F_0}}R$-module of rank $d$. Then the functor \begin{flalign*}
		    HQF(M): (\Sch/R)^\op &\lra \Sets\\ S &\mapsto \cbra{\text{$\CO_S$-valued hermitian quadratic forms on $M\otimes_R\CO_S$} }
	\end{flalign*} is represented by the affine space $\BA_R^{d^2}$ of dimension $d^2$ over $R$. 
\end{lemma}
\begin{proof}
	Choose a basis $e_1,\ldots,e_d$ of $M$ over $\CO_F\otimes_{\CO_{F_0}}R$. This is also a basis of $M\otimes_R\CO_S$. By the properties of hermitian quadratic forms, we can see that any hermitian quadratic form $q: M\otimes_R\CO_S\ra \CO_S$ is determined by values $q(e_i)$ for $1\leq i\leq d$ and $f(e_i,e_j)$, $f(e_i,\pi e_j)$ for $1\leq i<j\leq d$. More precisely, for any element $m=\sum_{i=1}^d(a_ie_i+b_i\pi e_i)\in M\otimes_R\CO_S$ for $a_i,b_i\in \CO_S$, we have \begin{flalign}
		    q(m) &= q(\sum_{i=1}^d a_ie_i)+f(\sum_{i=1}^da_ie_i,\sum_{i=1}^db_i\pi e_i)+q(\sum_{i=1}^db_i\pi e_i) \notag  \\ &=\sum_{i=1}^da_i^2q(e_i)+\sum_{1\leq i<j\leq d}a_ia_jf(e_i,e_j) + \sum_{1\leq i,j\leq d}a_ib_jf(e_i,\pi e_j) \notag  \\ &\quad +\sum_{i=1}^d\pi_0b_i^2q(e_i)+\sum_{i\leq i<j\leq d}\pi_0b_ib_jf(e_i,e_j).   \label{qeq}
	\end{flalign}
	Note also that for $1\leq i,j\leq d$, we have \begin{flalign*}
		   f(e_i,\pi e_j) &=f(\pi e_j,e_i) =f(e_j,\ol{\pi}e_i) =f(e_j,(t-\pi)e_i)=tf(e_j,e_i)-f(e_j,\pi e_i).
	\end{flalign*}
	
	Conversely, given $d^2$ elements in $\CO_S$ denoted as $A_{ii}$ for $1\leq i\leq d$ and $A_{ij}$, $B_{ij}$ for $1\leq i<j\leq d$, we can define a hermitian quadratic form on $M\otimes_R\CO_S$ as follows. We first define two $d\times d$ matrices $A$ and $B$ via setting $B_{ii}\coloneqq tA_{ii}$ for $1\leq i\leq d$, $A_{ij}\coloneqq A_{ji}$ and $B_{ij}\coloneqq tA_{ij}-B_{ji}$ for $i>j$. Then we define a map $q$ as in \eqref{qeq}. We can check that $q$ is an $\CO_S$-valued hermitian quadratic form. 
\end{proof}
The proof of Lemma \ref{lem-affine} also implies that the scheme $HQF(M)$ is (non-canonically) isomorphic to $\Spec R[A,B]/I$, where $A, B$ are two $d\times d$ matrices, and $I$ is the ideal generated by \begin{flalign*}
	    A_{ij}-A_{ji}, B_{k\ell}+B_{\ell k}-tA_{k\ell}, B_{ii}-tA_{ii}
\end{flalign*} for $1\leq i,j\leq d$ and $1\leq k<\ell\leq d$.

\begin{defn}
	Let $(M,q,\sL)$ be an $\sL$-valued hermitian quadratic module of rank $d$ over some $\CO_{F_0}$-algebra $R$. Then as an $R$-module, the rank of $M$ is $2d$. We define the \dfn{discriminant} as the morphism \begin{flalign*}
		    \disc(q): \wedge^{2d}_RM\ra \wedge_R^{2d}(M^\vee\otimes_R\sL)\simeq\wedge^{2d}_R(M^\vee)\otimes_R\sL^{2d}
	\end{flalign*} induced by the morphism $M\ra M^\vee\otimes_R\sL$, $m\mapsto f(m,-)$. Here $M^\vee$ denotes the $R$-dual module $\Hom_R(M,R)$. 
\end{defn}
\begin{example}\label{ex-rankonem}
	Assume $d=1$. Let $x\in M$ be a generator of $M$ over $\CO_F\otimes_{\CO_{F_0}}R$. Then with respect to the basis $\cbra{x,\pi x}$, the symmetric pairing $f: M\times M\ra \sL$ associated with $q$ is given by the matrix \begin{flalign*}
		    \begin{pmatrix}
		    	2q(x) &tq(x)\\ tq(x) &2\pi_0q(x)
		    \end{pmatrix}.
	\end{flalign*}
	Using the above basis, the discriminant map can be identified with the determinant of the previous matrix, as an element in $\sL^{2}$. Therefore, \begin{flalign*}
		    \disc(q) =(4\pi_0-t^2)q(x)^2.
	\end{flalign*}
\end{example}
We find that when $d=1$, the discriminant is ``divisible" by $4\pi_0-t^2$. More generally, we have the following lemma.

\begin{lemma}[{cf. \cite[Lemma 9.4]{anschutz2018extending}}] \label{lem-divided}
	Assume $d\geq 1$ is odd.  Then there exists a functorial factorization \begin{flalign*}
		  \xymatrix{
		     \wedge^{2d}_RM\ar[r]^{\disc(q)}\ar[d]_{\disc'(q)} &\wedge^{2d}_RM^\vee\otimes_R\sL^{2d}\\ \wedge^{2d}_RM^\vee\otimes_R\sL^{2d}\otimes_{\CO_{F_0}}(4\pi_0-t^2)\ar[ru]_{j}
		  }
	\end{flalign*}
	Here the map $j$ is induced by the natural inclusion of the ideal $(4\pi_0-t^2)$ in $\CO_{F_0}$. 
\end{lemma}
\begin{proof}
	It suffices to prove this in the universal case, i.e., $R$ is the ring $$R=\CO_{F_0}[A,B]/I,$$ where $I$ is the ideal generated by \begin{flalign*}
	            A_{ij}-A_{ji}, B_{k\ell}+B_{\ell k}-tA_{k\ell}, B_{ii}-tA_{ii}
           \end{flalign*} 
    for $1\leq i,j\leq d$ and $1\leq k<\ell\leq d$, and $M$ is equipped with the universal quadratic form $q:M\ra R$ given by \begin{flalign*}
    	    q(\sum_{i=1}^d(a_ie_i+b_i\pi e_i)) \coloneqq  \sum_{1\leq i,j\leq d}A_{ij}a_ia_j+ \sum_{1\leq i,j\leq d}B_{ij}a_ib_j +\pi_0\sum_{1\leq i,j\leq d}A_{ij}b_ib_j,
    	  \end{flalign*}
    for some $R$-basis $(e_i,\pi e_i)_{1\leq i\leq d}$ of $M$. Under the chosen basis, the associated symmetric bilinear form $f$ is given by the matrix \begin{flalign}
    	   C\coloneqq \begin{pmatrix}
    	    \wt{A} &B\\ B^t &\pi_0\wt{A}
    \end{pmatrix}\in M_{2d,2d}(R),  \label{symmatrix}
    \end{flalign}  where $\wt{A}_{ii}\coloneqq 2A_{ii}$ for $1\leq i\leq d$, $\wt{A}_{ij}\coloneqq A_{ij}$ for $i\neq j$, and the transpose matrix $B^t$ of $B$ equals $t\wt{A}-B$. We may identify $\disc(q)$ with the determinant of the above matrix $C$. To finish the proof, we need to show that the ideal $(\disc(q))$ is contained in the ideal $(4\pi_0-t^2)$ in $R$. As $(4\pi_0-t^2)$ becomes the unit ideal in $R[{1}/{\pi_0}]$, it suffices to show that the ideal $(\disc(q))$ is contained in $(4\pi_0-t^2)$ in the localization $R_\fm$, where $\fm$ is the ideal $(\pi_0)$. Equivalently, we need to show that $\disc(q)$ is divisible by $4\pi_0-t^2$ in $R_\fm/\fm^k$ for all $k\geq 1$.
    
    We will argue by induction on the rank $d$. If $d=1$, this follows by the computation in Example \ref{ex-rankonem}. Note that in the ring $R_\fm/\fm^k$, the element $B_{ij}=f(e_i,\pi e_j)$ is a unit for $i\neq j$ and $\pi_0$ is nilpotent. In particular, we may assume $f(e_1,\pi e_2)=1$. Then by Lemma \ref{lem-trick}, we may assume $f$ restricting to the submodule $R\pair{e_1,e_2,\pi e_1,\pi e_2}$ is given by the matrix \begin{flalign*}
    	    \begin{pmatrix}
    	    	0 &0 &0 &1\\ 0 &0 &-1 &0\\ 0 &-1 &0 &0\\ 1 &0 &0 &0
    	    \end{pmatrix}.
    \end{flalign*}
    The determinant of the above matrix is one. In particular, $f$ is perfect on $R\pair{e_1,e_2,\pi e_1,\pi e_2}$. Then we can write $M=R\pair{e_1,e_2,\pi e_1,\pi e_2}\oplus M'$, where $M'$ is the orthogonal complement of $R\pair{e_1,e_2,\pi e_1,\pi e_2}$ in $M$ with respect to $f$. The rank of $M'$ over $\CO_F\otimes_{\CO_{F_0}}R$ is $d-2$, which is odd. By induction, $\disc(q|_{M'})$ is divisible by $4\pi_0-t^2$. Hence, $\disc(q)=\disc(q|_{M'})$ is also divisible by $4\pi_0-t^2$. 
\end{proof}

\begin{defn}\label{app-defn1}
	We call the morphism $\disc'(q)$ in Lemma \ref{lem-divided} the \dfn{divided discriminant} of $q$. If $\disc'(q)$ is an isomorphism, then we say $(M,q,\sL)$ is a hermitian quadratic module of type $\Lambda_m$.
\end{defn}

\begin{example}
	[{cf. \cite[Definition 9.7]{anschutz2018extending}}] 
	Let $R$ be an $\CO_{F_0}$-algebra. Define $$M_{std,2}\coloneqq (\CO_F\otimes_{\CO_{F_0}}R)\pair{e_1,e_2}$$ with hermitian quadratic form $q_{std,2}:M_{std,2}\ra R$ determined by $$q_{std,2}(e_1)=q_{std,2}(e_2)=0, f_{std,2}(e_1,e_2)=0, f_{std,2}(e_1,\pi e_2)=1.$$ 
	For an odd integer $n=2m+1$, we define $$M_{std,n}\coloneqq M_{std,2}^{\oplus m}\oplus (\CO_F\otimes_{\CO_{F_0}}R) e_n$$ as an orthogonal direct sum and $q_{std,n}(e_n)\coloneqq 1$. Viewing $\disc'(q_{std,n})$ as an element in $R$, then we have $$\disc'(q_{std,n})=1.$$ Hence, $(M_{std,n},q_{std,n},R)$ is a hermitian quadratic module over $R$ of type $\Lambda_m$. 
\end{example}

\begin{example}\label{app-exlambm}
   By direct computation of the determinants of matrices \eqref{S1mRU} and \eqref{S1mRP}, the hermitian quadratic module $(\Lambda_m, q, \varepsilon\inverse\CO_{F_0})$ is of type $\Lambda_m$. 
\end{example}

\begin{lemma}\label{Gtorsorsurjective}
	Let $S$ be a scheme. Let $\sG$ be a smooth group scheme over $S$. Let $X$ be a scheme over $S$ equipped with a $\sG$-action $\rho: \sG\times_SX\ra X$. Assume $\rho$ is simply transitive in the sense that for any $S$-scheme $T$, the set $X(T)$ is either empty or the action of $\sG(T)$ on $X(T)$ is simply transitive. If the structure morphism $X\ra S$ is surjective, then $X$ is an \etale $\sG$-torsor over $S$.
\end{lemma}
\begin{proof}
	As $\rho$ is simply transitive, we have an isomorphism $\Phi: \sG\times_SX\simto X\times_SX$, $(g,x)\mapsto (\rho((g,x)),x)$ by \cite[0499]{stacks-project}. As $\sG\ra S$ is a smooth cover of $S$ and smoothness is an fpqc local property on the target, the isomorphism $\Phi$ implies that $X\ra S$ is smooth. If $X\ra S$ is surjective, then $X\ra S$ is a smooth cover of $S$. Let $s: X\ra \sG\times_SX$ be the morphism induced by the identity section of $\sG$. Then the composite $\Phi\circ s$ gives a section of $X\times_SX\ra X$. By \cite[055V]{stacks-project}, we can find an \etale cover $\tcbra{U_i}_{i\in I}$ of $S$ such that $X\times_SU_i\ra U_i$ has a section for each $i\in I$. Hence, we deduce that $X$ is an \etale $\sG$-torsor over $S$.
\end{proof}

\begin{thm}[{cf. \cite[Theorem 9.10]{anschutz2018extending}}] \label{thmstand}
	Let $(M,q,\sL)$ be a hermitian quadratic module of type $\Lambda_m$ of rank $n=2m+1$ over $R$. Then $(M,q,\sL)$ is \etale locally isomorphic to $(M_{std,n},q_{std,n},R)$ up to similitude. In particular, $(M,q,\sL)$ is \etale locally isomorphic to $(\Lambda_m,q,\varepsilon\inverse\CO_{F_0})\otimes_{\CO_{F_0}}R$ up to similitude. 
\end{thm}
\begin{proof}
	Denote $\sG_m\coloneqq \ud{\Sim}(M_{std,n})$. It suffices to show that the sheaf $$\CF\coloneqq \ud{\Sim}((M_{std,n},q_{std,n},R), (M,q,\sL))$$ of similitudes is an \etale $\sG_m$-torsor over $R$. 
	
	Clearly, $\CF$ is represented by an affine scheme of finite type over $R$. We next prove that $\CF$ is smooth over $R$. Over $R[1/\pi_0]$, the quadratic form is determined by the associated symmetric pairing, and both $M_{std}$ and $M$ are self-dual with respect to the symmetric pairing. Then by the arguments in \cite[Appendix to Chapter 3]{rapoport1996period}, we see that $\CF$ is smooth and surjective over $R[1/\pi_0]$. Hence, to show the smoothness of $\CF$ over $R$, it suffices to prove that the morphism $\CF\ra \Spec\CO_F$ is (formally) smooth at points over $\Spec R/\pi_0R$. For any surjection $S\ra \ol{S}$ in $\Nilp_R$ with nilpotent kernel $J$ and a similitude $(\ol{\varphi},\ol{\gamma})\in \CF(\ol{S})$, we need to show that there exists a lift of $(\ol{\varphi},\ol{\gamma})$ to $S$. We argue by induction on the rank $n$. We denote by $e_1,\ldots,e_n$ the standard basis of $M_{std,n}$. We reorder the basis such that $q(e_{m+1})=1$ and $(\CO_F\otimes_{\CO_{F_0}}R)\pair{e_i,e_{n+1-i}}\simeq M_{std,2}$. We claim that there exist elements $v_1,\ldots,v_n$ in $M\otimes_RS$ and a generator $u\in \sL\otimes_RS$ such that $\ol{v_i}=\ol{\varphi}(\ol{e_i})$ in $M\otimes_R\ol{S}$ and \begin{flalign*}
		    \text{$q(v_{m+1})=u$, $q(v_i)=f(v_i,v_j)=0$ and $f(v_i,\pi v_j)=u\delta_{i,n+1-j}$ for $1\leq i<j\leq n$ and $i,j\neq m+1$}.
	\end{flalign*}
	Then the maps $\varphi: e_i\mapsto v_i$ and $\gamma: 1\mapsto u$ define a lift of $(\ol{\varphi},\ol{\gamma})$. Thus, it suffices to prove the claim.
		
	Suppose $n=1$. Set $\ol{v}_1\coloneqq \ol{\varphi}(\ol{e_1})\in M\otimes_R\ol{S}$. Then $\ol{v}_1$ is a generator of $M\otimes_R\ol{S}$. Pick any lift $v_1\in M$ of $\ol{v}_1$. As $\disc'(q)$ is an isomorphism, $q(v_1)$ is a generator of $\sL$. Let $u=q(v_1)$. This proves the claim for $n=1$.	    
	For $n\geq 3$, pick lifts $v_1,\ldots,v_n$ in $M\otimes_RS$ such that $\ol{v_i}=\ol{\varphi}(\ol{e_i})$. Let $f$ be the associated symmetric pairing of $M$. Then $f(v_1,\pi v_n)$ is a generator in $\sL\otimes_RS$, as its reduction in $\sL\otimes_R\ol{S}$ is a generator. Set $u=f(v_1,\pi v_n)$. Using the generator $u$, we may identify $\sL\otimes_RS$ with $S$, and we may assume that $f(v_1,\pi v_2)=1$ in $\sL\otimes_RS \simeq S$. Note that as elements $q(v_1), q(v_2)$ and $f(v_1,v_2)$ reduce to zero in $\ol{S}$ by properties of $\ol{v_1}$ and $\ol{v_2}$, they lie in the kernel $J$. Then the linear transformation in Lemma \ref{lem-trick} does not change the reduction of $v_1$ and $v_2$, and hence, we may assume that  \begin{flalign*}
	    	     q(v_1)=q(v_n)=f(v_1,v_n)=0 \text{\ and\ }f(v_1,\pi v_n)=1.
	    \end{flalign*}
	    Then $f$ is perfect on the $S$-submodule $N$ generated by $v_1,v_n,\pi v_1,\pi v_n$. Let $N^\perp$ be the orthogonal complement of $N$ in $M\otimes_RS$. Then $N^\perp\otimes_R\ol{S}$ is the $\CO_F\otimes_{\CO_{F_0}}\ol{S}$-submodule in $M\otimes_R\ol{S}$ generated by $\ol{v_2},\ldots,\ol{v_{n-1}}$. For $2\leq i\leq n-1$, we can write $v_i=w'+w$, where $w'\in N^\perp$ and $w\in N$. As $\ol{v_i}$ is orthogonal to $\ol{N}$, we have $\ol{w}$ is orthogonal to $\ol{N}$. Since $f$ is perfect on $N$, we obtain $\ol{w}=0$. In particular, we may choose $v_i$ in $N^\perp$ as a lift of $\ol{v_i}$ for $2\leq i\leq n-1$. Now the claim follows by induction on the rank of $M$, and we deduce the (formal) smoothness of $\CF$ over $R$. 
	    
	    Note that the same proof implies that the group scheme $\sG_m$ is smooth over $R$. As the $\sG_m$-action on $\CF$ is simply transitive by construction, by Lemma \ref{Gtorsorsurjective}, it remains to show that $\CF$ is a surjective scheme over $R$. Since we have already shown that $\CF$ is surjective over $R[1/\pi_0]$, it suffices to prove the surjectivity of $\CF$ over $R/\pi_0R$. Then we may assume $R=\ol{k}$ is the algebraic closure of the residue field $k$ of $\CO_{F_0}$ and $\sL=\ol{k}$. We need to show that there exists a similitude isomorphism $(\varphi,\gamma)$ between $(M_{std,n},q_{std,n},\ol{k})$ and $(M,q,\ol{k})$. For the case $n=1$, we can construct a similitude as in the previous paragraph. For $n\geq 3$ odd, we first claim that there exist $v$ and $w$ in $M$ such that $f(v,\pi w)=1$. Otherwise, under a basis of the form $(v_1,\ldots,v_n,\pi v_1,\ldots,\pi v_n)$, the pairing $f$ corresponds to the $2n\times 2n$ matrix \begin{flalign*}
	    	    \begin{pmatrix}
	    	    	\wt{A} &0\\ 0 &0
	    	    \end{pmatrix}
	    \end{flalign*} for some $n\times n$ matrix $\wt{A}$, where $\wt{A}_{ii}=2q(v_i)=0$ for $1\leq i\leq n$ and $\wt{A}_{ij}=f(v_i,v_j)$ for $i\neq j$. Suppose for some indices $i_0\neq j_0$, we have $f(v_{i_0},v_{j_0})\neq 0$. We may assume $f(v_1,v_2)\neq 0$. Then by a suitable linear transformation of the basis $v_1,\ldots,v_n$, we may assume that $\wt{A}$ is of the form \begin{flalign*}
	    	     \begin{pmatrix}
	    	     	\begin{matrix}
	    	     		0 &1\\ 1 &0
	    	     	\end{matrix} &\rvline &\bigzero \\ \hline \bigzero &\rvline &\wt{A}_1
	    	     \end{pmatrix}
	    \end{flalign*}
	    In particular, $M_1\coloneqq (\CO_{F}\otimes_{\CO_{F_0}}\ol{k}) \pair{v_1,v_2}$ and $M_2\coloneqq (\CO_F\otimes_{\CO_{F_0}}\ol{k})\pair{v_3,\ldots,v_n}$ are orthogonal complement of each other. Then \begin{flalign*}
	    	     \disc'(q)= \disc(q|_{M_1})\disc'(q|_{M_2}).
	    \end{flalign*}
	    However,  \begin{flalign*}
	    	   \disc(q|_{M_1})=\det \begin{pmatrix}
	    	   	    0 &1 &0 &0\\ 1 &0 &0 &0\\ 0 &0 &0 &0\\ 0 &0 &0 &0
	    	   \end{pmatrix} =0.
	    \end{flalign*}
	    This contradicts the assumption that $\disc'(q)$ is a unit. Then we see $f(v_i,v_j)=0$ for any $i\neq j$, i.e., $\wt{A}$ is a diagonal matrix. Hence, $M$ is an orthogonal direct sum of rank one hermitian quadratic modules. This also contradicts $\disc'(q)\neq 0$. Then we conclude that there exist $v$ and $w$ in $M$ such that $f(v,\pi w)=1$. Then as in Lemma \ref{lem-trick}, we may assume that $f$ restricting to $(\CO_F\otimes_{\CO_{F_0}}\ol{k})\pair{v, w}$ corresponds to the matrix \begin{flalign*}
    	    \begin{pmatrix}
    	    	0 &0 &0 &1\\ 0 &0 &-1 &0\\ 0 &-1 &0 &0\\ 1 &0 &0 &0
    	    \end{pmatrix}.
    \end{flalign*}
     Hence, $(\CO_F\otimes_{\CO_{F_0}}\ol{k})\pair{v, w}$ is isomorphic to $M_{std,2}$. Its orthogonal complement is a hermitian quadratic module of type $\Lambda_m$ of rank $n-2$. Now we can finish the proof by induction on the rank of $M$.
\end{proof}

\begin{thm}[{cf. \cite[Proposition 9.9]{anschutz2018extending}}] \label{thmsimm}
	The group functor $\ud{\Sim}(\Lambda_m)$ is representable by an affine smooth group scheme over $\CO_{F_0}$ whose generic fiber is $\GU(V,h)$. 
\end{thm}
\begin{proof}
	By the proof of Theorem \ref{thmstand}, the functor $\ud{\Sim}(\Lambda_m)$ is representable by an affine smooth group scheme of finite type over $\CO_{F_0}$. It remains to prove the assertion for the generic fiber. Following the notations in \S \ref{subsec-hermpara}, we denote by $s$ the symmetric pairing on $\Lambda_m$. For any $F_0$-algebra $R$, we have \begin{flalign*}
		    \ud{\Sim}(\Lambda_m)(R) &=\cbra{(\varphi,\gamma)\ \vline \ \Centerstack[l]{$\varphi$ is an automorphism of the $\CO_F\otimes_{\CO_{F_0}}R$-module $\Lambda_m\otimes_{\CO_{F_0}} R$ \\ $\gamma: \sL\otimes_{\CO_{F_0}}R \simto \sL\otimes_{\CO_{F_0}}R $ \\ $q(\varphi(x))=\gamma(q(x))$ for $x\in \Lambda_m\otimes_{\CO_{F_0}}R=V\otimes_{F_0}R$ } } \\ &=\cbra{\varphi\in \GL_{F\otimes_{F_0}R}(V\otimes_{F_0}R)\ \vline \ 
		    \Centerstack[l]{$\gamma: R\simto R$\\ $s(\varphi(x),\varphi(y))=\gamma\rbra{s(x,y) } $ for $x,y\in V\otimes_{F_0}R$ } } \\ &=\cbra{\varphi\in\GL_{F\otimes_{F_0}R}(V\otimes_{F_0}F)\ \vline \ \Centerstack[l]{$s(\varphi(x),\varphi(y))=c(\varphi)s(x,y)$\\ for $x,y\in V\otimes_{F_0}R$ and some $c(\varphi)\in R\cross$ }  } \\ &= \cbra{\varphi\in\GL_{F\otimes_{F_0}R}(V\otimes_{F_0}F)\ \vline \ \Centerstack[l]{$h(\varphi(x),\varphi(y))=c(\varphi)h(x,y)$\\ for $x,y\in V\otimes_{F_0}R$ and some $c(\varphi)\in R\cross$ }  } \\ &=\GU(V,h)(R).
	\end{flalign*}
	Therefore, the generic fiber of $\ud{\Sim}(\Lambda_m)$ is $\GU(V,h)$. 
\end{proof}

\begin{corollary}\label{app-coroparam}
	The scheme $\ud{\Sim}(\Lambda_m)$ is isomorphic to the parahoric group scheme attached to $\Lambda_m$.
\end{corollary}
\begin{proof}
	Let $\breve F_0$ denote the completion of the maximal unramified extension of $F_0$. By construction, we know that $\ud{\Sim}(\Lambda)(\CO_{\breve F_0})$ is the stabilizer of $\Lambda_m$ in $\GU(V,h)(\breve F_0)$, which is a parahoric subgroup by Proposition \ref{prop-para}. As $\ud{\Sim}(\Lambda)$ is smooth over $\CO_{F_0}$ by Theorem \ref{thmsimm}, the corollary follows by \cite[1.7.6]{bruhat1984groupes}.
\end{proof}

\subsection{Hermitian quadratic modules of type $\Lambda_0$}

Let $R$ be an $\CO_{F_0}$-algebra. Recall that in Definition \ref{defn-hermphi}, we have defined the category $\CC_R$ of hermitian quadratic modules with $\phi$.
By a similar proof as in Lemma \ref{lem-affine}, we can show that for a fixed free $\CO_F\otimes_{\CO_{F_0}}R$-module $M$ of rank $d$, the moduli functor of all bilinear forms $\phi$ and quadratic forms $q$ on $M$ satisfying \eqref{eqphiq} in Definition \ref{defn-hermphi} is representable by the affine space of dimension $d^2$ over $R$. 

Let $(M,q,\sL,\phi)\in\CC_R$. Choose a basis $(e_1,\ldots,e_d,\pi e_1,\ldots,\pi e_d)$ of $M$. The pairing $\phi$ is then given by the matrix  \begin{flalign*}
	    \begin{pmatrix}
	    	\wt{A} &\wt{B}\\ t\wt{A}-\wt{B} &\pi_0\wt{A}
	    \end{pmatrix},
\end{flalign*} where $\wt{A}_{ii}=(t/\pi_0) q(e_i)$ and $\wt{B}_{ii}=2q(e_i)$ for $1\leq i\leq d$, $\wt{A}_{ij}=\phi(e_i,e_j)$ and $\wt{B}_{ij}=\phi(e_i,\pi e_j)$ for $1\leq i,j\leq d$ and $i\neq j$, and they satisfy $\wt{A}=-\wt{A}^t+(t/\pi_0)\wt{B}$ and $\wt{B}=\wt{B}^t$. 

\begin{defn} \label{defn-discphi}
	Let $(M,q,\sL,\phi)\in\CC_R$ and the rank of $M$ over $R$ is $2d$. We define the \dfn{discriminant} as the morphism \begin{flalign*}
		    \disc(\phi): \wedge^{2d}_RM\ra \wedge_R^{2d}(M^\vee\otimes_R\sL)\simeq\wedge^{2d}_R(M^\vee)\otimes_R\sL^{2d}
	\end{flalign*} induced by the morphism $M\ra M^\vee\otimes_R\sL$, $m\mapsto \phi(m,-)$. 
\end{defn}

\begin{example}\label{ex-rankone}
	Assume $d=1$. Let $x\in M$ be a generator of $M$ over $\CO_F\otimes_{\CO_{F_0}}R$. Suppose $(M,q,\sL)$ is a hermitian quadratic module. Then we can define a bilinear form $\phi: M\times M\ra \sL$ given by the matrix \begin{flalign*}
		    \begin{pmatrix}
		    	t/\pi_0q(x) &2q(x)\\ (t^2-2\pi_0)/\pi_0 q(x) &tq(x)
		    \end{pmatrix}
	\end{flalign*}
with respect to the basis $\cbra{x,\pi x}$. Equipped with such $\phi$, we have $(M,q,\sL,\phi)\in\CC_R$. Using the basis $\cbra{x,\pi x}$, we may view the discriminant map $\disc(\phi)$ as the determinant of the above matrix. We have $$\disc(\phi)= \frac{4\pi_0-t^2}{\pi_0} q(x)^2.$$
\end{example}

Arguing similarly as in Lemma \ref{lem-divided}, we can show the following result.
\begin{lemma}
	Assume $d\geq 1$ is odd. Then there exists a functorial factorization \begin{flalign*}
		  \xymatrix{
		     \wedge^{2d}_RM\ar[r]^{\disc(\phi)}\ar[d]_{\disc'(\phi)} &\wedge^{2d}_RM^\vee\otimes_R\sL^{2d}\\ \wedge^{2d}_RM^\vee\otimes_R\sL^{2d}\otimes_{\CO_{F_0}}(\frac{4\pi_0-t^2}{\pi_0})\ar[ru]_{j}
		  }
	\end{flalign*}
	Here the map $j$ is induced by the natural inclusion of the ideal $(\frac{4\pi_0-t^2}{\pi_0})$ in $\CO_{F_0}$. 
\end{lemma}
\begin{proof}
	As in the proof of Lemma \ref{lem-divided}, we can reduce to show that the determinant, which equals $\disc(\phi)$, of a matrix of the form \begin{flalign*}
	    \begin{pmatrix}
	    	\wt{A} &\wt{B}\\ t\wt{A}-\wt{B} &\pi_0\wt{A}
	    \end{pmatrix}\in M_{2d,2d}(R),
\end{flalign*} is divisible by $(4\pi_0-t^2)/\pi_0$ in $R$, where $\wt{A}_{ii}=(t/\pi_0) q(e_i)$ and $\wt{B}_{ii}=2q(e_i)$ for $1\leq i\leq d$, $\wt{A}_{ij}=\phi(e_i,e_j)$ and $\wt{B}_{ij}=\phi(e_i,\pi e_j)$ for $1\leq i,j\leq d$ and $i\neq j$, and they satisfy $\wt{A}=-\wt{A}^t+(t/\pi_0)\wt{B}$ and $\wt{B}=\wt{B}^t$. 
   
   If $d=1$, then the lemma follows by Example \ref{ex-rankone}. Suppose $d\geq 3$. We may assume $\pi_0$ is nilpotent in $R$ and $B_{12}=\phi(e_1,\pi e_2)=1$ as in the proof of Lemma \ref{lem-divided}. As in Lemma \ref{lem-trick}, replacing $e_1$ by $r_1e_1+r_2\ol{\pi}e_1$ for suitable $r_1$ and $r_2$ in $R$, we may assume further that $\phi(e_1,e_2)=0$. Then restricting to the submodule $\pair{e_1,e_2,\pi e_1,\pi e_2}$, the pairing $\phi$ is given by the matrix \begin{flalign*}
   	       \begin{pmatrix}
   	       	    \frac{t}{\pi_0}q(e_1) &0 &2q(e_1) &1\\ \frac{t}{\pi_0} &\frac{t}{\pi_0}q(e_2) &1 &2q(e_2)\\ \frac{t^2-2\pi_0}{\pi_0}q(e_1) &-1 &tq(e_1) &0 \\ \frac{t^2-\pi_0}{\pi_0} &\frac{t^2-2\pi_0}{\pi_0}q(e_2) &t &tq(e_2)
   	       \end{pmatrix}.
   \end{flalign*}
      By direct computation, the above is an invertible matrix, and hence the pairing $\phi$ is perfect on the module $\pair{e_1,e_2,\pi e_2,\pi e_2}$. Therefore, the orthogonal complement $M'$ of $\pair{e_1,e_2,\pi e_2,\pi e_2}$ in $M$ has rank $n-2$ over $\CO_F\otimes_{\CO_{F_0}}R$, and $M'\in\CC_R$. Then we finish the proof by induction on the rank of $M$.
\end{proof}

\begin{defn}\label{app-defn2}
	Let $R$ be an $\CO_{F_0}$-algebra. We say a hermitian quadratic module $(M,q,\sL,\phi)\in\CC_R$ over $R$ is of type $\Lambda_0$ if $\disc'(\phi)$ is an isomorphism. 
\end{defn}

\begin{example} \label{example-std0}
   Let $R$ be an $\CO_{F_0}$-algebra.
   \begin{enumerate}
   	\item  Suppose $(M,q,R)$ is a hermitian quadratic module of rank one. Let $x\in M$ be a generator and assume $q(x)=1$. We can define a bilinear form $\phi_{std,1}: M\times M\ra R$ as in Example \ref{ex-rankone}. Then $(M,q,\sL,\phi_{std,1})\in\CC_R$. Viewing $\disc'(\phi_{std,1})$ as an element in $R$, we have  $\disc'(\phi_{std,1})=1$. 
    \item  Define $$N_{std,2}\coloneqq (\CO_F\otimes_{\CO_{F_0}}R)\pair{e_1,e_2}$$ with hermitian quadratic form $q_{std,2}: N_{std,2}\ra R$ determined by  $$q_{std,2}(e_1)=q_{std,2}(e_2)=0, \phi_{std,2}(e_1,e_2)=0, \phi_{std,2}(e_1,\pi e_2)=1.$$ 
	For an odd integer $n=2m+1$, we define $$N_{std,n}\coloneqq N_{std,2}^{\oplus m}\oplus (\CO_F\otimes_{\CO_{F_0}}R) e_n.$$ Here $(\CO_F\otimes_{\CO_{F_0}}R) e_n$ is a hermitian quadratic module of rank one as in (1), and the direct sum is an orthogonal direct sum with respect to $\phi_{std,n}\coloneqq \phi_{std,2}^{\oplus m}\oplus\phi_{std,1}$. Viewing $\disc'(\phi_{std,n})$ as an element in $R$, we have $$\disc'(\phi_{std,n})=1.$$ Hence, $(N_{std,n},q_{std,n},R,\phi_{std,n})$ is a hermitian quadratic module over $R$ of type $\Lambda_0$. 
   \end{enumerate}
\end{example}

\begin{example}\label{exlambda0}
	Equipped with the following bilinear form \begin{flalign*}
		 \phi(-,-): \Lambda_0\times \Lambda_0 \lra \sL=\varepsilon\inverse\CO_{F_0},\quad   (x,y) \mapsto s(x,\pi\inverse y)=\varepsilon\inverse \Tr_{F/F_0} h(x,\pi\inverse y),
	\end{flalign*}  the hermitian quadratic module $(\Lambda_0, q, \varepsilon\inverse\CO_{F_0},\phi)$ is of type $\Lambda_0$. 
\end{example}

\begin{thm} \label{thmstand0}
	Let $(M,q,\sL,\phi)$ be a hermitian quadratic module of type $\Lambda_0$ of rank $n=2m+1$ over $R$. Then $(M,q,\sL,\phi)$ is \etale locally isomorphic to $(N_{std,n},q_{std,n},R,\phi_{std,n})$ up to similitude. In particular, $(M,q,\sL,\phi)$ is \etale locally isomorphic to $(\Lambda_0,q,\varepsilon\inverse\CO_{F_0},\phi)\otimes_{\CO_{F_0}}R$ up to similitude. 
\end{thm}
\begin{proof}
	As in the proof of Theorem \ref{thmstand}, it suffices to show that the representable sheaf $$\CF\coloneqq \ud{\Sim}((N_{std,n},q_{std,n},R,\phi_{std,n}), (M,q,\sL,\phi))$$ of similitudes is surjective over $R$ and smooth at points over $\Spec R/\pi_0R$.
	
	We first check that for any surjection $S\ra \ol{S}$ in $\Nilp_R$ with nilpotent kernel $J$ and a similitude $(\ol{\varphi},\ol{\gamma})\in \CF(\ol{S})$, there exists a lift of $(\ol{\varphi},\ol{\gamma})$ to $S$. We denote by $e_1,\ldots,e_n$ the standard basis of $N_{std,n}$. We reorder the basis such that $q(e_{m+1})=1$ and $(\CO_F\otimes_{\CO_{F_0}}R)\pair{e_i,e_{n+1-i}}\simeq N_{std,2}$. We claim that there exist lifts $v_i\in M\otimes_RS $ of $\ol{v_i}\coloneqq \ol{\varphi}(\ol{e_i})$ for $1\leq i\leq n$ and a generator $u\in \sL\otimes_RS$ such that \begin{flalign*}
		    \text{$q(v_{m+1})=u$, $q(v_i)=\phi(v_i,v_j)=0$ and $\phi(v_i,\pi v_j)=u\delta_{i,n+1-j}$ for $1\leq i<j\leq n$ and $i,j\neq m+1$}.
	\end{flalign*}
	The the maps $\varphi: e_i\mapsto v_i$ and $\gamma: 1\mapsto u$ defines a lift of $(\ol{\varphi},\ol{\gamma})$ and $(\varphi,\gamma)$ preserves $\phi$. Thus it suffices to prove the claim. 
	
	Suppose $n=1$. Pick any lift $v_1$ of $\ol{v_1}$. As $\disc'(\phi)$ is an isomorphism, $q(v_1)$ is a generator of $\sL\otimes_RS$. Set $u=q(v_1)$. This proves the claim for $n=1$.	    
	For $n\geq 3$, pick any lifts $v_1,\ldots,v_n$ in $M\otimes_RS$ of $\ol{v_1},\ldots,\ol{v_n}$. As in the proof of Theorem \ref{thmstand}, we may assume that $\sL\otimes_RS\simeq S$ and $\phi(v_1,\pi v_n)=1$ in $S$. Let $r_0\in R$ be a solution of the quadratic equation $q(v_n)r^2+r+q(v_1)=0$, which exists by arguments in Lemma \ref{lem-trick}. Since $q(v_1)$ and $q(v_n)$ lie in $J$, we have $r_0\in J$. Then $v_1'\coloneqq v_1+r_0v_n$ and $\ol{v_1'}=\ol{v_1}$. So we may find a lift $v_n'$ such that $\phi(v_1',v_n')=1$. Set $v_n''\coloneqq v_n'-q(v_n')v_1'$. Then $q(v_n'')=0$ and $\ol{v_n''}=\ol{v_n}$. Set $$r_1\coloneqq (1-\phi(v_1',v_n'')\phi(v_1',\pi^2v''_n))\inverse \text{\ and\ } r_2\coloneqq -r_1\phi(v_1',v_n'').$$ Since $(\ol{\varphi},\ol{\gamma})$ preserves $\phi$, we have $\phi(\ol{v_1'},\ol{v_n''})=\ol{\gamma}(\phi_{std,n}(e_1,e_n))=0$. Thus, $\phi(v_1',v_n'')$ and $r_2$ are in $J$. Set $v_1''\coloneqq r_1v_1'+r_2\ol{\pi}v_1'$. Then $\ol{v_1''}=\ol{v}$. As in Lemma \ref{lem-trick}, we have $\phi(v_1'',\pi v_n'')=1$ and $\phi(v_1'',v_n'')=0$.  
	    By replacing $v_1$ by $v_1''$ and $v_n$ by $v_n''$, we may assume that 
	     \begin{flalign*}
	    	     q(v_1)=q(v_n)=\phi(v_1,v_n)=0 \text{\ and\ }\phi(v_1,\pi v_n)=1.
	    \end{flalign*}	    
	Then $\phi$ is perfect on the $S$-submodule $N$ generated by $v_1,v_2,\pi v_1,\pi v_2$. Let $N^\perp$ be the orthogonal complement (with respect to $\phi$) of $N$ in $M\otimes_RS$. As in the proof of Theorem \ref{thmstand}, we may assume that lifts $v_i$ for $2\leq i\leq n-1$ lie in $N^\perp$. The claim follows by induction on the rank of $M$, and hence, we deduce the smoothness of $\CF$ over $R$. 
	    
	Next we prove the surjectivity of $\CF$ over $R$. It suffices to prove that $\CF$ has non-empty fibers over $R/\pi_0 R$. Then we may assume $R=\ol{k}$ is the algebraic closure of the residue field of $\CO_{F_0}$ and $\sL=\ol{k}$. We need to show that there exists a similitude isomorphism $(\varphi,\gamma)$ preserving $\phi$ between $(N_{std,n},q_{std,n},\ol{k},\phi_{std,n})$ and $(M,q,\ol{k},\phi)$. Suppose $n=1$. Then $M\otimes_RS=(\CO_F\otimes_{\CO_{F_0}}S)v$ for some $v$. Define 
	\begin{alignat*}{2}
		   \varphi: N_{std}\otimes_RS &\lra M\otimes_RS=(\CO_F\otimes_{\CO_{F_0}}S)v,\quad \gamma: &S\lra \sL\otimes_RS \\ e_1 &\mapsto v,  & 1\mapsto q(v).
	\end{alignat*}
	As $\disc'(\phi)$ is an isomorphism, $q(v)$ is a generator. 
	Since $\phi$ is determined by $q$ in this case by computation in Example \ref{ex-rankone}, the similitude $(\varphi,\gamma)$ preserves $\phi$. For $n\geq 3$ odd, we claim that there exist $v$ and $w$ in $M\otimes_RS$ such that $\phi(v,\pi w)=1$. This can be done using proof by contradiction as in Theorem \ref{thmstand}. Set $v'\coloneqq v+r_0w$, where $r_0\in \ol{k}$ is a solution for the quadratic equation $q(v')=q(w)r^2+r+q(v)$. Then $$\phi(v',\pi w)=\phi(v,\pi w)+r_0\phi(w,\pi w)=1+2r_0q(w)=1.$$ The last equality holds since $\Char\ \ol{k}=2$. Set $w'\coloneqq w-q(w)v'$. Then $q(w')=0$. As in the previous paragraph, we may find suitable $r_1$ and $r_2$ such that $v''\coloneqq r_1v'+r_2\ol{\pi}v'$ satisfies $\phi(v'',\pi w)=1$ and $\phi(v'',w')=0$. Replacing $v$ by $v''$ and $w$ by $w'$, we see that $\phi$ restricting to $(\CO_F\otimes_{\CO_{F_0}}\ol{k})\pair{v, w}$ acts the same as $\phi_{std,2}$. In particular, the subspace $(\CO_F\otimes_{\CO_{F_0}}\ol{k})\pair{v, w}$ is isomorphic to $N_{std,2}$. Its orthogonal complement is a hermitian quadratic module of type $\Lambda_0$ of rank $n-2$. Now we can finish the proof by induction on the rank of $M$.
\end{proof}

\begin{thm} \label{thmsimm0}
	The group functor $\ud{\Sim}((\Lambda_0,\phi))$ of similitudes preserving $\phi$ is representable by an affine smooth group scheme over $\CO_{F_0}$ whose generic fiber is $\GU(V,h)$. 
\end{thm}
\begin{proof}
    By the proof of Theorem \ref{thmstand0}, the functor $\ud{\Sim}((\Lambda_0,\phi))$ is representable by an affine smooth group scheme over $\CO_{F_0}$. It remains to show the assertion for the generic fiber.
	Let $R$ be an $F$-algebra. For any similitude $(\varphi,\gamma)\in \Sim(\Lambda_0)$ and $x,y\in \Lambda_0\otimes_{\CO_{F_0}}R=V\otimes_{F_0}R$, we have \begin{flalign*}
		    \phi(\varphi(x),\varphi(y)) &=\phi(\varphi(x),\pi (\pi\inverse \varphi(y)))= q(\varphi(x)+\varphi(\pi\inverse y)) -q(\varphi(x))- q(\varphi(\pi\inverse y))\\ &= \gamma(q(x+\pi\inverse y)-q(x)-q(\pi\inverse y) )=\gamma(\phi(x,y)).
	\end{flalign*} Hence, over the generic fiber, any similitude of $\Lambda_0$ preserves $\phi$. Then as in the proof of Theorem \ref{thmsimm}, we see that the generic fiber of $\ud{\Sim}((\Lambda_0,\phi))$ is $\GU(V,h)$. 
\end{proof}

The same argument as in the proof of Corollary \ref{app-coroparam} implies the following.
\begin{corollary}\label{app-coropara0}
	The scheme $\ud{\Sim}((\Lambda_0,\phi))$ is isomorphic to the parahoric group scheme attached to $\Lambda_0$.
\end{corollary}



\end{document}